%% file: Thesis.tex
\documentclass[twoside,a4paper,12pt]{memoir}

\usepackage{a4}
\usepackage{latexsym}
\usepackage{amsfonts}
\usepackage{amssymb}
\usepackage{amsmath}
\usepackage{amsthm}
\usepackage{dsfont}
\usepackage{thmtools}
\usepackage{float}
\usepackage{tikz}
\usepackage{wrapfig}
\usetikzlibrary{calc,intersections,through,backgrounds,positioning,decorations,shapes,arrows}
\usepgflibrary{decorations.pathmorphing,decorations.pathreplacing,decorations.markings,decorations.footprints,decorations.shapes}

\newsubfloat{figure}

\settocdepth{subsection}

\captiondelim{. }
\captionnamefont{\bfseries\small}
\captiontitlefont{\small}

\renewcommand{\contentsname}{Table of contents}
\renewcommand{\bibname}{References}

\theoremstyle{plain}
\newtheorem{theorem}{Theorem}[chapter] 
\newtheorem{corollary}[theorem]{Corollary}
\newtheorem{lemma}[theorem]{Lemma}

\newcommand{\ffield}{\mathbb{F}_q(t)}

\renewcommand{\printchaptername}{}
\renewcommand{\chapternamenum}{}
\renewcommand{\printchapternum}{\begin{center}{\HUGE  \textbf{\thechapter}}\end{center}}
\renewcommand{\afterchapternum}{}
\renewcommand{\printchaptertitle}[1]{\hrule \smallskip \begin{center} \HUGE{#1}\end{center}}
\renewcommand{\afterchaptertitle}{\bigskip\bigskip\bigskip\bigskip\bigskip}
\renewcommand{\printchapternonum}{\vphantom{\HUGE  \textbf{1}}\mbox{}\newline}

\makepagestyle{myruled}
\makeevenhead {myruled}{}{\small\itshape\leftmark} {}
\makeoddhead {myruled}{}{\small\itshape\rightmark}{}
\makeevenfoot {myruled}{}{\small \thepage} {}
\makeoddfoot {myruled}{}{\small \thepage} {}
\makeatletter % because of \@chapapp
\makepsmarks {myruled}{
\nouppercaseheads
\createmark {chapter} {both} {shownumber}{\@chapapp\ }{. \ }
\createmark {section} {right}{shownumber}{} {. \ }
\createplainmark {toc} {both} {\contentsname}
\createplainmark {lof} {both} {\listfigurename}
\createplainmark {lot} {both} {\listtablename}
\createplainmark {bib} {both} {\bibname}
\createplainmark {index} {both} {\indexname}
\createplainmark {glossary} {both} {\glossaryname}
}
\makeatother
\setsecnumdepth{subsubsection}
\renewcommand{\listfigurename}{List of figures}

\author{Timothy Jones}

\setSingleSpace{1.2}
\OnehalfSpacing 

\begin{document}

\newpage
%\setcounter{page}{1}
%\pagenumbering{roman}
%\pagestyle{plain}
\pagestyle{myruled}

\include{frontmatter}

\settocdepth{section}

\OnehalfSpacing 
%\DoubleSpacing

%\setcounter{page}{1}
%\pagenumbering{arabic}

\pagestyle{myruled}

\include{introduction}
\include{introincidences}
\include{introgrowth}

\include{introapproxgroup}

\include{incidences}
\include{expanders}
\include{functionfield}

%Appendix chapter heading style
\renewcommand{\printchaptername}{}
\renewcommand{\chapternamenum}{}
\renewcommand{\printchapternum}{\begin{center}{\HUGE  \textbf{Appendix \thechapter}}\end{center}}
\renewcommand{\afterchapternum}{}
\renewcommand{\printchaptertitle}[1]{\hrule \smallskip \begin{center} \HUGE{#1}\end{center}}
\renewcommand{\afterchaptertitle}{\bigskip\bigskip\bigskip\bigskip\bigskip}
\renewcommand{\printchapternonum}{\vphantom{\HUGE  \textbf{1}}\mbox{}\newline}

\include{appendices}

\SingleSpacing

\addtocontents{toc}{\protect\addvspace{10 pt}}
\addcontentsline{toc}{section}{\protect\hspace{-1.5 em} \emph{References}}
\nobibintoc
\bibliographystyle{plain}
\bibliography{thesisbib}

\end{document}

%% file: frontmatter.tex
\pagestyle{empty}

\bigskip\mbox{}\bigskip
\bigskip\mbox{}\bigskip

\begin{center}
\Huge
\textbf{New quantitative estimates on the incidence geometry and growth of finite sets}

\bigskip\mbox{}\bigskip\mbox{}\bigskip\mbox{}\bigskip\mbox{}\smallskip

{\Huge Timothy Gareth Fellgett Jones}
\end{center}

\bigskip\mbox{}\bigskip\mbox{}\bigskip
\bigskip\mbox{}\bigskip\mbox{}\bigskip
\bigskip\mbox{}\bigskip\mbox{}\bigskip

\noindent
\emph{\large
A dissertation submitted to the University of Bristol in accordance with the requirements for award of the degree of Doctor of Philosophy in the Faculty of Science}

\begin{center}
\bigskip\mbox{}\bigskip

{\Large School of Mathematics \\
\smallskip
January 2013 \\ }

\end{center}

\mbox{}\bigskip\mbox{}\bigskip

\begin{flushright}
$20,000$ words \end{flushright}

\cleardoublepage

\pagestyle{headings}

%\addcontentsline{toc}{section}{\protect\hspace{-1.5 em} \emph{Abstract}}
\markboth{Abstract}{Abstract}
\cleardoublepage
\chapter*{Abstract}
\begin{SingleSpace}
This thesis establishes new quantitative records in several problems of incidence geometry and growth. After the necessary background in Chapters 1, 2 and 3, the following results are proven.

Chapter 4 gives new results in the incidence geometry of a plane determined by a finite field of prime order. These comprise a new upper bound on the total number of incidences determined by finitely many points and lines, and a new estimate for the number of distinct lines determined by a finite set of non-collinear points.

Chapter 5 gives new results on expander functions. First, a new bound is established for the two-variable expander $a+ab$ over a finite field of prime order. Second, new expanders in three and four variables are demonstrated over the real and complex numbers with stronger growth properties than any functions previously considered. 

Finally, Chapter 6 gives the first bespoke sum-product estimate over function fields, a setting that has so far been largely unexplored for these kinds of problems. This last chapter is joint work with Thomas Bloom.
\end{SingleSpace}

%\addcontentsline{toc}{section}{\protect\hspace{-1.5 em} \emph{Acknowledgments}}
\chapter*{Acknowledgements}
\begin{SingleSpace}
This thesis would not exist, nor would the years leading up to its creation have been as fun, without the support of a great many people.

My supervisor is Misha Rudnev, and without his guidance, support, encouragement and patience nothing here would have been possible. Whilst none of the work here is specifically joint with Misha, everything herein has benefitted from his scrutiny and suggestions. My studentship is funded by the EPSRC, and I have been looked after very well by both the University of Bristol and the School of Mathematics.  

The postgraduate community of the School of Mathematics, and my long-running housemates Amanda and Nic, who do proper science, have been an invaluable source of support and distraction. To list everyone from maths would necessistate several pages, so I will limit myself to mentioning my combinatorial colleagues Tom Bloom and Olly Roche-Newton, with whom I have had many fruitful conversations. One of the results in this thesis is joint work with Tom.  

I am also grateful to Jo Bryant and everyone in At-Bristol, and to Alastair Iles and the leaders and cubs of the 7th Bristol cub pack, for reminders of the importance of the world outside sumset calculus and covering methods, and for giving me fun, challenging and most especially non-mathematical things to do.

Finally, Helen has made me more cups of tea than I care to recall, and put up with me at the lows (not being able to prove things, and so fed up) and highs (being able to prove something, and so unable to sit still) of research.
  
I have, overall, been very lucky indeed.

\end{SingleSpace}
\cleardoublepage

%\addcontentsline{toc}{section}{\protect\hspace{-1.5 em} \emph{Author's declaration}}
\chapter*{Author's declaration}
\begin{SingleSpace}I declare that the work in this dissertation was carried out in accordance with the requirements of the University's Regulations and Code of Practice for Research Degree Programmes and that it has not been submitted for any other academic award. Except where indicated by specific reference in the text, the work is the candidate's own work. Work done in collaboration with, or with the assistance of, others is indicated as such. Any views expressed in the dissertation are those of the author.
\end{SingleSpace}
\bigskip

\begin{quote}
Signed: 

\medskip

Date:
\end{quote}

\cleardoublepage

%\addcontentsline{toc}{section}{\protect\hspace{-1.5 em} \emph{Table of contents}}
\tableofcontents*
\cleardoublepage
\listoffigures*

%\clearpage
%\addcontentsline{toc}{section}{\protect\hspace{-1.5 em} \emph{Notation}}

%% file: introduction.tex
\chapter*{Standard notation}
\addcontentsline{toc}{section}{\protect\hspace{-1.5 em} \emph{Standard notation}}
\markboth{Standard notation}{Standard notation}

The following standard conventions are used.

\begin{itemize}
\item Write $Y=O(X)$, $Y \ll X$, or $X=\Omega(Y)$ if there is a fixed constant $C$ such that $Y \leq CX$. The constant $C$ is referred to as the `implicit constant'.
\item Using this notation, we will often use the fact that $X \ll Y + Z$ if and only if $X \ll \max \left\{Y,Z\right\}$.
\item If $X \ll Y$ and $Y \ll X$ then write $X \approx Y$ or $Y= \Theta (X)$.
\item	If the implicit constant depends on some parameter $\lambda$ then this is reflected with a subscript, e.g. $Y=O_{\lambda}(X)$ or $Y \ll_{\lambda}X$.
\item Write $Y \ll X^{1+o(1)}$, $Y \lesssim X$, $Y=\tilde{O}(X)$ or $X=\tilde{\Omega}(Y)$ if $Y \ll_{\epsilon}X^{1+\epsilon}$ for all $\epsilon>0$. In particular, this notation is used when $Y \ll X\left(\log X\right)^{\alpha}$ for some $\alpha>0$.
\item Write $\mathds{1}$ for an indicator function, that is if $\mathcal{P}$ is a proposition then $\mathds{1}(\mathcal{P})$ is equal to $1$ if $\mathcal{P}$ is true and is zero otherwise.
\item The letters $\mathbb{R},\mathbb{C},\mathbb{N},\mathbb{Z}$ have their usual meaning. By $\mathbb{F}_p$ we mean a finite field of prime order $p$, and by $\mathbb{F}_q$ a finite field of order $q$, where $q=p^{\alpha}$ is a prime power. The letters $p$ and $q$ will also be used to denote points in a plane or higher-dimensional vector space, but the context will ensure that no confusion arises.
\item For subsets $A,B$ of a field, write $A+B=\left\{a+b:a \in A, b \in B\right\}$ and $AB=\left\{ab:a \in A, b \in B\right\}$ for the sumset and product set of $A$ and $B$. The difference set $A-B$ and ratio set $A/B$ are defined analogously.
\item If $x$ is an element and $A$ a subset of a field then write $A+x$ for the translation $A+\left\{x\right\}$ and $xA$ for the dilation $\left\{x\right\}A$.
\item If $k$ is a natural number and $A$ is a subset of a field then write $kA$ for the iterated sumset $kA = \underbrace{A+\ldots+A}_k.$ Context will ensure that no confusion between a dilation and an iterated sumset of $A$ arises.
\end{itemize}

\chapter*{Prologue}
\addcontentsline{toc}{section}{\protect\hspace{-1.5 em} \emph{Prologue}}
%\begin{SingleSpace}
\markboth{Prologue}{Prologue}

\begin{footnotesize}
\begin{quote}
\emph{`And if you take one from three hundred and sixty-five what remains?'}

\emph{`Three hundred and sixty-four, of course.'}

\emph{Humpty Dumpty looked doubtful, `I'd rather see that done on paper,' he said.
}

\sourceatright{Lewis Carroll, Through the Looking-Glass}
\end{quote}
\end{footnotesize}
%\end{SingleSpace}
\vspace{1cm}

This thesis sets new records at the interface of two areas of mathematics: incidence geometry and growth. The first of these, incidence geometry, is about points and lines in a plane and incidences between them; a point is `incident' to a line if it lies on that line. The following sorts of problem are typical:
	\label{problems}
	\begin{itemize}[ ]
	\item \textbf{Incidence bounds.} \textit{Incidences are counted with multiplicity, in the sense that if several lines cross at a single point then we count several incidences. In total, how many incidences could there be between a finite set of points and a finite set of lines?}
	\item \textbf{Line counting.} \textit{A pair of distinct points determines a line. In total, how many distinct lines might be determined by pairs from a finite set of points?}
	\end{itemize}
There are many natural generalisations, for example to higher dimensional vector spaces instead of planes, to curves instead of lines, and so forth.

The second area, growth, is about how much bigger a finite subset of a field becomes after passing to its image under some function of two or more variables. Examples are:

\begin{itemize}[ ]
\item \textbf{Sum-product estimates.} \textit{For a set $A$ of given cardinality, how much bigger must the quantity $\max\left\{|A+A|,|AA|\right\}$ be? Here, $A+A$ and $AA$ are respectively the sets of pairwise sums and products of elements of $A$.}
\item \textbf{Expander functions.} \textit{Let $F$ be the ambient field and $n \geq 2$ be an integer. What functions $f:F^n \to F$ are there for which the set $f(A)=\left\{f(a_1,\ldots,a_n):a_i\in A\right\}$ is always much bigger than a finite subset $A$ of $F$?}
\end{itemize}

There is a common theme to all of this: estimating how much regimentation can be forced onto a finite set. For example, on the incidence geometry side the existence of too many incidences corresponds to the existence of extremely structured sets of points and lines. Understanding incidence bounds means placing a cap on such structure. As an example on the growth side, the only way in which the sumset $A+A$ or the product set $AA$ can be small is if the elements of $A$ are arranged respectively in some kind of additive or multiplicative `conspiracy', such as an arithmetic or geometric progression. Understanding sum-product estimates therefore means working out how far such conspiracies could coincide.   

There are applications of this theme, and of the specific problems above, to theoretical computer science and cryptography, where they can be used to obtain rigorous estimates of pseudorandom behaviour. However these are not pursued here, the motivation instead being for progress on the problems in their own right. The only exception to this approach of not worrying about applications is the interaction between the problems themselves. Incidence results have applications to growth, and growth results have applications to incidences, and this is one of the things that will contribute to progress on both.

When it comes to results the philosophy is one of `hard analysis', placing the emphasis on finding explicit bounds and making them as strong as possible. Quite how strong this might be depends very much on the choice of underlying field. Historically the attention was on real numbers. There are powerful results in this setting, which is still extremely active. With work, many incidence and growth results extend verbatim from the real to complex settings. More recently a lot of research has focused on finite fields\footnote{Recall that the finite field $\mathbb{F}_p$ of prime order $p$ is simply the set of residues modulo $p$ under addition and multiplication, and that the finite field $\mathbb{F}_q$ of prime power order $q=p^{\alpha}$ is a degree $\alpha$ algebraic extension of $\mathbb{F}_p$.}. This thesis presents work in all of these areas, and also breaks new ground in the relatively unexplored setting of function fields\footnote{The function field $\ffield$ is the field of rational functions over the finite field $\mathbb{F}_q$}.

\section*{Some remarks on finite fields}
Much, but by no means all, of the work in this thesis concerns finite fields, and it is worth pausing to explain a couple of considerations versus the real and complex settings. There is also a slightly different set of considerations in the study of function fields, but these are deferred to Chapter \ref{chapter:functionfield} since that is the only place they arise.

There are two main complications in finite fields. The first is the existence of finite subfields, which must be ruled out before anything nontrivial can be said. In the case of the field $\mathbb{F}_p$ of prime order $p$, this collapses into a cardinality condition since there are no proper subfields. For this reason much finite field research focuses on the prime order case; the hoops that one must jump through are similar to the general case, but the mathematics is usually cleaner. 

The second complication is an issue of tools and difficulty. Many methods in the real and complex settings depend critically on their topologies and so do not extend to finite fields. This makes it a lot harder to prove things and forces a more combinatorial approach, which is a mixed blessing. The bad news is that quantitative results are usually not as strong. But the good news is that methods and results achieved in finite fields usually extend elsewhere without difficulty. Thus finite field results carry a certain amount of moral authority as they correspond to the worst possible cases.

Although not considered here in detail, it is worth remarking that there is an additional `large set' paradigm for finite field results, with extensive literature. This features strong results for sets satisfying an additional minimum density condition, typically at least a square-root barrier. However the methods have a very different flavour, drawing for example on estimates for exponential sums. On top of this, they cannot be so easily exported to other settings; the interest here is in finite sets, but a minimum density condition in finite fields is usually analogous to a requirement for infinite sets elsewhere.

\section*{Structure of the thesis}

The first three chapters build up background and preliminaries.

\begin{itemize}
\item \textbf{Chapters \ref{chapter:introincidences} and \ref{chapter:introgrowth}} introduce incidence geometry and growth respectively in more detail. They also record particular incidence and growth results for use in later chapters.
\item
 \textbf{Chapter \ref{chapter:introapproxgroup}} is a handbook of results in sumset calculus, which is a workhorse of much of the later mathematics. 
\end{itemize}

The subsequent chapters are concerned with original results, developing state of the art incidence geometry and growth in several directions. Between them, there are new results on all four problems from page \pageref{problems}: incidence bounds, line counting, sum-product estimates and expander functions. Moreover they encompass real, complex and finite field settings, and also break new ground in the relatively unexplored function field setting.

\begin{itemize}
\item \textbf{Chapter \ref{chapter:incidences}} considers incidence geometry over finite fields and sets two new records. 
\begin{itemize}[\textasteriskcentered]
\item The first is a new incidence bound. If $P$ and $L$ are a set of points and lines respectively in $\mathbb{F}_p^2$ with $|P|,|L| \leq N$ then, writing $I(P,L)$ for the number of incidences between $P$ and $L$, non-trivial bounds are of the form  $I(P,L)\ll N^{\frac{3}{2}-\epsilon}$
for $\epsilon>0$. 

A new bound of $\epsilon \geq \frac{1}{662}-o(1)$ is established, holding whenever $N<p$. This improves by an order of magnitude on the previous bound of $\epsilon \geq \frac{1}{10,678}.$  

\item The second is a new line counting result. If $P$ is a set of points in $\mathbb{F}_p^2$ with $|P|<p$ then either $\Omega(|P|^{1-o(1)})$ of the points are collinear, or $P$ determines at least $\Omega\left(|P|^{1+\frac{1}{133}-o(1)}\right)$ 
distinct lines. 

This improves on previous results in two ways. Quantitatively, the exponent is stronger than the previously best-known  $1+\frac{1}{267}$. And qualitatively, the result applies to all subsets of $\mathbb{F}_p^2$ satisfying the cardinality condition; the previously best-known result applies only for $P$ of the form $P=A \times A$ for $A \subseteq \mathbb{F}_p$. 
\end{itemize}

\item \textbf{Chapter \ref{chapter:expanders}} sets three new records for expander functions. 
\begin{itemize}[\textasteriskcentered]
\item First, there is a result on two-variable expanders in finite fields. If $f(a,b)=a(b+1)$ then $|f(A)|\gg |A|^{1+\frac{1}{53}-o(1)}$ whenever $A$ is a subset of $\mathbb{F}_p$ with $|A|<p^{1/2}$. This improves on the previous best-known exponent of $1+\frac{1}{106}-o(1)$.

\item Second, there is a result on three-variable expanders for complex numbers. The function $g(a,b,c)=\frac{a-b}{a-c}$, which has not been previously considered in this context, is shown to satisfy $|g(A)|\gg |A|^{2-o(1)}$ for any finite $A\subseteq \mathbb{C}$. Previously-known functions with this property were all of four variables and applied only to sets of real numbers.   

\item Third, there is a result on four-variable expanders for real numbers. The function $h(a,b,c,d)=\frac{(a-b)(c-d)}{(b-c)(a-d)},$ again not previously considered in this context, is shown to satisfy $|h(A)|\gg |A|^2$ whenever $A$ is a finite set of real numbers. The previously best-obtained bound for a function of four variables was $|A|^{2-o(1)}$.
\end{itemize}

\item \textbf{Chapter \ref{chapter:functionfield}} is joint work with Thomas Bloom and breaks new ground by establishing a sum-product estimate in the function field $\ffield$, a relatively unexplored setting for this kind of work. Function fields form an interesting intermediate case between the finite field and real and complex settings since they have an unusually rigid `non-archimedean' topology. 

A sum-product estimate of
	$$\max\left\{|A+A|,|AA|\right\}\gg_q |A|^{1+\frac{1}{5}-o(1)}$$
is established for any finite subset $A$ of $\ffield$. The exponent of $1+\frac{1}{5}-o(1)$ lies between the $1+\frac{1}{11}-o(1)$ known for finite fields and the $1+\frac{1}{3}-o(1)$ known for real and complex numbers. 
\end{itemize}

There are two appendices, which summarise standard background information. 
\begin{itemize}
\item \textbf{Appendix \ref{chapter:pigeon}} covers standard pigeonholing results used throughout the thesis. The phrases `by averaging', `by Cauchy-Schwarz' and `by dyadic pigeonholing' are deployed frequently and implicitly refer to results from here.
\item \textbf{Appendix \ref{chapter:proj}} covers material on projective geometry necessary for Chapter \ref{chapter:incidences} and Chapter \ref{chapter:expanders}. Additional standard material on the projective theory of cross ratios is covered at an appropriate point in Chapter \ref{chapter:expanders}.
\end{itemize}

%% file: introincidences.tex
\chapter{Incidence geometry}
\label{chapter:introincidences}

The prologue mentioned two areas of incidence geometry: incidence bounds and line counting. This chapter introduces them properly, with an emphasis on the real and complex settings, the case of finite fields being deferred to Chapter \ref{chapter:incidences}. In so doing it also records results which will be of use in Chapter \ref{chapter:expanders}.

Incidence bounds are covered first, establishing a trivial estimate which holds regardless of the underlying field. Non-trivial incidence bounds depend on the setting; the best-understood case is the plane $\mathbb{R}^2$, where the classical Szemer\'edi-Trotter theorem holds. Line counting is then tackled in $\mathbb{R}^2$, establishing Beck's theorem as a consequence of Szemer\'edi-Trotter.

Three ways of generalising to higher-dimensions are also considered, as is a generalisation from the real to complex setting.

Apart from Theorem \ref{theorem:circlebeck}, which is a straightforward variation of existing results, all of the mathematics in this chapter is drawn from the literature.

\section{Incidence bounds in a plane}\label{section:incidencesintrointro}

A \textbf{plane} is simply $F^2$ where $F$ is a field. A \textbf{point} is an element of $F^2$ and a \textbf{line} is the set of points $(x,y) \in F^2$ satisfying an equation 
	$$ax+by+c=0$$
for fixed $a,b,c \in F$ that are not all zero. A point $p$ is \textbf{incident} to a line $l$ if $p \in l$. If $P$ is a finite set of points in a plane, and $L$ is a finite set of lines, then write $I(P,L)$ for the number of incidences between points in $p$ and lines in $l$, that is
\begin{equation}\label{eq:incidencecount}
I(P,L)=\sum_{p\in P}\sum_{l \in L}\delta_{pl}.
\end{equation}
where
	$$
	\delta_{pl}=\left\{
	\begin{array}{cc}
	1,& \text{if } p \in l\\
	0,& \text{if } p \notin l.
	\end{array}
	\right.
	$$
Incidences are therefore counted with multiplicity, as illustrated in Figure \ref{fig:countingincidences}.

\begin{figure}[ht]
\vspace{20pt}
\centering
\subbottom[One incidence]{
\begin{tikzpicture}[scale=4]
\coordinate (A) at (0,0);
\coordinate (B) at (1,1);
\draw [name path=A--B] (A) -- (B);
\node (Y) at ($ (A)!.5!(B) $) {};
\foreach \point in {Y}
\fill [red,opacity=1] (\point) circle (0.75pt);
\end{tikzpicture}
}
\hspace{20pt}\subbottom[Two incidences]{
\begin{tikzpicture}[scale=4]
\coordinate (A) at (0,0);
\coordinate (B) at (1,1);
\coordinate (C) at (0,1);
\coordinate (D) at (1,0);
\draw [name path=A--B] (A) -- (B);
\draw [name path=C--D] (C) -- (D);
\path [name intersections={of=A--B and C--D,by={X}},];
\foreach \point in {X}
\fill [red,opacity=1] (\point) circle (0.75pt);
\end{tikzpicture}
}
\hspace{20pt}\subbottom[Five incidences]{
\begin{tikzpicture}[scale=4]
\coordinate (A) at (0,0);
\coordinate (B) at (1,1);
\coordinate (C) at (0,1);
\coordinate (D) at (1,0);
\coordinate (E) at (0.5,1);
\coordinate (F) at (1,0.5);
\draw [name path=A--B] (A) -- (B);
\draw [name path=C--D] (C) -- (D);
\draw [name path=E--F] (E) -- (F);
\path [name intersections={of=A--B and E--F,by={Y}},];
\path [name intersections={of=A--B and C--D,by={X}},];
\node (Z) at ($ (C)!.75!(D) $) {};
\foreach \point in {X,Y,Z}
\fill [red,opacity=1] (\point) circle (0.75pt);
\end{tikzpicture}
}
\caption{Counting incidences}\label{fig:countingincidences}
\end{figure}
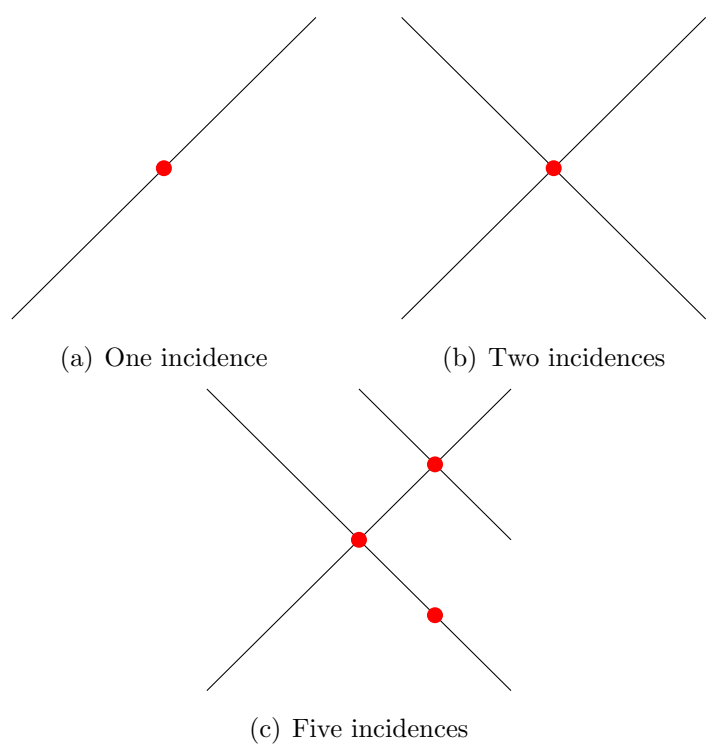

It is worth remarking (see Appendix \ref{chapter:proj}) that a `plane' could instead be taken to be the projective plane $\mathbb{P}F^2$ containing $F^2$, and `lines' to be projective lines within it. Indeed, all the results in this chapter extend without difficulty to projective space. But since no properties of the projective setting will be used, it makes sense to stay with the affine one. Chapters \ref{chapter:incidences} and \ref{chapter:expanders} make active use of the properties of projective space, and we will cross that bridge when we come to it. 

An important topic in incidence geometry is the study of upper bounds on $I(P,L)$ in terms of $|P|$ and $|L|$, referred to as \textbf{incidence bounds}. When considering these, the most straightforward observation to make is that each point in $P$ is incident to at most $|L|$ lines in $L$, implying that
	$$I(P,L)\leq |P||L|.$$
However it is immediate from the definition of points and lines that two lines in $L$ are simultaneously incident to at most one point in $P$, and two points in $P$ are simultaneously incident to at most one line in $L$. This fact, combined with the Cauchy-Schwarz inequality, leads to the following somewhat better estimate, which is nonetheless referred to as trivial.

\begin{lemma}[Trivial incidence bound]\label{theorem:trivialincidences}
If $P$ is a set of points in a plane, and $L$ is a set of lines, then
	$$I(P,L)\ll \min \left\{|P|+|P|^{1/2}|L|,|L|+|L|^{1/2}|P|\right\}.$$
\end{lemma}

\begin{proof}
Cauchy-Schwarz and (\ref{eq:incidencecount}) imply that
	\begin{align}
	I(P,L)^2& = \left(\sum_{p \in P}\sum_{l \in L}\delta_{pl}\right)^2\nonumber\\
	& \leq |P|\sum_{p \in P}\sum_{l_1,l_2 \in L}\delta_{pl_1}\delta_{pl_2}.\label{eq:introdincidences1}
	\end{align}
Split the summation over lines $l_1,l_2 \in L$ according to whether or not $l_1$ and $l_2$ are distinct, so that
	\begin{align}
	\sum_{p \in P}\sum_{l_1,l_2 \in L}\delta_{pl_1}\delta_{pl_2}&=\sum_{p \in P}\sum_{l \in L}\delta_{pl}+\sum_{p \in 	
	P}\sum_{l_1\neq l_2 \in L}\delta_{pl_1}\delta_{pl_2}\nonumber\\
	&= I(P,L)+ \sum_{p \in P}\sum_{l_1\neq l_2 \in L}\delta_{pl_1}\delta_{pl_2}.\nonumber
	\end{align}
Substituting into \eqref{eq:introdincidences1} yields
	\begin{align}
	I(P,L)^2 &\leq |P|I(P,L)+ |P|\sum_{l_1\neq l_2 \in L}\sum_{p \in P}\delta_{pl_1}\delta_{pl_2}.\label{eq:introincidences3}
	\end{align}
Two distinct lines are simultaneously incident to at most one point, and so 
	$$\sum_{p \in P}\delta_{pl_1}\delta_{pl_2}\leq 1$$
whenever $l_1\neq l_2$, since this is the number of points in $P$ incident to both $l_1$ and $l_2$. Hence from \eqref{eq:introincidences3},
	$$I(P,L)^2\leq |P|I(P,L)+|P||L|^2.$$
So either $I(P,L)^2 \ll |P|I(P,L)$ in which case $I(P,L)\ll |P|$, or $I(P,L)^2\ll |P||L|^2$ 
in which case $I(P,L)\ll |P|^{1/2}|L|$. Overall therefore 
	$$I(P,L)\ll |P|+|P|^{1/2}|L|.$$

The proof that $I(P,L)\ll |L|+|L|^{1/2}|P|$ is the same, except that the roles of points and lines are reversed.
\end{proof}

It is often helpful to consider the `critical' case where $P$ and $L$ have the same cardinality, say $|P|=|L|=N$. In this case, the trivial bound from Lemma \ref{theorem:trivialincidences} is 
	$$I(P,L)\ll N^{3/2}$$
and so \textbf{non-trivial} bounds will be of the form 
	$$I(P,L)\ll N^{3/2-\epsilon}$$ 
for $\epsilon>0$. 

The following standard result shows that best non-trivial bound that can be hoped for is $\epsilon=\frac{1}{6}$, that is $I(P,L)\ll N^{4/3}$. 

\begin{lemma}[Constraint on incidence bounds]\label{theorem:incidences1/6}
Let $F$ be a field. If $F$ has characteristic $0$ then for any positive integer $N$ there is a set $P$ of points in $F^2$ and a set $L$ of lines, with $|P|,|L|\approx N$ such that 
	$$I(P,L)\approx N^{4/3}.$$ 
If the characteristic of $F$ is $p>0$ then the same result holds, but with the constraint that $N$ must be less than $\left(\frac{p}{2}\right)^{3/2}$. 
\end{lemma}

\begin{proof}
Write $l_{rs}$ for the line given by $y=rx+s$. In the zero-characteristic case take 
	\begin{align*}
	P&=\left\{(x,y) \in \mathbb{Z}^2:1\leq x \leq N^{1/3},1\leq y \leq 2N^{2/3}\right\}\\
	L&=\left\{l_{rs}:(r,s) \in \mathbb{Z}^2:1\leq r \leq N^{1/3},1\leq s \leq N^{2/3}\right\}.
	\end{align*}
There are $\Theta\left(N\right)$ distinct points in $P$ and $\Theta\left(N\right)$ distinct lines in $L$, and it is easy to check that each line in $L$ is incident to $\Theta\left(N^{1/3}\right)$ points in $P$, meaning that there are $\Theta(N^{4/3})$ incidences. 

The characteristic $p>0$ case is the same; simply replace $\mathbb{Z}^2$ in the construction with $\mathbb{F}_p^2$. The constraint that $N<\left(\frac{p}{2}\right)^{3/2}$ ensures that all the points and lines are distinct. 
\end{proof}

\section{The Szemer\'edi-Trotter theorem}

What can be said about non-trivial incidence bounds? The answer depends on the underlying field over which the plane is defined. This section considers the classical case of the plane $\mathbb{R}^2$ where Szemer\'edi and Trotter \cite{ST} proved the following flagship result of incidence geometry.

\begin{theorem}[Szemer\'edi, Trotter]\label{theorem:ST}
If $P$ is a set of points in $\mathbb{R}^2$ and $L$ is a set of lines, then
	$$I(P,L)\ll |P|^{2/3}|L|^{2/3}+|P|+|L|.$$
\end{theorem}

In the critical case this gives $I(P,L)\ll N^{4/3}$, which is sharp up to the implicit constant by Lemma \ref{theorem:incidences1/6}. Before proving Szemer\'edi-Trotter it is worth recording an immediate consequence which is often useful in applications to other problems.

\begin{corollary}\label{theorem:ST'}
Let $L$ be a set of lines in $\mathbb{R}^2$. Then the number of points incident to at least $k$ lines in $L$ is $O\left(\frac{|L|^2}{k^3}+\frac{|L|}{k}\right)$. Similarly, the number of lines incident to at least $k$ points in a point set $P$ is $O\left(\frac{|P|^2}{k^3}+\frac{|P|}{k}\right)$.
\end{corollary}

\begin{proof}
Let $P_k$ be the set of points incident to at least $k$ lines in $L$. Then 
	$$|P_k|k\leq I(P_k,L).$$
On the other hand, the Szemer\'edi-Trotter theorem shows that
	$$I(P_k,L) \ll |P_k|^{2/3}|L|^{2/3}+|L| $$
and so comparing upper and lower bounds on $I(P_k,L)$ gives  
	$$|P_k|\ll \frac{|L|^2}{k^3}+\frac{|L|}{k}$$
as required. The proof for the number of lines incident to at least $k$ points is similar.
\end{proof}

Now for a proof of the Szemer\'edi-Trotter theorem. There have been several proofs since Szemer\'edi and Trotter's original. The one here is due to Sz\'ekely \cite{szekely}. A third proof, using different techniques again, can be found in a recent paper of Kaplan, Matou\u{s}ek and Sharir \cite{KMS}.

\begin{proof}[Proof of Theorem \ref{theorem:ST}.]
Without loss of generality assume that every point is incident to at least one line, and every line is incident to at least one point. Recall that a \textbf{graph} $G(V,E)$ consists of a set $V$ of \textbf{vertices}, and a set $E$ of unordered pairs of vertices, called \textbf{edges}.

Construct a graph $G(V,E)$ as follows. Take the set $V$ of vertices to be the set of points $P$. To construct the set $E$ of edges, say that $(p_1,p_2)$ is an edge if and only if $p_1$ and $p_2$ are adjacent along a line in $L$, in the sense that the line segment connecting them is contained in a line from $L$ and contains no other points from $P$.

For each $l \in L$ write $k(l)=\sum_{p \in P}\delta_{pl}$ for the number of points $p \in P$ that are incident to $L$. Note that $l$ contains $k(l)-1$ edges from $E$, as illustrated in Figure \ref{fig:STproof}. 

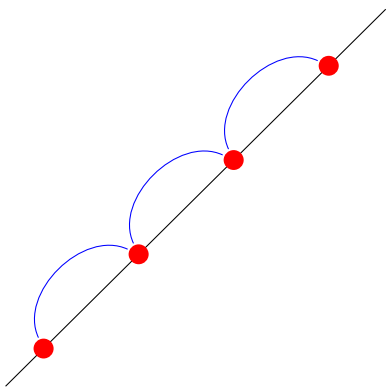
\begin{figure}[ht]
\vspace{20pt}
\centering
\begin{tikzpicture}[scale=5]
\coordinate (A) at (0,0);
\coordinate (B) at (1,1);
\draw [name path=A--B] (A) -- (B);

\node (W) at ($ (A)!.1!(B) $) {};
\node (X) at ($ (A)!.35!(B) $) {};
\node (Y) at ($ (A)!.6!(B) $) {};
\node (Z) at ($ (A)!.85!(B) $) {};

\foreach \point in {W,X,Y,Z}
\fill [red,opacity=1] (\point) circle (0.75pt);

\draw [-,color=blue] (W) to [bend left=70](X);
\draw [-,color=blue] (X) to [bend left=70](Y);
\draw [-,color=blue] (Y) to [bend left=70](Z);

\end{tikzpicture}

\caption{$k$ collinear points determine $k-1$ edges.}
\label{fig:STproof}
\end{figure}

It follows that
	\begin{align}
	|E|&= \sum_{l \in L}(k(l) -1)\nonumber \\
	&=\sum_{l \in L}\sum_{p \in P}\delta_{pl} - |L|\nonumber \\
	&=I(P,L) - |L|\label{eq:edgeincidence}. 
	\end{align}
	
We now apply a graph-theoretic result. A \textbf{drawing} of $G=G(V,E)$ is an identification of each vertex in $V$ with a distinct point in $\mathbb{R}^2$, and of each edge $(u,v)\in E$ with a curve connecting $u$ and $v$. A \textbf{crossing} occurs whenever two curves corresponding to edges intersect away from a vertex. Write $\text{cross}(G)$ for the minimum number of crossings in a drawing of $G$. We will employ the following lemma; see Chapter 8 of the book of Tao and Vu \cite{TV} for a proof.

\begin{lemma}[Crossing number lemma]
Let $G(V,E)$ be a graph with $|E| \geq 4|V|$. Then $\text{cross}(G) \geq \frac{|E|^3}{64 |V|^2}$.
\end{lemma}

Applying the crossing number lemma and \eqref{eq:edgeincidence} shows that at least one of the following bounds holds 
	\begin{align}
	I(P,L) &\ll |P| +|L|\label{eq:crossingfails}\\
	\text{cross}(G) &\gg \frac{I(P,L)^3}{|P|^2}.\label{eq:crossinglower}
	\end{align}
Let's consider what happens when \eqref{eq:crossinglower} holds. It is clear that 
	\begin{equation}\label{eq:crossingupper}
	\text{cross}(G) \leq |L|^2
	\end{equation}
since for two edges to cross it is necessary that two lines cross. Comparing the bounds \eqref{eq:crossinglower} and \eqref{eq:crossingupper} yields 
	$$ \frac{I(P,L)^3}{|P|^2}\ll\text{cross}(G)\leq |L|^2$$
and hence 
	\begin{equation}\label{eq:crossingholds}
	I(P,L)\ll |P|^{2/3}|L|^{2/3}
	\end{equation}
whenever \eqref{eq:crossinglower} holds. Thus either \eqref{eq:crossingfails} or \eqref{eq:crossingholds} holds, implying overall that
	$$I(P,L)\ll |P|^{2/3}|L|^{2/3}+|P|+|L|$$
as required.	
\end{proof}

The key to the above proof is the crossing number inequality. Beyond this the only properties used are the trivial facts that any two lines are simultaneously incident to at most one point, and any two points are simultaneously incident to at most one line. With a few modifications\footnote{See for example Theorem 8.10 of the book of Tao and Vu \cite{TV}.} the same argument goes through when $L$ is a set of curves rather than lines, such that any two points in $P$ are simultaneously incident to at most $\alpha$ curves and any two curves in $L$ are simultaneously incident to at most $\beta$ points in $P$. The implicit constant in the statement then depends on $\alpha$ and $\beta$ and so the theorem becomes 
	$$I(P,L)\ll_{\alpha,\beta} |P|^{2/3}|L|^{2/3}+|P|+|L|.$$	
The required changes are to replace the crossing number lemma for graphs with one for multigraphs, and to replace \eqref{eq:crossingupper} with 
	$$\text{cross}(G) \leq \beta |L|^2.$$	

Pach and Sharir \cite{pachsharir} took this approach further, to encompass the case where any $k$ points, rather than any two points, are simultaneously incident to at most $\alpha$ curves. Note that the condition on curves remains the same; any \textit{two} curves are simultaneously incident to at most $\beta$ points, rather than any $k$.

\begin{theorem}[Pach, Sharir] \label{theorem:pachsharir}
Let $P$ be a set of points in $\mathbb{R}^2$ and $L$ be a set of curves, such that any $k$ points in $P$ are simultaneously incident to at most $\alpha$ curves in $L$ and any two curves in $L$ are simultaneously incident to at most $\beta$ points in $P$. Then
	$$I(P,L)\ll_{\alpha,\beta} |P|^{\frac{k}{2k-1}}|L|^{\frac{2k-2}{2k-1}}+|P|+|L|.$$	
\end{theorem}
	
\section{Line counting and Beck's theorem}

This section gives an application of the Szemer\'edi-Trotter theorem to the \textbf{line counting} problem. Two distinct points in $\mathbb{R}^2$ determine a line, and so for a set $P$ of points we can consider the set $L(P)$ of lines determined by pairs of points in $P$. This could have just one element, which would happen if all the points in $P$ lie along a single line as in Figure \ref{figure:beck} (a). Or it could be as large as $\binom{|P|}{2}\approx |P|^2$ if the points are in general position as in Figure \ref{figure:beck} (b). 

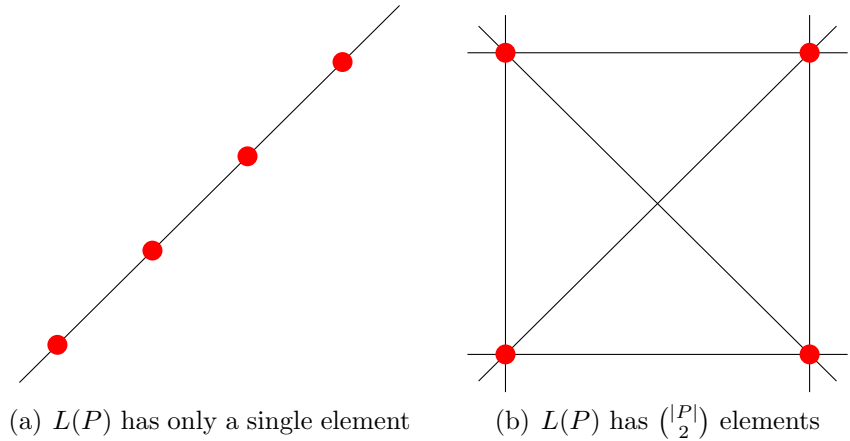
\begin{figure}[ht]
\centering
\subbottom[$L(P)$ has only a single element]
{
\begin{tikzpicture}[scale=5]
\coordinate (A) at (0,0);
\coordinate (B) at (1,1);
\draw [name path=A--B] (A) -- (B);
\node (W) at ($ (A)!.1!(B) $) {};
\node (X) at ($ (A)!.35!(B) $) {};
\node (Y) at ($ (A)!.6!(B) $) {};
\node (Z) at ($ (A)!.85!(B) $) {};
\foreach \point in {W,X,Y,Z}
\fill [red,opacity=1] (\point) circle (0.75pt);
\end{tikzpicture}
}
\hspace{20pt}
\subbottom[$L(P)$ has $\binom{|P|}{2}$ elements]
{
\begin{tikzpicture}[scale=5]
\coordinate (O) at (0.1,0.05);
\coordinate (A) at (0.1,0.1);
\coordinate (B) at (0.1,0.9);
\coordinate (C) at (0.9,0.9);
\coordinate (D) at (0.9,0.1);
\draw [shorten >=-0.5cm,shorten <=-0.5cm] (A) -- (B);
\draw [shorten >=-0.5cm,shorten <=-0.5cm] (A) -- (C);
\draw [shorten >=-0.5cm,shorten <=-0.5cm] (A) -- (D);
\draw [shorten >=-0.5cm,shorten <=-0.5cm] (B) -- (C);
\draw [shorten >=-0.5cm,shorten <=-0.5cm] (B) -- (D);
\draw [shorten >=-0.5cm,shorten <=-0.5cm] (C) -- (D);
\foreach \point in {A,B,C,D}
\fill [red,opacity=1] (\point) circle (0.75pt);
\foreach \point in {O}
\fill [red,opacity=0] (\point) circle (0.75pt);
\end{tikzpicture}
}
\caption{The two extremes for counting lines.}
\label{figure:beck}
\end{figure}

The following theorem of Beck \cite{beck} asserts that, up to multiplicative constants, these two extremes are essentially the only two possibilities.

\begin{theorem}[Beck]\label{theorem:beck}
If $P$ is a set of points in $\mathbb{R}^2$ then at least one of the following holds
\begin{enumerate}
\item At least $\Omega\left(|P|\right)$ points in $P$ are incident to a single line.
\item $|L(P)|\approx|P|^2$. 
\end{enumerate}
\end{theorem}

Beck's theorem was originally published in the same journal edition as the Szemer\'edi-Trotter theorem. But it is in fact also a consequence of Szemer\'edi-Trotter. This latter proof, given here, is more straightforward and can be found for example as exercise 8.2.6 in the book of Tao and Vu \cite{TV}.

\begin{proof}[Proof of Theorem \ref{theorem:beck}]
For each $l \in L(P)$ write $\mu(l)=\sum_{p \in P}\delta_{pl}$ for the number of points $p \in P$ that are incident to $l$. It is clear that
	$$\mu(l)^2 \approx \sum_{p_1 \neq p_2 \in P}\delta_{p_1l}\delta_{p_2l}$$
since the right-hand side counts pairs of distinct points in $P$ incident to $l$. It follows that
	\begin{align*}
	\sum_{l \in L(P)}\mu(l)^2 &\approx \sum_{p_1\neq p_2 \in P}\sum_{l \in L(P)}\delta_{p_1l}\delta_{p_2l}.
	\end{align*}
For fixed $p_1 \neq p_2$ there is precisely one line in $L(P)$ that is incident to $p_1$ and $p_2$, so 
	$$\sum_{l \in L(P)}\delta_{p_1l}\delta_{p_2l}=1,$$
and hence by combining the above equations
	\begin{equation}\label{eq:introincidences4}
	\sum_{l \in L(P)}\mu(l)^2 \approx |P|^2.
	\end{equation}
	
Now for each integer $j$ let $L_{j}$ be the set of $l \in L(P)$ for which $2^{j}\leq \mu(l) < 2^{j+1}.$ Corollary \ref{theorem:ST'} implies that  
	\begin{equation}\label{eq:beckst}
	|L_j|\ll \frac{|P|^2}{2^{3j}}+\frac{|P|}{2^j}.
	\end{equation}
For points $p,q \in P$ write $l_{pq}$ for the line determined by $p$ and $q$. For any constant $C$, let $X_C$ be the set of pairs of distinct points $p,q \in P$ for which  $$C \leq \mu(l_{pq}) \leq \frac{|P|}{C}.$$
From \eqref{eq:beckst},
	\begin{align}
	|X_C| &\approx \sum_{j = \log_2 C}^{\log_2 (|P|/C)} |L_j|2^{2j} \nonumber\\
	&\ll \sum_{j = \log_2 C}^{\log_2 (|P|/C)} \left(\frac{|P|^2}{2^{3j}}+\frac{|P|}{2^j}\right)2^{2j}\label{eq:introincidences5}
	\end{align}
Summing geometric series yields 
	\begin{align*}
	\sum_{j = \log_2 C}^{\log_2 (|P|/C)} \frac{1}{2^{j}} &\ll \frac{1}{C}\\
	\sum_{j = \log_2 C}^{\log_2 (|P|/C)} 2^{j}&\ll \frac{|P|}{C}
	\end{align*}
and so (\ref{eq:introincidences5}) implies
	\begin{equation}
	|X_C| \ll \frac{|P|^2}{C}.\label{eq:introincidences6}
	\end{equation}
Comparing (\ref{eq:introincidences4}) and (\ref{eq:introincidences6}) it is possible to pick a constant $C$ sufficiently large that
	$$|X_C| \leq \frac{1}{2}\sum_{l \in L(P)}\mu(l)^2.$$
Thus for this fixed $C$, either 
	$$\sum_{l \in L(P):\mu(l)\geq |P|/C}\mu(l)^2 \gg |P|^2.$$
or
	$$\sum_{l \in L(P):\mu(l)\leq C}\mu(l)^2 \gg |P|^2$$ 

In the former case there are $\Omega(|P|^2)$ pairs $(p,q)$ for which $\mu(l_{pq})\geq |P|/C$, and so in particular there is at least one line in $L(P)$ that is incident to $\Omega(|P|)$ points. This corresponds to the first case in the conclusion of the theorem.
	
In the latter case there are $\Omega(|P|^2)$ pairs of points $p,q \in P$ for which $\mu(l_{pq})\leq C$, meaning that there are $\Omega(|P|^2)$ distinct lines in $L(P)$. This corresponds to the second case in the conclusion of the theorem. 
\end{proof}

Since it is a consequence of Szemer\'edi-Trotter, which extends to points and curves, the statement of Beck's theorem generalises in the same way. This can be further developed by using the Pach-Sharir theorem in place of Szemer\'edi-Trotter, for example to the following result. 

\begin{theorem}\label{theorem:circlebeck}
For finite $P \subseteq \mathbb{R}^2$, write $C(P)$ for the set of circles determined by non-collinear triples of points from $P$. For any $P$, at least one of the following holds:
\begin{enumerate}
\item There are $\Omega(|P|)$ cocircular points in $P$.
\item There are $\Omega(|P|)$ collinear points.
\item $|C(P)|\gg |P|^3$.
\end{enumerate}
\end{theorem}

The details of the proof are similar to that of Beck's theorem, using instead the $k=3$ case of the Pach-Sharir theorem and the fact that three distinct non-collinear points determine a circle.

\section{Incidence geometry in $\mathbb{R}^n$}\label{section:highdim}

It is interesting, and useful for applications to other problems, to consider incidence geometry in higher-dimensional vector spaces than the plane. This section generalises the material of the previous two sections from $\mathbb{R}^2$ to $\mathbb{R}^n$. Since Beck's line-counting theorem follows from the Szemer\'edi-Trotter incidence bound, the focus is on higher-dimensional incidence bounds, with their application to line counting in $\mathbb{R}^n$ left as an exercise.

There are several ways to formulate higher-dimensional results. One is to simply consider point-line incidences in the higher-dimensional setting. Another is to consider higher-dimensional objects than lines, for example incidences between points and planes in $\mathbb{R}^3$. Somewhere between these two is the study of `pseudolines' which may be higher dimensional than lines but intersect pairwise in at most $O(1)$ points. 

\subsection{Points and lines} \label{section:highdimline}

The Szemer\'edi-Trotter theorem extends by a random projection argument to the setting where the points and lines lie in $\mathbb{R}^n$ rather than $\mathbb{R}^2$. Simply project $P$ and $L$ onto a randomly chosen $2$-plane in $\mathbb{R}^n$. With positive probability the number of incidences between the projected points and projected lines will be the same as those between the originals.

Since Szemer\'edi-Trotter is sharp in $\mathbb{R}^2$ it is also sharp in $\mathbb{R}^n$. However improvements are available subject to constraints on the arrangements of lines so that, for example, not too many of them lie in any one plane. Guth and Katz \cite{GK} obtained a breakthrough result of this kind in $\mathbb{R}^3$ which enabled them to solve the famous `distinct distances' problem of Erd\"os. Their incidence result is as follows:

\begin{theorem}[Guth, Katz]\label{theorem:gk}
Let $L$ be a set of lines in $\mathbb{R}^3$ of which no more than $O\left(|L|^{1/2}\right)$ are simultaneously incident to the same point, plane or regulus\footnote{A regulus is a `doubly-ruled' surface. That is, every point in a regulus $R$ is incident to at least two lines entirely contained in $R$.}. Then the number of points in $\mathbb{R}^3$ incident to at least $k$ lines in $L$ is at most $O\left(\frac{|L|^{3/2}}{k^2}\right).$  
\end{theorem}

The bound in Theorem \ref{theorem:gk} should be compared to the Szmer\'edi-Trotter bound $O\left(\frac{|L|^{2}}{k^3}\right)$ arising from Corollary \ref{theorem:ST'} when $k \ll |L|^{1/2}$. The two bounds agree when $k \approx |L|^{1/2}$ but Theorem \ref{theorem:gk} is stronger for smaller $k$.

\subsection{Points and planes in $\mathbb{R}^3$}\label{section:highdimplane}

Some kind of nondegeneracy condition is necessary to say anything interesting about incidences between a set $P$ of points and a set $\Pi$ of planes, since if all the planes in $\Pi$ intersect along a single line that is incident to all the points in $P$, then $I(P,\Pi)=|P||\Pi|.$

Edelsbrunner, Guibas and Sharir \cite{EGS} proved the following result under the fairly strong condition that no three planes are collinear. There are also plenty of examples \cite{agarwalaronov,AS,brassknauer,ET} of point-plane incidence results under other nondegeneracy conditions.

\begin{theorem}[Edelsbrunner, Guibas, Sharir]\label{theorem:EGS}
Let $P$ and $\Pi$ be a set of points and planes respectively in $\mathbb{R}^3$. If no three planes are collinear then 
	$$I(P,\Pi)\ll |P|^{4/5}|\Pi|^{3/5}+|P|+|\Pi|.$$ 
\end{theorem}

In the original paper \cite{EGS} this bound is multiplied by a factor of the form $|P|^{o(1)}|\Pi|^{o(1)}$. However Apfelbaum and Sharir \cite{AS} showed that this additional factor can be eliminated with more careful analysis and so the refined version is used here.

Like the Szemer\'edi-Trotter theorem, the Edelsbrunner-Guibas-Sharir theorem has a standard corollary, the derivation of which follows precisely as for Corollary \ref{theorem:ST'}. 

\begin{corollary}\label{theorem:EGS'}
Let $\Pi$ be a set of planes in $\mathbb{R}^3$, no three of which are collinear. Then the number of points incident to at least $k$ planes in $\Pi$ is $O\left(\frac{|\Pi|^3}{k^5}+\frac{|\Pi|}{k}\right).$
\end{corollary}
	
\subsection{Points and pseudolines}\label{section:psuedolines}

A set $P$ of points and a collection $V$ of varieties in $\mathbb{R}^n$ form a \textbf{pseudoline system} if

\begin{enumerate}
\item Any pair of varieties from $V$ intersect in at most $O(1)$ points in $P$.
\item Any pair of points from $P$ are simultaneously incident to at most $O(1)$ varieties in $V$.
\item Points in $P$ can be incident to varieties in $V$ only at their smooth points.
\item If two varieties intersect at a point in $P$ then their tangent spaces intersect \textit{only} at that point.
\end{enumerate}

Solymosi and Tao \cite{SolymosiTao} proved the following result which generalises the Szemer\'edi-Trotter theorem, up to a loss of $o(1)$ in the exponent of $|P|$, to pseudoline systems in $\mathbb{R}^n$ for which the varieties are of dimension at most $n/2$. 

\begin{theorem}[Solymosi, Tao]
Let $n \geq 2k$ and let $P$ and $V$ form a pseudoline system in $\mathbb{R}^n$. If the varieties in $V$ are at most $k$-dimensional, and all have degree at most $O(1)$ then
	$$I(P,V)\ll_{\epsilon} |P|^{2/3+\epsilon}|V|^{2/3} +|P| +|L|.$$
\end{theorem}

The case with $n=4$ and $k=2$ is particularly important in the next section, as it corresponds to a Szemer\'edi-Trotter theorem for complex numbers. In this case, Zahl \cite{zahl} succeeded in eliminating the $o(1)$ loss to give the following result.

\begin{theorem}[Zahl]\label{theorem:zahl}
Let $P$ and $V$ form a pseudoline system in $\mathbb{R}^4$. If the varieties in $V$ are all $2$-dimensional, and all have degree at most $O(1)$ then
	$$I(P,V)\ll |P|^{2/3}|V|^{2/3} +|P| +|L|.$$
\end{theorem}

\section{Other fields}

So far we have considered only real vector spaces. But we are also interested in other settings, in particular complex numbers and finite fields. Chapter \ref{chapter:incidences} examines incidence geometry over finite fields and proves several original results, so we keep our powder dry and consider only the complex setting here.

T\'oth \cite{toth} announced\footnote{The 2011 preprint referenced here is a more recent version of the 2003 original.} the following generalisation of the Szemer\'edi-Trotter theorem to $\mathbb{C}^2$ in 2003, but nearly ten years later it is still in the review process.

\begin{theorem}[T\'oth]\label{theorem:toth}
If $P$ is a set of points in $\mathbb{C}^2$ and $L$ is a set of lines, then
	$$I(P,L)\ll |P|^{2/3}|L|^{2/3}+|P|+|L|.$$
\end{theorem}    

The usual corollary to Szmer\'edi-Trotter applies here as well:

\begin{corollary}\label{theorem:toth'}
Let $L$ be a set of lines in $\mathbb{C}^2$. Then the number of points incident to at least $k$ lines in $L$ is $O\left(\frac{|L|^2}{k^3}+\frac{|L|}{k}\right)$. Similarly, the number of lines incident to at least $k$ points in $P$ is $O\left(\frac{|P|^2}{k^3}+\frac{|P|}{k}\right)$.
\end{corollary}

More recently, the work of Solymosi, Tao and Zahl on pseudolines yields T\'oth's result as a special case. Indeed, lines in $\mathbb{C}^2$ are pseudolines in $\mathbb{R}^4$ and so Theorem \ref{theorem:toth} is an immediate corollary of Theorem \ref{theorem:zahl}.

One can of course then consider complex versions of the higher-dimensional generalisations in Section \ref{section:highdim}. Theorem \ref{theorem:toth} extends to point-line incidences $\mathbb{C}^n$, in the same way as Szemer\'edi-Trotter, although it is not known if the Guth-Katz theorem for $\mathbb{R}^3$ generalises to $\mathbb{C}^3$.

%% file: introgrowth.tex
\chapter{Growth}
\label{chapter:introgrowth}

Just as the last chapter did for incidence geometry, the purpose here is to introduce the two examples of growth highlighted in the prologue: sum-product estimates and expander functions. The chapter also records particular results that will be useful in Chapter \ref{chapter:incidences}.

The idea for sum-product estimates is that, for a subset $A$ of a field, at least one of the sumset $A+A$ and the product set $AA$ must always be large. Chapter \ref{chapter:incidences} will use finite field sum-product estimates as a key ingredient to proving new finite field incidence theorems. Chapter \ref{chapter:functionfield} will establish a sum-product theorem in function fields, a relatively unexplored setting for the problem. 

Expander functions are slightly different from sum-products. Rather than considering both the sumset and the product set, the concern is with the cardinality of a single image set $f(A)$ of some multivariate function $f$. A medley of new results on expander functions are proved in Chapter \ref{chapter:expanders}.

All of the mathematics in this chapter is drawn from the literature.

\section{Sum-product estimates}

Let $A$ be a finite subset of a field $F$, and consider the cardinalities of the \textbf{sumset}
	$$A+A=\left\{a+b:a,b \in A\right\}$$
and the \textbf{product set}
	$$AA=\left\{ab:a,b \in A\right\}.$$
	
It is clear that both these sets have at least the same cardinality as $A$. And it is certainly possible to have $|A+A|\approx |A|$ or $|AA|\approx |A|$, for example if $A$ is either an arithmetic or geometric progression. But the idea of sum-product estimates is that it is \textit{not possible to have both of these at once}. Results are therefore of the form 
	$$\max\left\{|A+A|,|AA|\right\}\gg |A|^{1+\delta}$$
with $\delta>0$, holding for any finite set $A \subseteq F$ that possibly satisfies some nondegeneracy conditions. Larger values of $\delta$ correspond to stronger results.

Sum-products were first studied by Erd\"os and Szemer\'edi \cite{erdosszemeredi}, who proved the existence of an unquantified $\delta>0$ holding for any finite set $A$ of integers. They conjectured that $\delta \geq 1-o(1)$. It is not possible to remove the $o(1)$ from this conjecture, as demonstrated by the case where $A$ is the set of integers between $1$ and $N$; in this case $|A+A|\approx |A|$ but $|AA|\approx |A|^{2-o(1)}$.   

Sum-product estimates can of course be considered in the more general setting of rings instead of fields, and indeed Erd\"os and Szemer\'edi's initial work concerned the ring of integers. However, when working over general rings the sum-product phenomenon is constrained by the possibility of zero divisors. If $A$ contains too many then the sumset and the product set might both be small. The motivation for sticking with fields in this thesis is therefore that avoiding zero-divisor considerations means working with an integrel domain, and when given an intergral domain one may as well work with its field of fractions. Tao \cite{taorings} has recently obtained a general, non-explicit, sum-product result for rings in the absence of too many zero divisors. 

\subsection{Real and complex sum-products}

The strongest known sum-product estimate is $\delta \geq \frac{1}{3}-o(1)$, which was proved for any finite set of reals by Solymosi \cite{solymosi1} and recently generalised to any finite set of complex numbers by Konyagin and Rudnev \cite{konyaginrudnev}. There is a fairly extensive list \cite{elekes,ford,nathanson,solymosi3,solymosi2} of previous results in the real and complex settings.

The incidence geometry of Chapter \ref{chapter:introincidences} yields a lot of information about sum-products via the Szemer\'edi-Trotter theorem. An example, not quite as strong as the more recent Solymosi and Konyagin-Rudnev bounds, is the following result of Elekes \cite{elekes}. A more involved version of this approach can be found in another paper of Solymosi \cite{solymosi2}.

\begin{theorem}[Elekes]\label{theorem:elekes} 
Let $A \subseteq \mathbb{C}$ be finite. Then 
	$$\max\left\{|A+A|,|AA|\right\}\gg |A|^{1+\frac{1}{4}}.$$
\end{theorem}

\begin{proof}
Define a point set $P$ and a line set $L$ by 
	\begin{align*}
	P&=(A+A)\times (AA)\\
	L&=\left\{l_{ab}:a,b\in A\right\} 
	\end{align*}
where $l_{ab}$ is the line given by $y=a(x-b)$. It is clear that
	\begin{align*}
	|P|&=|A+A||AA|\\
	|L|&=|A|^2.
	\end{align*} 
Each line $l_{ab} \in L$ is incident to at least $|A|$ points in $P$ since for each $c \in A$ the point $(b+c,ac)$ is an element of $P \cap l_{ab}$. Thus
	$$I(P,L)\geq |L||A|=|A|^3.$$ 
Applying the the complex Szemer\'edi-Trotter theorem (Theorem \ref{theorem:toth}) shows that
	\begin{align*}
	I(P,L) &\ll |P|^{2/3}|L|^{2/3}+|P|+|L|\\
	&\approx |A+A|^{2/3}|AA|^{2/3}|A|^{4/3}+|A+A||AA|+|A|^2.
	\end{align*}
Comparing the upper and lower bounds on $I(P,L)$ yields
	$$|A|^3 \ll |A+A|^{2/3}|AA|^{2/3}|A|^{4/3}+|A+A||AA|.$$
So either 	
	$$|A|^3 \ll |A+A|^{2/3}|AA|^{2/3}|A|^{4/3}$$
or
	$$|A|^3\ll |A+A||AA|.$$
In the event of the former, rearranging gives
	$$\max\left\{|A+A|,|AA|\right\}\gg |A|^{5/4}.$$ 
On the other hand, if the latter holds then 
	$$\max\left\{|A+A|,|AA|\right\}\gg |A|^{3/2} \geq |A|^{5/4}.$$ Either way the proof is complete.
\end{proof}

\subsection{Finite field sum-products}

Sum-product results for finite fields must be prefaced with a nondegeneracy condition, to prevent the set $A$ being too close to a copy of a subfield. For example if $A$ is a field then $\max\left\{|A+A|,|AA|\right\}=|A|$ and so non-trivial estimates are impossible. To keep this issue as straightforward as possible it is often convenient to concentrate  on the finite field $\mathbb{F}_p$ of prime order $p$, where it collapses into the requirement for an upper bound on $|A|$ in terms of $p$.

Let's now see what estimates are known in $\mathbb{F}_p$. The breakthrough result was due to Bourgain, Katz and Tao \cite{BKT} who proved the existence of an absolute $\delta>0$ dependent on $\gamma>0$ that holds whenever $|A|<p^{1-\gamma}$.

Garaev \cite{garaev2} made the Bourgain-Katz-Tao result explicit, obtaining $\delta \geq \frac{1}{14}-o(1)$ whenever $|A|<p^{1/2}$. This has been subsequently improved by a variety of authors \cite{BG, KS, li, shen, rudnev}. The most recent estimate is $\delta>\frac{1}{11}-o(1)$ whenever $|A|<p^{1/2}$, due to Rudnev \cite{rudnev}. Li and Roche-Newton \cite{LiORN1} built on a technique of Katz and Shen \cite{KSgeneral} to extend this estimate to any finite field, not necessarily of prime order, so long as $A$ is not too close to being a subfield. 

Chapter \ref{chapter:expanders} will use a particular form of Rudnev's $\frac{1}{11}-o(1)$ result. Its statement involves the multiplicative energy $E_{\times}(A)$ of $A$, which is the number of solutions to $ab=cd$ with $a,b,c,d \in A$; this is developed more fully in Chapter \ref{chapter:introapproxgroup}. As Rudnev remarks in his paper, the result is really an upper bound on $E_{\times}(A)$ in terms of the sumset $A+A$, and works just as well when this is replaced with the difference set $A-A$. In these terms, the estimate can be formulated as follows.

\begin{theorem}[Rudnev]\label{theorem:mishasumprodvariation}
Let $A \subseteq \mathbb{F}_p$. If $|A|\ll p^{1/2}$ then
	$$E_{\times}(A)^4\lesssim |A-A|^7 |A|^4.$$
\end{theorem}

The estimate $\delta \geq \frac{1}{11}-o(1)$ follows from the fact (again, see Chapter \ref{chapter:introapproxgroup}) that $E_{\times}(A)$ is greater than or equal to $\frac{|A|^4}{|AA|}.$

The proofs of all the above finite field results, including Theorem \ref{theorem:mishasumprodvariation}, take the same overall approach, sketched below.

\begin{proof}[Sketch of finite field sum-product proofs]
Given a set $A \subseteq \mathbb{F}_p$ consider the set 
	$$R(A)=\left\{\frac{a-b}{c-d}:a,b,c,d \in A, a \neq b, c \neq d\right\}.$$
Note that if $\xi \notin R(A)$ then 
	$$|A+\xi A|=|A|^2$$
since different pairs of elements from $A$ give different elements of $A+\xi A$. Indeed if there were a solution to  to 
	$$a+ \xi b = c + \xi d.$$
with $(a,b)\neq (c,d)$then rearranging would yield the contradiction
	$$\xi = \frac{a-c}{b-d}\in R(A).$$
	
Now, assume that $R(A) \neq \mathbb{F}_p$. This is immediate if $|A|<p^{1/4}$, and with a little effort we may assume something very similar so long as $|A|\ll p^{1/2}$. Since $\mathbb{F}_p$ contains no non-trivial additive subgroups, there must exist 
	$$\frac{a-b}{c-d}\in R(A)$$ 
such that 
	$$\frac{a-b}{c-d}+1\notin R(A).$$
Therefore
	\begin{align*}
	|A|^2 &= \left|A+ \left(\frac{a-b}{c-d}+1\right)A\right|\\
	&=\left|(c-d)A+(a-b+c-d)A\right|\\
	& \leq \left|AA-AA+AA-AA+AA-AA\right|. 
	\end{align*}
Thus an iterated sumset of the product set $AA$ must be large. With some work, and analysing multiplicative energy rather than the product set directly, it turns out that the only way this can happen is if either $AA$ is large itself or $A+A$ is large.	
\end{proof}

As a final remark for this section, it is worth thinking about what a reasonable sum-product conjecture might be for finite fields. Since the conjecture in the real case is $\delta \geq 1-o(1)$, it might be natural to suppose that this should hold over $\mathbb{F}_p$ whenever $|A|<p^{1/2}$. After all, this condition ensures there is no danger of the sumset or product set filling up the whole field. However the following result of Bourgain \cite{bourgain}, refined by Garaev \cite{garaev2}, shows that under this constraint the best that can be hoped for is $\delta \geq \frac{1}{2}$. To obtain $\delta \geq 1-o(1)$ will require $|A|\lesssim p^{1/3}$ as a necessary condition.

\begin{lemma}[Bourgain, Garaev]\label{theorem:bgexample}
For any prime $p$ and any integer $1 \leq N \leq p$ there is a set $A \subseteq \mathbb{F}_p$ with $|A|\approx N$ such that 		
	$$\max\left\{|A+A|,|AA|\right\}\ll N^{1/2}p^{1/2}.$$ 
In particular there is a set $A \subseteq \mathbb{F}_p$ with $|A|\approx p^{1/2}$ such that
	$$\max\left\{|A+A|,|AA|\right\}\ll |A|^{1+\frac{1}{2}}.$$ 
\end{lemma}

\begin{proof}
Let $g \in \mathbb{F}_p^*$ be a generator of $\mathbb{F}_p^*$. Let $M=p^{1/2}N^{1/2}$. By rearranging orders of summation it is immediate that
	\begin{align*}
	\sum_{y \in \mathbb{F}_p}\#\left(\left\{g^n:1 \leq n \leq M\right\}\cap \left\{y+j:1 \leq j \leq M\right\} \right)&\approx M^2.
	\end{align*}
Therefore there exists $y \in \mathbb{F}_p$ such that
	$$\#\left(\left\{g^n:1 \leq n \leq M\right\}\cap \left\{y+j:1 \leq j \leq M\right\} \right)\gg \frac{M^2}{p}=N.$$
It is therefore possible to pick a set $A$ contained in the set on the left hand side such that $|A| \approx N$. Since $A\subseteq\left\{g^n:1 \leq n \leq M\right\}$, 
	$$|AA| \ll M.$$
And since $A\subseteq \left\{y+j:1 \leq j \leq M\right\}$,  
	$$|A+A|\ll M.$$
So altogether,
	$$\max\left\{|A+A|,|AA|\right\}\ll M= N^{1/2}p^{1/2}$$
as required.
\end{proof}

As explained in the prologue, we are not too worried in this thesis about the `large-set' finite field paradigm, where sets are subject to a \textit{minimum} as well as a maximum density requirement. However it is worth noting that under this regime, with $|A|>p^{2/3}$, Garaev \cite{garaev2} obtained the bound
	$$\max\left\{|A+A|,|AA|\right\}\gg N^{1/2}p^{1/2}$$
which is optimal in view of Lemma \ref{theorem:bgexample}.  

\section{Expander functions}

Expander functions are another example of growth, and a natural variation on sum-products. An \textbf{$n$-variable expander function} is a function $f:F^n \to F$ for which the set 
	$$f(A)=\left\{f(a_1,\ldots,a_n):a_i \in A\right\}$$
has cardinality at least $\Omega\left(|A|^{1+\delta}\right)$ for some $\delta>0$ and any $A \subseteq F$, again possibly satisfying some nondegeneracy conditions. Stronger expander results correspond not only to a larger growth exponent $\delta$ as per sum-product estimates, but also to a smaller number $n$ of variables.

Sum-product estimates give a trivial method of constructing four-variable expander functions, because it is immediate that
	$$\left|\left\{ab+cd:a,b,c,d \in A\right\}\right|\geq \max\left\{|AA|,|A+A|\right\}.$$
However it is often possible to do better, both in terms of fewer variables and bigger exponents. 

In the two-variable case, Elekes, Nathanson and Ruzsa \cite{ENR} obtained $\delta\geq \frac{1}{4}$ for the function
	$$a+\frac{1}{b}$$ 
whenever $A$ is a finite set of reals. This was subsequently extended to $\delta \geq \frac{5}{19}$ by Li and Roche-Newton \cite{LiORN2}, and was also considered by Bourgain \cite{bourgain} in the prime finite field setting for an absolute $\delta>0$. The same paper of Bourgain also gives the example of the function 
	$$f(a,b)=a^2+ab$$
which has the peculiar property of not even requiring $a$ and $b$ to be drawn from the same set in order to establish a growth result. Hart, Li and Shen \cite{hartlishen} obtained $\delta \geq \frac{1}{146}$ over finite fields for the function
	$$a+b^2.$$
However, the strongest-known two-variable expander in the real, complex and finite field setting is the function
	$$a+ab.$$
This function was first studied by Garaev and Shen \cite{GS}, and a new result is established in Chapter \ref{chapter:expanders}.

In the four-variable case, an expander follows from the recent breakthrough of Guth and Katz \cite{GK} on the Erd\"os distinct distance problem. They showed that a finite point-set $P \subseteq \mathbb{R}^2$ determines at least $\Omega\left(|P|^{1-o(1)}\right)$ distinct pairwise distances. In the particular case $P=A \times A$ this shows that 	
	$$(a-b)^2+(c-d)^2$$
is a four-variable expander over $\mathbb{R}$ with $\delta>1-o(1)$. This methodology was also adapted by Iosevich, Roche-Newton and Rudnev \cite{IRR} to show that 
	$$ad- bc$$ 
is likewise a four-variable expander with $\delta = 1-o(1)$. New, stronger, results on three and four-variable expanders over $\mathbb{R}$ and $\mathbb{C}$ are proved in Chapter \ref{chapter:expanders}.

%% file: introapproxgroup.tex
\chapter{A handbook on sumset calculus}
\label{chapter:introapproxgroup}

This chapter comprises a useful arsenal of technical results for analysing sumsets and product sets. Everything is stated in terms of addition and sumsets, but extends without complication to multiplication and product sets.

The basic tools of sumset calculus are the Pl\"unnecke-Ruzsa inequalities, which relate the cardinalities of different sumsets. These will be of use in Chapters \ref{chapter:incidences}, \ref{chapter:expanders} and \ref{chapter:functionfield}.

Partial sumsets are a generalisation of sumsets, where interest is restricted to sums determined by only a subset of possible pairs. Balog-Szemer\'edi-Gowers type results reduce the study of partial sumsets to that of complete sumsets, but at the price of reduced efficiency. Two such results will come in handy in Chapters \ref{chapter:incidences} and \ref{chapter:expanders}.

The additive and multiplicative energy of a set were briefly touched on in Chapter \ref{chapter:introgrowth}, but the treatment here is more developed. Energy is useful for two reasons. First, sets with small sumset have high energy, which is helpful to know when analysing their structure. Second, energy-based arguments go through just as easily for partial as complete sumsets, and so can in some cases be more efficient than a Balog-Szemer\'edi-Gowers approach. This is useful in Chapters \ref{chapter:incidences} and \ref{chapter:expanders}.

Energy considerations also have a particular application in covering methods. When considering complete sumsets, these can sometimes more efficient than the Pl\"unnecke-Ruzsa inequalities. And since they are energy-based, they apply equally well to complete and partial sumsets, and so in some cases are stronger than Balog-Szemer\'edi-Gowers type results. These will be useful in Chapter \ref{chapter:expanders}. 

Except for Lemma \ref{theorem:cover1} and Lemma \ref{theorem:cover2}, which are original variations on covering lemmata, all of the mathematics in this chapter is drawn from the literature. However it should be remarked that the very useful Lemma \ref{theorem:densebsg} does not seem to be widely known or used in the research community. 

\section{Introduction}

Chapter \ref{chapter:introgrowth} defined the sumset $A+A$ and product set $AA$ of a subset $A$ of a field. This extends naturally to the \textbf{sumset} $A+B$ and \textbf{product set} $AB$ of two different sets $A,B$, defined as 
	\begin{align*}
	A+B&=\left\{a+b:a \in A, b \in B\right\}\\
	AB&=\left\{ab:a \in A, b \in B\right\}
	\end{align*}
Define also the \textbf{difference set} 
$$A-B=A+(-B)$$ 
and \textbf{ratio set} 
$$A/B=A\left(B \setminus \left\{0\right\}\right)^{-1}.$$ 

The essential problem of sumset calculus is that we are given a pair of sets $A,B$ and another related pair $C,D$. Given information about the sum, difference, product or ratio set of $A$ and $B$, what can be said about that of $C$ and $D$? 

The results in this chapter are general, in the sense that they depend only on the fact that sets are finite and contained in an abelian group. So they are quoted and proved for sumsets $A+B$, taking $+$ to be an abelian group operation rather than necessarily the additive group of a field. They therefore extend to multiplication and product sets without complication. In some instances it will be convenient to adopt slightly different notation in the multiplicative case to avoid confusion in later chapters; this will be indicated where appropriate.

\section{The Pl\"unnecke-Ruzsa inequalities}\label{section:sumsetcalculus}

Information on sumsets is often expressed in terms of cardinalities. There are trivial estimates
	$$\max{\left(|A|,|B|\right)}\leq \left|A+B\right|\leq |A||B|.$$
Informally, a sumset set is `small' if its cardinality is close to the lower bound $\max{\left(|A|,|B|\right)}$ and `large' if it is close to the upper bound $|A||B|$. The `efficiency' of an estimate corresponds to the number of factors of the form $|A+B|$ which turn up; the fewer the better since this allows the proof of stronger theorems. 

The most basic tool available for analysing sumset cardinalities is the Ruzsa triangle inequality, which applies to difference sets but not immediately to sumsets.

\begin{lemma}[Ruzsa triangle inequality]\label{theorem:ruzsa}
For sets $A,B,C$ we have
$$|A-B|\leq \frac{|A-C||B-C|}{|C|}.$$
\end{lemma}

\begin{proof}
For each element $x \in A-B$, pick $a_x \in A$ and $b_x \in B$ such that 
	$$x=a_x-b_x.$$
Consider the map $f$ from $(A-B) \times C$ to $(A-C) \times (B-C)$ given by $$f(x,c)=(a_x-c,b_x-c).$$
We show that $f$ is an injection. Suppose that $f(x,c)=f(y,d)$. Then
	\begin{align}
	a_x-c&=a_y-d \label{eq:ruzsa1}\\
	b_x-c&=b_y-d \label{eq:ruzsa2}
	\end{align}
Subtracting (\ref{eq:ruzsa2}) from (\ref{eq:ruzsa1}) gives $a_x-b_x=a_y-b_y$ and so $x=y$, which in turn forces $a_x=a_y$ and $b_x=b_y$. Plugging this back into (\ref{eq:ruzsa1}) shows that $c=d$, and so $(x,c)=(y,d).$ Thus $f$ is an injection and 	
	$$|A-B||C| \leq |A-C||B-C|.$$ 
Rearranging gives the required inequality.
\end{proof}

Note that the essential observation used in the above proof is that
	$$a-b=(a-c)-(b-c)$$
for any $a,b,c$. Variations on this approach are used several more times in the thesis, and are referred to as `Ruzsa-type'.

The Ruzsa triangle inequality is a simple and useful tool, but it applies only to \textit{difference} sets and it is often necessary to care about general \textit{sum}sets. This can be mitigated by clever applications, for example in the case $A=B$ taking $C=-A$ gives
	$$|A-A|\leq \frac{|A-(-A)||(-A)-A|}{|A|}=\frac{|A+A|^2}{|A|} $$ 
but this is not always terribly efficient. Moreover the Ruzsa triangle inequality does not help very much when analysing \textbf{iterated sumsets} of the form
	$$kA = \underbrace{A+\ldots+A}_k.$$ 

Such situations call for Pl\"unnecke's theorem, which was promulgated in this context by Ruzsa \cite{plunnecke}. A much shorter proof was recently given by Petridis \cite{petridis}, for which we follow an exposition given by Gowers \cite{gowerspetridis}.

\begin{lemma}[Pl\"unnecke]\label{theorem:pl}
For sets $A,B$ there exists $A' \subseteq A$ such that for any natural number $k$,
	\begin{equation}\label{eq:semicleverpl}
	\left|A'+kB\right|\leq \frac{|A'||A+B|^k}{|A|^k}.
	\end{equation}
In particular,
	\begin{equation}\label{eq:stupidpl}
	|kB|\leq \frac{\left|A+B\right|^k}{|A|^{k-1}}.
	\end{equation}
\end{lemma}

\begin{proof}
Pick $A' \subseteq A$ such that 
	$$\frac{|A'+B|}{|A'|}=K$$
is minimal across all subsets of $A$. It suffices to prove that
	\begin{equation}\label{eq:plsuffice1}
	|A'+kB|\leq |A'| K^k
	\end{equation}
since by minimality in the choice of $K$,
	$$K \leq \frac{|A+B|}{|A|}.$$ 
To establish (\ref{eq:plsuffice1}) it suffices in turn to show that
	\begin{equation}\label{eq:plsuffice2}
	|A'+B+C|\leq K|A'+C|
	\end{equation}
for any set $C$. Indeed, once \eqref{eq:plsuffice2} is established, inequality (\ref{eq:plsuffice1}) follows by induction on $k$, since taking $C=(k-1)B$ gives
	$$|A'+kB| \leq K |A'+(k-1)B| $$
and by inductive hypothesis
	$$|A'+(k-1)B|\leq K^{k-1}|A'|.$$

Now to prove (\ref{eq:plsuffice2}). Induction is used here as well, this time on $|C|$. When $|C|=1$ it is immediate that $|D+C|=|D|$ for any set $D$. In particular, 
	$$|A'+B+C|=|A'+B|$$
and 
	$$K|A'+C|=K|A'|$$ 
so \eqref{eq:plsuffice2} holds with equality by definition of $K$. Now suppose that $|C|\geq 2$ and write 
	$$C=C' \cup \left\{x\right\}$$
so that by inductive hypothesis,
	$$|A'+B+C'|\leq K|A'+C'|.$$
By inclusion-exclusion,
	\begin{align}
	|A'+B+C|&=|A'+B+(C' \cup \left\{x\right\})|\nonumber\\
	&=|A'+B+C'|+|\left(A'+B+\left\{x\right\}\right)\backslash (A'+B+C')|\nonumber\\
	&=|A'+B+C'|+|A'+B+\left\{x\right\}|\nonumber\\
	&\qquad -|\left(A'+B+\left\{x\right\}\right)\cap (A'+B+C')|\nonumber\\
	&\leq|A'+B+C'|+|A'+B|-|\left(A'+\left\{x\right\}\right)\cap (A'+C')+B|\nonumber\\
	&= |A'+B+C'|+|A'+B|-|A'\cap \left(A'+C'-\left\{x\right\}\right)+B|\nonumber\\
	&\leq K \left(\left|A'+C'\right|+\left|A'\right| - \left|A'\cap \left(A'+C'-\left\{x\right\}\right)\right|\right)\label{eq:pl1}
	\end{align}
On the other hand,
	\begin{align}
	|A'+C|&=\left|A'+C'\right|+\left|(A'+\left\{x\right\})\backslash (A'+C')\right|\nonumber\\
	&=\left|A'+C'\right|+\left|A'+\left\{x\right\}\right|-\left|(A'+\left\{x\right\})\cap (A'+C')\right|\nonumber\\
	&=|A'+C'|+|A'|-\left|A'\cap \left(A'+C'-\left\{x\right\}\right)\right|\label{eq:pl2}
	\end{align}
and so subsituting (\ref{eq:pl2}) into (\ref{eq:pl1}) gives 
	$$|A'+B+C|\leq K|A'+C|$$ 
as required.
\end{proof}

Lemma \ref{theorem:pl} is typically applied in the form (\ref{eq:stupidpl}) since this omits mention of the subset $A' \subseteq A$. However Katz and Shen \cite{KS} observed that with a little work it is possible to show that $|A'|\approx |A|$, which can in some circumstances make (\ref{eq:semicleverpl}) a better bet. 

\begin{corollary}[Katz, Shen]\label{theorem:cleverpl}
For sets $A,B$, there exists $A' \subseteq A$ with $|A'| \geq \frac{|A|}{2}$ such that 
	$$\left|A'+kB\right|\ll_k \frac{\left|A+B\right|^k}{|A|^{k-1}}.$$
\end{corollary}

In situations where nothing is lost by passing to a constant-proportion subset of $A$, Corollary \ref{theorem:cleverpl} effectively implies that	
	$$|kA|\ll_k\frac{|A+A|^{k-1}}{|A|^{k-2}},$$
which is better than Lemma \ref{theorem:pl} by a factor of $\frac{|A+A|}{|A|}$. 

\begin{proof}[Proof of Corollary \ref{theorem:cleverpl}]
Let $A_*$ be any subset of $A$ with $|A_*|\geq \frac{|A|}{2}$. By Lemma \ref{theorem:pl} there exists $A_*' \subseteq A_*$ such that
	$$|A_*'+kB|\leq \frac{|A_*'||A_*+B|^k}{|A_*|^k}\ll_k \frac{|A_*'||A+B|^k}{|A|^k}.$$

Apply the above observation recursively. Begin by taking $A_*=A$ to find $A_1 \subseteq A$ such that
	$$|A_1+kB|\ll_k \frac{|A_1|\left|A+B\right|^k}{|A|^{k}}.$$
Now take $A_*=A \setminus A_1$ to find $A_2 \subseteq A$ disjoint from $A_1$ such that
	$$|A_2+kB|\ll_k \frac{|A_2|\left|A+B\right|^k}{|A|^{k}}.$$
Repeat this process until reaching $n$ such that $\bigcup_{i=1}^n A_i$ is of cardinality at least $\frac{|A|}{2}$. Let 
	$$A'=\bigcup_{i=1}^n A_i.$$
Since all of the $A_i$ are disjoint it follows that $|A'|\geq \frac{|A|}{2}$ and
	\begin{align*}
	|A'+kB|&\leq \sum_{i=1}^n |A_i+kB|\\
	&\ll_k \frac{|A+B|^k}{|A|^k} \sum_{i=1}^n|A_i|\\
	& \approx \frac{|A+B|^k}{|A|^{k-1}} 
	\end{align*}
as required.
\end{proof}

\section{Partial sumsets}
For sets $A,B$ and $G \subseteq A \times B$, call the set 
	$$A \overset{G}+ B=\left\{a+b:(a,b) \in G\right\}$$
a \textbf{partial sumset}. In the case where $G= A \times B$ this collapses into the sumset $A+B$, referred to as the \textbf{complete sumset} where there is a possibility for confusion. The motivation for using the letter $G$ is that $G \subseteq A \times B$ corresponds to (the edges of) a bipartite graph connecting $A$ and $B$. 

Analysing partial sumsets is critical to the research in Chapters \ref{chapter:incidences} and \ref{chapter:expanders}, but the Pl\"unnecke-Ruzsa inequalities developed above apply only to \textit{complete} sumsets and so are not immediately useful. Fortunately, it turns out that if $G$ is a reasonably large part of $A \times B$ and the \textit{partial} sumset $A\overset{G}+B$ is small, then there are large subsets of $A$ and $B$ whose \textit{complete} sumset is small as well. Results of this kind are called \textbf{Balog-Szemer\'edi-Gowers type} or \textbf{BSG-type} after the standard theorem in this area, due to Balog and Szemer\'edi and strengthened by Gowers\footnote{See Theorem 2.29 of \cite{TV} for a formulation and proof of the Balog-Szemer\'edi-Gowers theorem. Note however that the factor $K^4$ in equation (2.20) of that formulation should be replaced with $K^5$ due to an error in the text.}.

Despite their usefulness, Balog-Szemer\'edi-Gowers type results are disproportionately costly. This is in the sense that they yield upper bounds on complete sumsets that are out of proportion to the density of $G$ in $A \times B$, and in particular they pick up more than a constant factor when $|G|\approx |A||B|$. Sections \ref{section:energy} and \ref{section:covermethods} describe methods which do not see this distinction and so can in some circumstances be more efficient.

This thesis uses two BSG-type results. The first result is more efficient but is applicable only when $G$ is especially dense in $A \times B$. It is a consequence of Exercise 2.5.4 in \cite{TV}.

\begin{lemma}[BSG-type for dense sets]\label{theorem:densebsg}
Let $0<\epsilon<1/4$ and let $G \subseteq A \times B$ with $|G| \geq (1-\epsilon)|A||B|$. Then there exists $A' \subseteq A$ with $|A'|\geq (1-\sqrt{\epsilon})|A|$ such that 
	$$|A'-A'|\ll_{\epsilon}\frac{|A \overset{G}-B|^2}{|A|}.$$
\end{lemma}

Lemma \ref{theorem:densebsg} does not seem to be widely used in the literature. But, where it is applicable, it is remarkably efficient for a BSG-type result.

The second BSG-type result, due to Bourgain and Garaev \cite{BG}, is less efficient but is on the other hand applicable in a wider variety of situations.

\begin{lemma}[BSG-type for less-dense sets ]\label{theorem:bg}
For sets $A,B$ and $G \subseteq A \times B$ there exists $A' \subseteq A$ with $|A'|\gg \frac{|G|}{|B|}$ such that 
	$$|A'-A'|\ll \frac{|A|^4|B|^3|A \overset{G}-B|^4}{|G|^5}.$$
\end{lemma}

There is considerable overlap between the proofs of Lemma \ref{theorem:densebsg} and Lemma \ref{theorem:bg}. For both it is convenient to adopt the following definitions.

For $a \in A$, write $B_G(a)$ for the set of $b \in B$ such that $(a,b)\in G$. Call the cardinality of $B_G(a)$ the \textbf{$G$-degree} of $a$, and the cardinality of the intersection $B_G(a_1)\cap B_G(a_2)$ the \textbf{joint $G$-degree} of $a_1$ and $a_2$. 

For both results it suffices to find a large subset of $A$ with the property that the joint $G$-degree of any two distinct elements is large. This is demonstrated by the following Ruzsa-type lemma.

\begin{lemma}\label{theorem:bsgsufficient}

Let $G \subseteq A \times B$. Suppose that $A' \subseteq A$ and $H \subseteq A' \times A'$ are such that every the joint $G$-degree of any $(a_1,a_2)\in H$ is at least $K$. Then 
	$$|A'\overset{H}-A'|\leq \frac{|A\overset{G}-B|^2}{K}.$$
In particular, if the joint $G$-degree of any  pair of elements from $A'$ is at least $K$ then 
	$$|A'-A'|\leq \frac{|A\overset{G}-B|^2}{K}.$$
\end{lemma}

\begin{proof}
For each $x \in A'\overset{H}-A'$ pick $a_x^1,a_x^2 \in A'$ such that 
	$$a_x^1-a_x^2=x.$$ 
Let $Y \subseteq (A' \overset{H}-A')\times B$ be given by
	$$Y= \left\{(x,b): b \in B_G(a_x^1)\cap B_G(a_x^2) \right\}. $$
By hypothesis, 
	$$|Y|\geq |A'\overset{H}-A'|K.$$ 
On the other hand, the injection
	$$(x,b)\mapsto \left(a_x^1-b,a_x^2-b\right)$$
from $Y$ into $(A \overset{G}-B)\times (A \overset{G}-B)$ shows that 
	$$|Y|\leq |A\overset{G}-B|^2.$$
Comparing the upper and lower bounds on $|Y|$ gives the result. 
\end{proof}

The ease of finding sets satisfying the conditions of Lemma \ref{theorem:bsgsufficient} depends on the density of $G$, corresponding to the two different BSG-type results. Over the next few pages, Section \ref{section:highdensity} builds the proof of Lemma \ref{theorem:densebsg} and Section \ref{section:mediumdense} builds the proof of Lemma \ref{theorem:bg}.

\subsection{High-density partial sumsets}\label{section:highdensity}
This section proves Lemma \ref{theorem:densebsg}, the BSG-type result for high-density partial sumsets. In view of Lemma \ref{theorem:bsgsufficient} this is accomplished by the following.

\begin{lemma}\label{theorem:highdensityG}
Let $\epsilon>0$ and $G \subseteq A \times B$ with $|G|\geq (1-\epsilon)|A||B|$. There exists $A' \subseteq A$ with $|A'|\geq (1-\sqrt{\epsilon})|A|$ such that the $G$-degree of every $a \in A'$ is at least $(1-\sqrt{\epsilon})|B|$. 

In particular, any two elements of $A'$ have joint $G$-degree at least $(1-2\sqrt{\epsilon})|B|$ whenever $\epsilon<1/4$.
\end{lemma}

Applying Lemma \ref{theorem:bsgsufficient} to the conclusion of Lemma \ref{theorem:highdensityG} with $H=A' \times A'$ and $K=(1-2\sqrt{\epsilon})|B|$ immediately yields the statement of Lemma \ref{theorem:densebsg}.

\begin{proof}[Proof of Lemma \ref{theorem:highdensityG}]
Let $A'\subseteq A$ be the set of $a \in A$ with $G$-degree at least $(1-\sqrt{\epsilon})|B|$. It suffices to show that $|A'|\geq (1-\sqrt{\epsilon})|A|$. To this end, observe that
	\begin{align*}
	(1-\epsilon)|A||B|& \leq |G|\\
	&= \sum_{a \in A}|B_G(a)|\\
	&\leq \sum_{a \in A'}|B| + \sum_{a \in A \setminus A'}\left(1-\sqrt{\epsilon}\right)|B|\\
	&= |A'||B| + (|A|-|A'|)\left(1-\sqrt{\epsilon}\right)|B|.
	\end{align*}
It follows that
	$$(1-\epsilon)|A| \leq |A'|+(|A|-|A'|)\left(1-\sqrt{\epsilon}\right)$$
and so $|A'|\geq (1-\sqrt{\epsilon})$ as required. 
\end{proof}

\subsection{Less-dense partial sumsets}\label{section:mediumdense}

This section proves Lemma \ref{theorem:bg}, the BSG-type result for less-dense partial sumsets. The proof is not quite as direct as that for Lemma \ref{theorem:densebsg}, and the main additional ingredient is the following lemma.

\begin{lemma}\label{theorem:bsgrefine}
Let $G \subseteq A \times B$. Then for any $\epsilon >0$ there exists $A' \subseteq A$ with $|A'|\gg \frac{|G|}{|B|}$, and $H \subseteq A' \times A'$ with $|H|\geq(1-\epsilon)|A'|^2$, such that the joint $G$-degree of any $(a_1,a_2)\in H$ is at least $
\frac{\epsilon|G|^2}{2|A|^2|B|}$.
\end{lemma} 

\begin{proof}
Let $H$ be the set of $(a_1,a_2)\in A \times A$ with joint $G$-degree at least $
\frac{\epsilon|G|^2}{2|A|^2|B|}$. To prove the lemma it suffices to find $A' \subseteq A$ with $|A'|\gg \frac{|G|}{|B|}$ such that 
	$$|(A' \times A')\cap H|\geq (1-\epsilon)|A'|^2.$$ 
It is clear that
	\begin{equation*}
	\sum_{(a_1,a_2) \notin H}\left|B_G(a_1)\cap B_G(a_2)\right|< \frac{\epsilon|G|^2}{2|B|}
	\end{equation*}
and hence 
	\begin{equation}\label{eq:bsg2}
	\sum_{a_1,a_2 \in A}\left|B_G(a_1)\cap B_G(a_2)\right|\frac{\mathds{1}\left((a_1,a_2)\notin H\right)}{\epsilon}< 			
	\frac{|G|^2}{2|B|}.
	\end{equation}
Also, $\sum_{b \in B}|B_G(b)|=|G|$ and so by Cauchy-Schwarz 
	\begin{equation}\label{eq:bsg1}
	\sum_{a_1,a_2 \in A}\left|B_G(a_1)\cap B_G(a_2)\right|\geq \frac{|G|^2}{|B|}.
	\end{equation}
Comparing (\ref{eq:bsg2}) and (\ref{eq:bsg1}) shows that
	\begin{equation*}\label{eq:bsg3}
	\sum_{a_1,a_2 \in A} \left|B_G(a_1)\cap B_G(a_2)\right|\left(1-\frac{\mathds{1}\left((a_1,a_2)\notin 			
	H\right)}{\epsilon}\right)\geq \frac{|G|^2}{2|B|}.
	\end{equation*}
Rearranging gives
	\begin{equation*}
	\sum_{b \in B}\sum_{a_1,a_2 \in A_G(b)} \left(1-\frac{\mathds{1}\left((a_1,a_2)\notin H\right)}{\epsilon}\right)\geq 		
	\frac{|G|^2}{2|B|}.
	\end{equation*}
So there exists $b \in B$ such that 
	\begin{equation*}\label{eq:bsg4}
	\sum_{a_1,a_2 \in A_G(b)} \left(1-\frac{\mathds{1}\left((a_1,a_2)\notin H\right)}{\epsilon}\right)\geq \frac{|G|^2}{2|B|^2}.
	\end{equation*}
Evaluating the left hand side shows
	\begin{equation*}
	|A_G(b)|^2 - \frac{\left|(A_G(b) \times A_G(b)) \setminus H \right|}{\epsilon} \geq \frac{|G|^2}{2|B|^2}.
	\end{equation*}
Since the left hand side is at most $|A_G(b)|^2$ it follows that $|A_G(b)|\gg \frac{|G|}{|B|}$. Take $A'=A_G(b)$ so that 
	$$\left|(A' \times A') \setminus H \right|\leq \epsilon \left(|A'|^2 - \frac{|G|^2}{2|B|^2}\right)\leq \epsilon |A'|^2.$$
This implies $|(A' \times A')\cap H| \geq (1-\epsilon)|A'|^2$ as required.
\end{proof}

With this established, it is now possible to prove Lemma \ref{theorem:bg}. 

\begin{proof}[Proof of Lemma \ref{theorem:bg}]
Let $\epsilon>0$ be sufficiently small, and fixed. By Lemma \ref{theorem:bsgrefine} there exists $A' \subseteq A$ and $H \subseteq A' \times A'$ with $|A'|\gg \frac{|G|}{|B|}$ and $|H| \geq (1-\epsilon)|A'|^2$, such that any $(a_1,a_2)\in H$ have joint $G$-degree at least $\Omega\left(\frac{|G|^2}{|A|^2|B|}\right)$. Lemma \ref{theorem:bsgsufficient} implies
	$$|A'\overset{H}-A'|\ll \frac{|A\overset{G}-B|^2|A|^2|B|}{|G|^2}.$$
By Lemma \ref{theorem:highdensityG} there is a subset $A'' \subseteq A'$ with $|A''|\gg |A'|$ such that any two $a_1,a_2 \in A''$ have joint $H$-degree at least $(1-2\sqrt{\epsilon})|A'|$. A final application of Lemma \ref{theorem:bsgsufficient} shows that
	\begin{align*}
	|A''-A''|&\ll \frac{|A'\overset{H}-A'|^2}{|A'|}\ll \frac{|A|^4|B|^3 |A\overset{G}-B|^4}{|G|^5}
	\end{align*}
as required.
\end{proof}

\section{Additive energy}\label{section:energy}

Define the \textbf{additive energy} $E_+(A,B)$ of set $A$ and $B$ to be the number of solutions to the equation
	\begin{equation}\label{eq:mainenergy}
	a+b=a'+b'
	\end{equation}
with $a,a' \in A$ and $b,b' \in B$. Note that 
	$$E_+(A,B)=E_+(A,-B)$$ 
since (\ref{eq:mainenergy}) holds if and only if $a-b'=a'-b$. Define the additive energy of a single set $A$ to be 		
	$$E_+(A)=E_+(A,A).$$
	
When working with the multiplicative group of a field, we refer instead to \textbf{multiplicative energy}, denoted by $E_{\times}(A,B)$, i.e. the number of solutions to $a b=a'b'$ with $a,a' \in A$ and $b,b' \in B$. 

Energy and sumsets are closely related, since the only way in which $|A+B|$ can be small is if there are many pairs of elements giving the same sums. This would necessitate many solutions to \eqref{eq:mainenergy} and thus a large energy. This intuition will shortly be made rigorous, but in order to do so it is first necessary to generalise and slightly reformulate energy.

Just like generalising sumsets to partial sumsets, given $G \subseteq A \times B$ it is natural to consider the number of solutions to (\ref{eq:mainenergy}) with $(a,b),(a',b')\in G$. This restriction is called the additive energy of $G$ and denoted by $E_+(G)$. As with partial sumsets, this collapses to the definition of $E_+(A,B)$ when $G=A \times B$.

The following lemma shows that additive energy can be helpfully formulated as an $L^2$ norm.

\begin{lemma}[Reformulation of additive energy]\label{theorem:energyformulation}
For $G \subseteq A \times B$, let $\mu_G(x)$ denote the number of $(a,b)\in G$ with $a+b=x$. Then
	\begin{align}E_+(G)&=\sum_{x \in A\overset{G}+B} \mu_G(x)^2\label{eq:energy1}\\
	&=\sum_{(a,b)\in G} \mu_G(a+b).\label{eq:energy2}
	\end{align}
\end{lemma}

\begin{proof}
It is clear that
	\begin{align*}
	E_+(G)&=\sum_{x \in A\overset{G}+B}\#\left\{(a,b),(a',b')\in G:a+b=a'+b'=x\right\}\\
	&=\sum_{x \in A\overset{G}+B}\#\left\{(a,b)\in G:a+b=x\right\}^2\\
	&=\sum_{x \in A+B}\mu_G(x)^2
	\end{align*}
which establishes \eqref{eq:energy1}. Similarly,
	\begin{align*}
	E_+(G)&=\sum_{(a,b)\in G}\#\left\{(a',b'):a+b=a'+b'\right\}\\
	&=\sum_{(a,b)\in G} \mu_G(a+b)
	\end{align*}
which gives \eqref{eq:energy2}.
\end{proof}

When $G=A \times B$ we have $\mu_G(x)=|A \cap (x-B)|$ and the following corollary.

\begin{corollary}\label{theorem:energyformulation2}
For sets $A,B$ we have
	\begin{align}E_+(A,B)&=\sum_{x \in A+B} |A \cap(x-B)|^2 \label{eq:eg1}\\
	&=\sum_{a \in A, b \in B} |A \cap(a+b-B)|\label{eq:eg2}\\
	&=\sum_{a,a' \in A}|(B+a)\cap (B+a')|\label{eq:eg3}
	\end{align}
\end{corollary}

\begin{proof}
Equations \eqref{eq:eg1} and \eqref{eq:eg2} follow immediately from Lemma \ref{theorem:energyformulation} and the fact that $\mu_G(x)=|A \cap (x-B)|$ when $G=A \times B$. Equation \eqref{eq:eg3} follows by rearranging orders of summation from \eqref{eq:eg2}. 

\end{proof}

Now let's make rigorous the earlier discussion of the relationship between energy and sumsets. The following estimates are trivial.
	$$|A||B| \leq E_+(A,B)\leq |A|^2|B|.$$
Energy close to $|A||B|$ is `low' and that close to $|A|^2|B|$ is `high'. The following result shows that, as expected, small sumset implies high energy. This can be useful to know when analysing the structure of sets with small sumset.

\begin{lemma}\label{theorem:energycs1}
If $G \subseteq A\times B$ then
	$$E_+(A,B)\geq E_+(G) \geq \frac{|G|^2}{|A \overset{G}+B|}$$
\end{lemma}

Note that in this lemma, partial and complete sumsets are on the same footing: unlike BSG-type results there is no additional premium to pay for working with partial sumsets.

\begin{proof}[Proof of Lemma \ref{theorem:energycs1}]
It is clear that
	$$\sum_{x \in A\overset{G}+B}\mu_G(x)=|G|.$$
So by Cauchy-Schwarz,
	$$|G|^2 \leq |A\overset{G}+B|\sum_{x \in A\overset{G}+B}\mu_G(x)^2.$$
By Lemma \ref{theorem:energyformulation} the right-hand side is equal to $|A\overset{G}+B|E_+(G)$ and so
	$$E_+(A,B)\geq E_+(G) \geq \frac{|G|^2}{|A\overset{G}+B|}$$
which completes the proof.
\end{proof}

The converse to Lemma \ref{theorem:energycs1} is false: high energy does not automatically imply small sumset. For example if $A$ is the union of an arithmetic progression $A_1$ and a geometric progression $A_2$, each of cardinality $|A|/2$, then $A$ is high-energy since
	$$E_+(A)\gg E_+(A_1) \approx |A|^3.$$
However the sumset $|A+A|$ is also large since
	$$|A+A|\geq |A_2+A_2|\approx |A|^2.$$
	
What \textit{does} hold, however, is the following partial converse. If a set has high energy then there must exist a fairly large $G \subseteq A \times A$ for which the partial sumset $A \overset{G}+A$ is small. This fact will not be required in the thesis and so precise details are omitted.

\section{Covering methods}\label{section:covermethods}

A \textbf{covering lemma} is a result of the following form. If $A$ and $B$ have small sumset then a large part of $A$ can be `covered by' (contained in the union of) a small number of translates of $B$ or some modification of $B$. The canonical example is due to Ruzsa, for which a proof can be found in \cite{TV}:

\begin{lemma}[Ruzsa]
For any sets $A,B$, the set $A$ is contained in the union of $\frac{|A-B|}{|B|}$ translates of $B-B$ 
\end{lemma}

Another more recent example is due to Shen \cite{shen}. In a similar spirit to the improvement of Lemma \ref{theorem:cleverpl} over Lemma \ref{theorem:pl}, this enables a covering with translates of $B$ instead of $B-B$, at the expense of leaving a small part of $A$ uncovered.

\begin{lemma}[Shen]\label{theorem:shencover}
For sets $A,B$ and $\epsilon>0$ there exists $A' \subseteq A$ with $|A'|\geq (1-\epsilon)|A|$ such that $A'$ is contained in the union of $O_{\epsilon}\left(\frac{|A-B|}{|B|}\right)$ translates of $B$.
\end{lemma}

The proof of Shen's covering result is based on additive energy, and so with some tweaks applies to the situation where only a dense partial sumset, rather than a complete sumset, is small. In certain situations, such as in Chapter \ref{chapter:expanders}, this can be a more-efficient substitute for Balog-Szemer\'edi-Gowers type methods, since there is no additional cost for dealing with partial rather than complete sumsets. 

Two original variations on this theme are used in Chapter \ref{chapter:expanders}. The first, below, essentially restates Shen's result with $A-B$ replaced by $A\overset{G}-B$.

\begin{lemma}[Shen variation 1]\label{theorem:cover1}
Let $G \subseteq A \times B$ and $0<\epsilon<1/4$. If $|G|\geq (1-\epsilon)|A||B|$ then there exists $A' \subseteq A$ with $|A'|\geq (1-2\sqrt{\epsilon}) |A|$ such that $A'$ is contained in the union of $O_{\epsilon}\left(\frac{|A\overset{G}-B|}{|B|}\right)$ translates of $B$. 

Similarly, there is a subset $A'' \subseteq A$ with $|A''|\geq (1-2\sqrt{\epsilon}) |A|$  such that $A''$ in contained in the union of $O_{\epsilon}\left(\frac{|A\overset{G}-B|}{|B|}\right)$ translates of $-B$.
\end{lemma}

\begin{proof}
We shall prove the case for covering with translates of $B$, and remark on the slight alteration needed to cover with translates of $-B$. 

Since $|G|\geq (1-\epsilon)|A||B|$ there is by Lemma \ref{theorem:highdensityG} a subset $A_1 \subseteq A$ with $|A_1|\geq (1-\sqrt{\epsilon}) |A|$ such that every element of $A_1$ has $G$-degree at least $(1-\sqrt{\epsilon}) |B|$. 

Now for any subset $A_* \subseteq A_1$ let $G_*=G \cap (A_* \times B)$ so that 
	$$|G_*| \geq (1-\sqrt{\epsilon})|A_*||B|.$$ 
By Lemma \ref{theorem:energycs1} it follows that
	$$E_+(A_*,B) \geq \frac{|G_*|^2}{|A\overset{G}-B|}\geq \frac{|A_*|^2(1-\sqrt{\epsilon})^2|B|^2}{|A\overset{G}-B|}.$$
By Corollary \ref{theorem:energyformulation2},
	\begin{equation}\label{eq:energydef}
	E_+(A_*,B)=E_+(A_*,-B)=\sum_{a \in A_*, b \in B}|A_* \cap (a-b)+B |
	\end{equation}
and so for any $A_* \subseteq A_1$ there exist $a \in A_*$, $b \in B$ such that 
	$$|A_* \cap (a-b+B) |\geq \frac{|A_*| (1-\sqrt{\epsilon})^2 |B|}{|A\overset{G}-B|}.$$
	
Apply the above discussion to a sequence of subsets of $A_1$. Begin by taking $A_*=A_1$ to find $a_1 \in A_1$, $b_1 \in B$ such that
	$$|A_1 \cap (a_1-b_1+B) |\geq \frac{|A_1| (1-\sqrt{\epsilon})^2 |B|}{|A\overset{G}-B|}.$$
The translate $(a_1-b_1)+B$ covers $\frac{|A_1| (1-\sqrt{\epsilon})^2 |B|}{|A\overset{G}-B|}$ elements of $A_1$. Discard $A_1 \cap (a_1-b_1)+B$ from $A_1$ and let $A_2$ be the set of elements remaining, now taking $A_*=A_2$ and repeating the process. 

Iterate $O_{\epsilon}\left(\frac{|A\overset{G}-B|}{|B|}\right)$ times until the set remaining is of cardinality no more than $\sqrt{\epsilon}|A_1|$. Then take $A'$ to be the set of elements discarded across all iterations, so that $|A'|\geq (1-\sqrt{\epsilon})|A_1|$. Since $|A_1|\geq (1-\sqrt{\epsilon})|A|$ we get $|A'|\geq (1-2\sqrt{\epsilon})|A|$ as required, which completes the proof for covering with translates of $B$.

The proof for covering with translates of $-B$ is identical, except that in place of (\ref{eq:energydef}), the identity
	$$E_+(A_*,B)=\sum_{a \in A_*, b \in B}|A_* \cap (a+b-B) |$$
is used instead.
\end{proof}

The second original variation on Shen's result has a slightly different formulation. Instead of covering a large part of $A$ with translates of $B$, it yields a large part of $G$ whose corresponding partial difference set is covered by few translates of $B$. 

\begin{lemma}[Shen variation 2] \label{theorem:cover2}
Let $G \subseteq A \times B$ and $0<\epsilon<1$. Then there exists $G' \subseteq G$ with $|G'|\geq (1- \epsilon)|G|$ such that $A\overset{G'}-B$ is contained in the union of $O_{\epsilon}\left(\frac{|A\overset{G}-B||A|}{|G|}\right)$ translates of $B$. 
\end{lemma}

\begin{proof}
Let $G_*$ be any subset of $G$. By Lemma \ref{theorem:energycs1} there are at least $\frac{|G_*|^2}{|A\overset{G}-B|}$ solutions to 
	$$a-b=a'-b'$$
with $(a,b),(a',b') \in G_*$. So there exists $a_* \in A$ for which there are at least $\frac{|G_*|^2}{|A\overset{G}-B||A|}$ pairs $(a,b) \in G_*$ with
	$$a-b \in a_* - B.$$
In other words, the translate $a_*-B$ accounts for the differences of $\frac{|G_*|^2}{|A\overset{G}-B||A|}$ pairs from $G_*$.

Apply the above discussion to a sequence of subsets of $G$. Begin by taking $G_*=G$ to find a translate of $B$ accounting for the differences of $\frac{(1-\epsilon)^2|G|^2}{|A\overset{G}-B||A|}$ pairs from $G$. Discard these pairs from $G_1$ and let $G_2$ be the set of pairs remaining, now taking $G_*=G_2$ and repeating the process.

Iterate $O_{\epsilon}\left(\frac{|A\overset{G}-B||A|}{|G|}\right)$ times until the subset of $G$ remaining is of cardinality no more than $\epsilon |G|$. Then take $G'$ to be the set of discarded pairs, so that $|G'|\geq (1-\epsilon)|G|$, and $A\overset{G'}-B$ is contained in no more than $O_{\epsilon}\left(\frac{|A\overset{G}-B||A|}{|G|}\right)$ translates of $B$, as required. 
\end{proof}

%% file: incidences.tex
\chapter{Incidence theorems over finite fields}
\label{chapter:incidences}

Chapter \ref{chapter:introincidences} introduced incidence geometry but deferred discussion of the finite field case. This chapter now picks up that baton.

The state of the art for incidences is weaker over finite fields than $\mathbb{R}$ and $\mathbb{C}$, and the discrepancy is much greater than the analogous gap for growth results. This chapter narrows that divide with two new theorems: an incidence bound and a line-counting result. 

An earlier version of this work has been submitted to the European Journal of Combinatorics. A preprint \cite{me806}, and a more recent update \cite{me662} also used here, are available on the arXiv.

\section{Results}

This section describes the two results proved in this chapter.

\subsection{Incidence bounds}
Chapter \ref{chapter:introincidences} showed that in the `critical' case $|P|,|L|=N$, non-trivial incidence bounds are of the form 
	$$I(P,L)\ll N^{3/2-\epsilon}$$
for $\epsilon>0$. The Szemer\'edi-Trotter theorem (Theorem \ref{theorem:ST}) gives the sharp result $\epsilon = \frac{1}{6}$ in $\mathbb{R}^2$ and generalises to the complex setting $\mathbb{C}^2$ (Theorem \ref{theorem:toth}).

Chapter \ref{chapter:introgrowth} noted that any nontrivial results on growth in finite fields must be predicated on keeping away from subfields. The same is true for incidence bounds, but this time one must keep away from subplanes. Just like growth, to keep these considerations as straightforward as possible it is often convenient to work with the field $\mathbb{F}_p$ of prime order $p$. 

To see why avoiding subplanes is necessary, consider the example where $P$ is the whole plane $\mathbb{F}_p^2$ and $L$ is the set of all lines in $\mathbb{F}_p^2$. Taking $N=p^2$ gives $|P|,|L| \approx N$, but since every line in $L$ is incident to $p$ points in $P$ it follows that $I(P,L)\approx N^{3/2}$ and so in this case a non-trivial estimate is impossible. 

Working over $\mathbb{F}_p^2$, Bourgain, Katz and Tao \cite{BKT} proved the existence of a non-trivial $\epsilon >0$, dependent on $\gamma>0$, whenever $N<p^{2-\gamma}$. This has been made explicit in two cases:
\begin{itemize}
\item In the `small-set' regime $N<p$, Helfgott and Rudnev \cite{HR} obtained a bound of $\epsilon \geq \frac{1}{10,678}$. 
\item In the `large-set' regime $p^{1+\gamma} \leq N \leq p^{2-\gamma}$, Vinh \cite{vinh} obtained $\epsilon \geq \frac{\gamma}{4}$.
\end{itemize}
For the purposes of this thesis, where as explained in the prologue the emphasis is on small rather than large sets, the Helfgott-Rudnev bound is the one to beat. 

It is instructive to compare with the situation for growth and sum-products. With sum-products there is a growth exponent of $\delta \geq \frac{1}{3}-o(1)$ over $\mathbb{R}$ and $\mathbb{C}$, and $\delta \geq \frac{1}{11}-o(1)$ over finite fields, which is a bit weaker but not excessively so. By contrast the finite field incidence bound of $\epsilon \geq \frac{1}{10,678}$ is an awfully long way from the real and complex bound of $\epsilon \geq \frac{1}{6}$ implied by Szemer\'edi-Trotter. 

The first new result in this chapter goes some way to redressing this disparity by obtaining a much stronger finite field incidence bound of $\epsilon \geq \frac{1}{662}-o(1)$.

\begin{theorem}\label{theorem:incidences}
Let $N<p$. If $P$ and $L$ are a set of points and lines over $\mathbb{F}_p$ with $|P|,|L| \leq N$ then 
	$$I(P,L)\lesssim N^{\frac{3}{2}-\frac{1}{662}}.$$
\end{theorem}

\subsection{Line counting}
Beck's theorem from Chapter \ref{chapter:introincidences} shows that any set of points in $\mathbb{R}^2$ satisfies at least one of two extremes. Either there are at least $\Omega(|P|)$ collinear points, or the set $L(P)$ of lines determined by pairs of points in $P$ is of cardinality at least $\Omega(|P|^2)$. 

As with incidence bounds, a nondegeneracy condition is required for Beck-type theorems over finite fields. For example if $P=\mathbb{F}_p^2$ then no more than $|P|^{1/2}$ points are collinear, and so the first possible conclusion of Beck's theorem cannot hold. But the second cannot hold either since $|L(P)| \approx |P|$. 

As with incidence bounds, there have been explicit finite field versions of Beck's theorem in two instances:
\begin{itemize}
\item In the particular case of the `small-set' regime $|P|<p$ for which $P=A \times A$ with $A \subseteq \mathbb{F}_p$, Helfgott and Rudnev \cite{HR} showed that $|L(P)| \gg |P|^{1+ \frac{1}{267}}$.
\item In the `large-set' regime, Iosevich, Rudnev and Zhai \cite{IRZ} recently showed that $|L(P)|\approx p^2$ whenever $|P|>p \log p$.
\end{itemize}
The `small-set' Helfgott-Rudnev result is again, for the purposes of this thesis, the one to beat. The second new result in this chapter does so in two respects. First, there is a stronger exponent of $\frac{1}{133}-o(1)$ in place of $\frac{1}{267}$. Second, the result holds for general $P \subseteq \mathbb{F}_p^2$ with $|P|<p$ rather than simply those of the form $P=A \times A$. 

\begin{theorem}\label{theorem:beck2}
If $P \subseteq \mathbb{F}_p^2$ and $|P|<p$ then at least one of the following must occur:
\begin{enumerate}
\item At least $\widetilde{\Omega}\left(|P|\right)$ points from $P$ are contained in a single line.
\item $|L(P)|\gtrsim|P|^{1+\frac{1}{133}}$. \end{enumerate}
\end{theorem}

\subsection{Structure}
The body of this chapter is concerned with the proofs of Theorems \ref{theorem:incidences} and \ref{theorem:beck2}. Both follow from new observations that efficiently relate incidences to sum-product estimates.

Unlike the incidence material in Chapter \ref{chapter:introincidences}, the analysis here makes crucial use of the properties of the projective plane. \textbf{Section \ref{section:interpretation}} uses this to show that the existence of a certain configuration of points and lines would imply the existence of sets $A,B$ and fairly dense $G \subseteq A \times B$ such that the difference set $A \overset{G}-B$ and the ratio set $A \overset{G}/B$ are both small. Such a configuration is called a `sum-product configuration'.

\textbf{Section \ref{section:refine}} goes on to show that if there are too many points incident to too many lines, which would arise if either Theorem \ref{theorem:incidences} or Theorem \ref{theorem:beck2} were to fail, then there must also exist a large sum-product configuration.

The situation is then almost ripe for applying a finite field sum-product theorem from Chapter \ref{chapter:introgrowth} to show that a sum-product configuration cannot arise, and so there cannot be too many incidences. The final hurdle is to relate the \textit{partial} difference and ratio sets arising from a sum-product configuration to \textit{complete} sets. This is accomplished in \textbf{Section \ref{section:bounding}} using a concoction of Balog-Szemer\'edi-Gowers type results from Chapter \ref{chapter:introapproxgroup}.  

\textbf{Section \ref{section:incidenceproofs}} then uses the analysis from the previous sections to prove Theorems \ref{theorem:incidences} and \ref{theorem:beck2}.

\textbf{Section \ref{section:incidencesconclusion}} gives pointers for further work.

\section{Foci and configurations}\label{section:interpretation}

This section uses the theory of projective transformations to interpret particular types of point-line configurations in terms of sum-product (or strictly speaking, difference-ratio) estimates. The key observation is that difference and ratio sets can be interpreted in terms of gradients, and that a projective transformation enables these to be viewed as point-line incidences. 

There are several pictures of points and lines in this section, apparently treating them as if in $\mathbb{R}^2$. These are for illustration only; they do not correspond rigorously to the plane $\mathbb{F}_p^2$, although the definitions they illustrate do.  
  
Let $F$ be any field. Recall, or quickly check Appendix \ref{chapter:proj}, that the projective plane $\mathbb{P}F^2$ is given by equivalence classes of $F^3 \setminus \left\{(0,0,0)\right\}$ modulo dilation. More concretely, it can be viewed as the union of $F^2$ with a
 `line at infinity' $l_{\infty}$. Points of $l_{\infty}$ correspond to gradients in $\mathbb{P}F^1=F \cup \left\{\infty \right\}$, so that for each $\lambda \in \mathbb{P}F^1$ there is a point $p_{\lambda} \in l_{\infty}$ that is incident to all lines of gradient $\lambda$.

Recall also some theory on \textbf{projective transformations} of $\mathbb{P}F^2$. These are line-preserving permutations of $\mathbb{P}F^2$. A \textbf{frame} in $\mathbb{P}F^2$ is a set of four distinct points, no three of which are collinear. For any two frames there is a unique projective transformation that maps one to the other.

Now for some new definitions. If $P$ is a set of points in $\mathbb{P}F^2$ and $p \notin P$ is an individual point then say that $p$ is a \textbf{$K$-focus} for $P$ if $P$ is supported over at most $K$ lines through $p$. This is illustrated in Figure \ref{fig:focus} where the red point is a 3-focus for the blue points. Note that if $p$ is a $K$-focus for $P$ then it is also a $K$-focus for any subset of $P$.

\begin{figure}[ht]
\vspace{20pt}
\centering
\begin{tikzpicture}[scale=7]
\coordinate (focus) at (0,0);
\coordinate (point1) at (0.32,0.95);
\coordinate (point2) at (0.71,0.71);
\coordinate (point3) at (0.95,0.32);
\draw [name path=line1] (focus) -- (point1);
\draw [name path=line2] (focus) -- (point2);
\draw [name path=line3] (focus) -- (point3);
\node (point11) at ($ (focus)!.3!(point1) $) {};
\node (point12) at ($ (focus)!.55!(point1) $) {};
\node (point13) at ($ (focus)!.8!(point1) $) {};
\node (point21) at ($ (focus)!.3!(point2) $) {};
\node (point22) at ($ (focus)!.5!(point2) $) {};
\node (point23) at ($ (focus)!.8!(point2) $) {};
\node (point31) at ($ (focus)!.3!(point3) $) {};
\node (point32) at ($ (focus)!.5!(point3) $) {};
\node (point33) at ($ (focus)!.8!(point3) $) {};
\foreach \point in {point11,point12,point13,point21,point22,point23,point31,point32,point33}
\fill [blue,opacity=1] (\point) circle (0.75pt);
\foreach \point in {focus}
\fill [red,opacity=1] (\point) circle (0.75pt);
\end{tikzpicture}
\caption{A point set and focus}\label{fig:focus}
\end{figure}
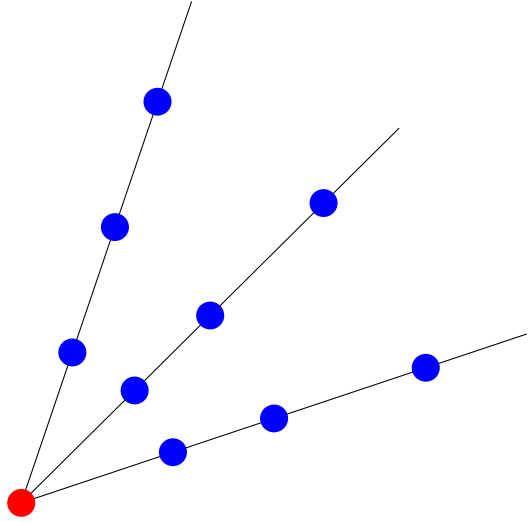

We also give a highly suggestive name to a particular configuration of points and foci. Let $P$ be a point set and $p_1,p_2,p_3,p_4 \notin P$ be distinct points. Say that $P$ and the $p_i$ form a \textbf{$K$-sum-product configuration} if 
	\begin{samepage}
	\begin{enumerate}
	\item Each $p_i$ is a $K$-focus for $P$.
	\item There is a line, which we call the \textbf{base line}, incident to $p_2,p_3$ and $p_4$ but not $p_1$.
	\item No point in $P$ is incident to the base line.
	\end{enumerate} 
	\end{samepage}
Figure \ref{fig:sumprodconfig} illustrates sum-product configurations with (a) the base line in general position and (b) the base line at infinity. The red, black, green and yellow points are the foci $p_i$ and the blue points are the elements of $P$.

\begin{figure}[ht]
\vspace{20pt}
\centering
\subbottom[Base line in general position]{
\begin{tikzpicture}[scale=7]
\coordinate (U) at (0,1);
\coordinate (V) at (1.5,1);
\draw [name path=base, thick] (U) -- (V);
\node (X) at ($ (U)!0.1!(V) $) {};
\node (Y) at ($ (U)!.5!(V) $) {};
\node (Z) at ($ (U)!0.9!(V) $) {};
\draw [name path=line1] (X) -- (1.4,0.3);
\draw [name path=line2] (X) -- (1,0.1);
\draw [name path=line3,color=yellow!50!brown] (Z) -- (0.1,0.3);
\draw [name path=line4,color=yellow!50!brown] (Z) -- (0.5,0.1);
\path [name intersections={of=line1 and line3,by={A}},];
\path [name intersections={of=line1 and line4,by={B}},];
\path [name intersections={of=line2 and line3,by={C}},];
\path [name intersections={of=line2 and line4,by={D}},];
\draw [shorten >=-2.5cm,color=green!70!black](Y)--(A);
\draw [shorten >=-2cm,color=green!70!black](Y)--(B);
\draw [shorten >=-2cm,color=green!70!black](Y)--(C);
\coordinate (W) at ($ (C)!-1.2!(B) $) {};
\draw [shorten >=-2cm, color=red](W)--(B);
\draw [shorten >=-3cm, color=red](W)--(A);
\draw [shorten >=-2.5cm, color=red](W)--(D);
\foreach \point in {W}
\fill [red,opacity=1] (\point) circle (0.75pt);
\foreach \point in {X}
\fill [black,opacity=1] (\point) circle (0.75pt);
\foreach \point in {Y}
\fill [green!70!black,opacity=1] (\point) circle (0.75pt);
\foreach \point in {Z}
\fill [yellow!50!brown,opacity=1] (\point) circle (0.75pt);
\foreach \point in {A,B,C,D}
\fill [blue,opacity=1] (\point) circle (0.75pt);
\end{tikzpicture}}

\subbottom[Base line at infinity]{
\begin{tikzpicture}[scale=7]
\coordinate (vert1) at (0.33,0);
\coordinate (vert2) at (0.66,0);
\coordinate (vert1') at (0.33,1);
\coordinate (vert2') at (0.66,1);
\coordinate (hor1) at (0,0.33);
\coordinate (hor2) at (0,0.66);
\coordinate (hor1') at (1,0.33);
\coordinate (hor2') at (1,0.66);
\draw [name path=vert1, color=yellow!50!brown] (vert1) -- (vert1');
\draw [name path=vert2, color=yellow!50!brown] (vert2) -- (vert2');
\draw [name path=hor1] (hor1) -- (hor1');
\draw [name path=hor2] (hor2) -- (hor2');
\coordinate (focus) at (0,0);
\path [name intersections={of=vert2 and hor2,by={point1}},];
\path [name intersections={of=vert2 and hor1,by={point2}},];
\coordinate (point3) at (0.33,0.33);
\coordinate (point4) at (0.33,0.66);
\draw [name path=line1,color=red] (focus) -- (0.9,0.9);
\draw [name path=line2,color=red] (focus) -- (1,0.5);
\draw [name path=line2,color=red] (focus) -- (0.5,1);
\foreach \point in {focus}
\fill [red,opacity=1] (\point) circle (0.75pt);
\draw [name path=add1,color=green!70!black] (0,1) -- (1,0);
\draw [name path=add2,color=green!70!black] (0,0.66) -- (0.66,0);
\draw [name path=add3,color=green!70!black] (0.33,1) -- (1,0.33);
\foreach \point in {point1,point2,point3,point4}
\fill [blue,opacity=1] (\point) circle (0.75pt);
\end{tikzpicture}}
\caption{Sum-product configurations}\label{fig:sumprodconfig}
\end{figure}

The following lemma justifies the choice of definition by showing that a sum-product configuration does indeed correspond to a (partial) sumset and product set.

\begin{samepage}
\begin{lemma}\label{theorem:reduction}
Let $F$ be a field. Suppose $P$ is a set of points in $F^2$, and that it forms a $K$-sum-product configuration with points $p_1,p_2,p_3,p_4$. Then there exist sets $A,B \subseteq F$ and $G \subseteq A \times B$ with 
	$$|G|=|P|$$
and 
	$$|A|,|B|,|A \overset{G}-B|,|A \overset{G}/B| \leq K.$$
\end{lemma}
\end{samepage}

\begin{proof}
There is a projective transformation mapping any frame (four points, no three of which are collinear) to any other frame, and so certainly there is one mapping any non-collinear triple to any other. Since $p_1,p_3$ and $p_4$ are not collinear, and nor are $[1,0,0],[0,1,0],[0,0,1]$, it is possible to pick a projective transformation $\tau$ such that
\begin{samepage}
\begin{itemize}
\item $p_1$ is sent to the origin $[0,0,1]$ of $F^2$.
\item $p_3$ is sent to $[0,1,0]\in l_{\infty}$.
\item $p_4$ is sent to $[1,0,0]\in l_{\infty}$.
\end{itemize}
\end{samepage}
Since $\tau$ preserves linearity and sends $p_3$ and $p_4$, which are incident to the base line, to points on $l_{\infty}$, it in fact sends the whole base line to $l_{\infty}$.  

Let $G=\tau(P)$. This is contained in $F^2$ since none of the points in $P$ are incident to the base line and so none are mapped into $l_{\infty}$. 

The set $G$ is supported over $K$ `vertical' lines\footnote{that is, lines of the form $x=c$ for some $c \in F$} and $K$ `horizontal' lines\footnote{that is, lines of the form $y=c$ for some $c \in F$}, since lines incident to $[0,1,0]\in l_{\infty}$ are all vertical and those incident to $[1,0,0]\in l_{\infty}$ are all horizontal. Let $A$ be the set of $x$-intercepts of the vertical lines and $B$ be the set of $y$-intercepts of the horizontal lines, so that 
	$$G \subseteq A \times B$$
and 
	$$|A|,|B|\leq K.$$
These correspond to the yellow and black lines respectively in Figure \ref{fig:sumprodconfig}. 

Furthermore, $G \subseteq A \times B$ is supported over $K$ lines through the origin. These are identified by their gradient, and a point $(a,b)\in G$ is incident to the line with gradient $\xi$ if and only if $\frac{a}{b}=\xi$. Thus each element of $A \overset{G}/ B$ corresponds to a different line through the origin, i.e. a red line in Figure \ref{fig:sumprodconfig}, and so 
	$$|A \overset{G}/ B|\leq K.$$ 

Finally, $G$ is supported over $K$ lines through $\tau(p_2)$. Since $\tau$ preserves linearity and sends the base line, which is incident to $p_2$, to $l_{\infty}$ we know that $\tau(p_2)\in l_{\infty}$. So all lines incident to $\tau(p_2)$ have the same gradient, say $\lambda \in F$. These are identified by the intercept with the $y$-axis, and a point $(a,b)\in G$ is incident to the line with intercept $\rho$ if and only if $a+\lambda b = \rho$. Thus each element of $A \overset{G}+ \lambda B$ corresponds to a different line of gradient $\lambda$, i.e. a green line in Figure \ref{fig:sumprodconfig}. Therefore 
	$$|A \overset{G}+ \lambda B|\leq K.$$  

Now let $B'= -\lambda B$ and $G'=\left\{(a,\lambda b):(a,b)\in G\right\}$ to obtain 
	$$|G'|=|P|$$ 
and 
	$$|A|,|B'|,|A \overset{G'}-B'|,|A \overset{G'}/B'| \leq K.$$
\end{proof}

\section{How to find foci}\label{section:refine}

Section \ref{section:interpretation} shows that a sum-product configuration of points and foci corresponds to an upper bound on partial sum-products. This section shows that such a configuration must arise whenever there are many point-line incidences. The main result is the following lemma.

\begin{lemma}[Finding a sum-product configuration] \label{theorem:mainrefine}
Let $P$ and $L$ be a set of points and lines respectively in a plane such that every point in $P$ is incident to $\Theta(K)$ lines in $L$. Suppose that $K \gg  \frac{|L|^{3/5}}{|P|^{1/5}}$, and that each line in $L$ is incident to at most 
$$O\left(\min \left\{\frac{|P|K^{8}}{|L|^4},\frac{|P|K^{4}}{|L|^2}\right\}^{1-o(1)}\right)$$ 
points in $P$. 

Then there exists a subset of $P$ of cardinality $\Theta\left(\frac{|P|K^8}{|L|^4}\right)$, and points $p_1,p_2,p_3,p_4$ with which it forms a $O(K)$-sum-product configuration.
\end{lemma}

It will be helpful when proving Lemma \ref{theorem:mainrefine} to adopt some additional notation. For points $p,q$ in the plane, let $l_{pq}$ be the line determined by $p$ and $q$. Given a set $P$ of points and a set $L$ of lines, both satisfying the conditions of Lemma \ref{theorem:mainrefine}, and a single point $p\in P$, define 
	$$P_{pL}=\left\{q\in P: l_{pq} \in L\right\}$$
so that $P_{pL}$ is the set of points in $P$ incident to lines in $L$ going through $p$. Thus $P_{pL}$ has $p$ as an $O(K)$-focus, since every point in $P$ is incident to $\Theta(K)$ lines in $L$.

As a first step, it is convenient to record the following standard result. 

\begin{lemma}[Rich points and lines]\label{theorem:pointrefine}
Let $P_1$ be the set of points in $P$ incident to at least $\frac{I(P,L)}{2|P|}$ lines in $L$. Then 
	$$I(P_1,L)\approx I(P,L).$$
Similarly, if $L_1$ is the set of lines in $L$ incident to at least $\frac{I(P,L)}{2|L|}$ points in $P$ then 
	$$I(P,L_1)\approx I(P,L).$$
\end{lemma}

\begin{proof}
We prove the result for points, leaving that for lines as an exercise. Let $P_2$ be the set of points in $P$ incident to at most $\frac{I(P,L)}{2|P|}$ lines in $L$. Then 
	\begin{align*}I(P_2,L)&=\sum_{p \in P_2}\#\left\{l \in L \text{ incident to } p\right\}\leq |P_2| \frac{I(P,L)}{2|P|}\leq 		
	\frac{I(P,L)}{2}.
	\end{align*} 
Since 
	$$I(P,L)=I(P_1,L)+I(P_2,L)$$ 
it follows that
	$$I(P_1,L)\geq \frac{I(P,L)}{2}$$ 
as required.
\end{proof}

Armed with Lemma \ref{theorem:pointrefine} we now build incrementally towards Lemma \ref{theorem:mainrefine}. We shall show how to find individual foci, and simultaneous pairs of foci, before finding four arranged in a sum-product configuration. The initial results on finding singleton or paired foci are adjustments of similar methods used in the proofs of Bourgain-Katz-Tao \cite{BKT} and Helfgott-Rudnev \cite{HR}; they are applied in the proof of the more developed Lemma \ref{theorem:mainrefine}.  

The following result enables us find individual foci. 

\begin{lemma}[Finding individual foci]\label{theorem:incidencescover1}
Let $P$ be a set of points and $L$ a set of lines, such that every point is incident to $\Theta(K)$ lines in $L$. Then there exists $P_1 \subseteq P$ with $|P_1|\approx |P|$ such that 
	$$|P_{pL}|\gg \frac{K^2|P|}{|L|}$$ 
for each $p \in P_1$.
\end{lemma}

\begin{proof}
Since every point in $P$ is incident to $\Theta(K)$ lines in $L$ it is immediate that
	$$I(P,L)\approx K|P|.$$
Let $L_1$ be the set of lines in $L$ incident to $\Omega\left(\frac{I(P,L)}{|L|}\right)=\Omega\left(\frac{K|P|}{|L|}\right)$ points in $P$ so that by Lemma \ref{theorem:pointrefine}, it follows that
	$$I(P,L_1)\approx I(P,L)\approx K|P|.$$ 
Now let $P_1$ be the set of points in $P$ incident to $\Omega\left(\frac{I(P,L_1)}{|P|}\right)=\Omega\left(K\right)$ lines in $L_1$. By Lemma \ref{theorem:pointrefine} again,
	\begin{equation}\label{eq:incidences1}
	I(P_1,L_1)\approx I(P,L_1)\approx K|P|.
	\end{equation}
Since $P_1 \subseteq P$ and $L_1 \subseteq L$, each point in $P_1$ is incident to at most $O(K)$ lines in $L_1$ and so  
	\begin{equation}\label{eq:incidences2}
	I(P_1,L_1)\ll K|P_1|.
	\end{equation}
Comparing (\ref{eq:incidences1}) and (\ref{eq:incidences2}) shows that that $|P_1|\gg |P|$, implying that $|P_1|\approx |P|$ since $P_1$ is a subset of $P$. 

Now, each $p \in P_1$ is incident to $\Omega\left(K\right)$ lines in $L_1$, and each of these lines is incident to $\Omega\left(\frac{K|P|}{|L|}\right)$ points in $P$. So 
	$$|P_{pL}|\geq |P_{pL_1}|\gg \frac{K^2|P|}{|L|}$$ 
for each $p \in P_1$, as required.
\end{proof}

The following result now takes the ability to find singleton foci, and uses it to find them in simultaneous pairs. 

\begin{lemma}[Finding paired foci]\label{theorem:bktrefine}
Let $P$ and $L$ be a set of points and lines respectively in a plane such that every point in $P$ is incident to $\Theta(K)$ lines in $L$. Then there exist $p_1,p_2 \in P$ such that
	$$\left|P_{p_1L}\cap P_{p_2L}\right|\gg \frac{|P|K^4}{|L|^2}.$$
\end{lemma} 

\begin{proof}
By Lemma \ref{theorem:incidencescover1} there exists $P_1 \subseteq P$ with $|P_1|\approx |P|$ such that $|P_{pL}|\gg \frac{K^2|P|}{|L|}$ for each $p \in P_1$. In particular there is a point $p_1 \in P$ such that 
	$$\left|P_{p_1L}\right|\gg \frac{K^2|P|}{|L|}.$$ 
Applying Lemma \ref{theorem:incidencescover1} again, this time to $P_{p_1L}$ and $L$, there is a point $p_2 \in P_{p_1L}$ such that 
	$$\left|P_{p_1L}\cap P_{p_2L}\right|\geq \left|\left(P_{p_1L}\right)_{p_2L}\right|\gg \frac{K^2|P_{p_1L}|}{|L|}\gg \frac{K^4|P|}{|L|^2}$$
as required.	
\end{proof}

Now with the ability to find foci both one and two at a time, it is time for a proof of Lemma \ref{theorem:mainrefine}

\begin{proof}[Proof of Lemma \ref{theorem:mainrefine}]
It suffices to find $p_1,p_2,p_3,p_4 \in P$ with
	$$\left|P_{p_1L}\cap P_{p_2L} \cap P_{p_3L} \cap P_{p_4L} \right| \gg \frac{|P|K^8}{|L|^4}$$
such that $p_2,p_3$ and $p_4$ are collinear along a line in $L$ that is \textit{not} incident to $p_1$. There will then exist an $O(K)$-sum-product configuration of appropriate size since:

\begin{itemize}
\item Each $p \in P$ is incident to $\Theta(K)$ lines in $L$, and so is an $O(K)$-focus for $P_{pL}$. This means that the first condition in the definition of an $O(K)$-sum-product configuration is satisfied for any subset of $P_{p_1L}\cap P_{p_2L} \cap P_{p_3L} \cap P_{p_4L}$.
\item The second condition of a sum-product configuration is clearly satisfied by virtue of the arrangement of $p_1,p_2,p_3,p_4$.
\item No line in $L$ is incident to more than $O\left(\left(\frac{|P|K^8}{|L|^4}\right)^{1-o(1)}\right)$ points in $P$, and so 
$P_{p_1L}\cap P_{p_2L} \cap P_{p_3L} \cap P_{p_4L}$ contains a subset of cardinality $\Theta\left(\frac{|P|K^8}{|L|^4}\right)$ for which the third condition of an $O(K)$-sum-product configuration is satisfied. 
\end{itemize}

Observe first that by Lemma \ref{theorem:bktrefine} there exist points $p_1,p_2 \in P$ such that
	\begin{equation}\label{eq:incidences14}
	\left|P_{p_1L}\cap P_{p_2L}\right|\gg \frac{|P|K^4}{|L|^2}.
	\end{equation}
For convenience, define 
	\begin{equation}\label{eq:incidences14'}
	Q=P_{p_1L}\cap P_{p_2L}.
	\end{equation}
The proof will be complete if we can find distinct $p_3,p_4 \in Q$ collinear with $p_2$ along a line in $L$ that is not incident to $p_1$ such that
	\begin{equation}\label{eq:target}
	\left|Q_{p_3L} \cap Q_{p_4L}\right|\gg \frac{|P|K^8}{|L|^4}.
	\end{equation}

Now since $Q \subseteq P$, every point in $Q$ is incident to $\Theta(K)$ lines in $L$. So by Lemma \ref{theorem:incidencescover1} there exists $Q_1 \subseteq Q$ with $|Q_1|\approx |Q|$ such that 
\begin{equation}\label{eq:incidences13}
|Q_{pL}|\gg \frac{K^2|Q|}{|L|}
\end{equation}	
for each $p \in Q_1$. 

Since $p_2$ is an $O(K)$-focus for $P_{p_2L}$, and $Q_1 \subseteq Q \subseteq P_{p_2L}$, it follows that $p_2$ is an $O(K)$-focus for $Q_1$ as well. Let $J \subseteq L$ be the set of $O(K)$ lines through $p_2$ supporting $Q_1$, so that 
	$$I(Q_1,J)=|Q_1|$$
and
	$$|J|\ll K.$$ 

Let $J_1$ be the set of $l \in J$ incident to at least $\Omega\left(\frac{I(Q,J)}{|J|}\right)=\Omega\left(\frac{|Q|}{K}\right)$ points in $Q_1$. Observe that $|J_1|\geq 2$. Indeed, Lemma \ref{theorem:pointrefine} implies  
	\begin{equation}\label{eq:incidences11}
	I(Q_1,J_1)\approx I(Q_1,J)\approx |Q|.
	\end{equation}
But since $Q_1\subseteq P$ and $J_1 \subseteq L$, and it is a hypothesis that each line in $L$ is incident to at most 	
	$$O\left(\left(\frac{|P|K^{4}}{|L|^2}\right)^{1-o(1)}\right)=O(|Q|^{1-o(1)})$$
points in $P$, it follows that each line in $J_1$ is incident to at most $O(|Q|^{1-o(1)})$ points in $Q$. Hence
	\begin{equation}\label{eq:incidences12}
	I(Q_1,J_1)\ll |J_1||Q|^{1-o(1)}.
	\end{equation}
Comparing (\ref{eq:incidences11}) and (\ref{eq:incidences12}) gives $|J_1|\gg |Q|^{o(1)}$. So by appropriate choice of constants in the statement of the theorem $|J_1| \geq 2$ as claimed.

Since there are at least two lines in $J_1$, and they are all incident to $p_2$, at least one of them is not incident to $p_1$. Fix this line $l^* \in J_1$, which will be the base line of the sum-product configuration. Since $l^*$ is incident to $p_2$ but not $p_1$, it suffices to establish that \eqref{eq:target} holds for some distinct $p_3,p_4 \in Q \cap l^*$.

Because $l^* \in J_1$ we have 
	\begin{equation}\label{eq:incidences15}
	|Q_1 \cap l^*|\gg \frac{|Q|}{K}
	\end{equation}
From \eqref{eq:incidences13} it follows that
	\begin{equation}\label{eq:incidences6}
	\frac{K^2|Q|}{|L|}\left|Q_1 \cap l^*\right|\ll \sum_{p \in Q_1 \cap l^*}|Q_{pL}|.
	\end{equation}
On the other hand, by Cauchy-Schwarz, 
	\begin{align*}
	\sum_{p \in Q_1 \cap l^*}|Q_{pL}|& \leq |Q|^{1/2}\left(\sum_{p_3,p_4 \in Q_1\cap l^*}\left|Q_{p_3L}\cap Q_{p_4L}\right|\right)^{1/2}\\
	&= |Q|^{1/2}\left( \sum_{p \in Q_1 \cap l^*}|Q_{pL}|+ \sum_{\substack{p_3, p_4 \in Q_1\cap l^* \\p_3 \neq p_4}}\left|Q_{p_3L}\cap Q_{p_4L}\right|\right)^{1/2}.
	\end{align*}
If the first summation on the right were to dominate then it would mean 
	$$\sum_{p \in Q_1 \cap l^*}|Q_{pL}|\ll |Q|.$$
Comparing with \eqref{eq:incidences6} and applying \eqref{eq:incidences15} would then yield $|Q|K \ll |L|$. By \eqref{eq:incidences14} and \eqref{eq:incidences14'}, this would mean $K \ll \frac{|L|^{3/5}}{|P|^{1/5}}$, and so by an appropriate choice of constant contradict the hypothesis $K \gg \frac{|L|^{3/5}}{|P|^{1/5}}$. 

Thus the second summation on the right dominates instead and so by \eqref{eq:incidences6},
	$$\frac{K^4 |Q| \left|Q \cap l_*\right|^2}{|L|^2} \ll \sum_{\substack{p_3,p_4 \in Q_1\cap l^* \\ p_3 \neq p_4}}|Q_{p_3L}\cap Q_{p_4L}|.$$
Hence there exist distinct $p_3,p_4 \in Q_1 \cap l^* \subseteq Q \cap l^*$ such that 
	$$|Q_{p_3L}\cap Q_{p_4L}|\gg \frac{K^4 |Q|}{|L|^2}\gg \frac{|P|K^8}{|L|^4}$$
as required.
\end{proof}

\section{Bounding partial sum-products}\label{section:bounding}

Between them, Section \ref{section:interpretation} and Section \ref{section:refine} show that the existence of too many incidences implies the existence of $A,B \subseteq \mathbb{F}_p$ and a large $G \subseteq A \times B$ for which the partial difference and ratio sets $|A\overset{G}-B|$ and $|A\overset{G}/B|$ are both small relative to $|A|$ and $|B|$. 

This offends our sum-product sensibilities, since Chapter \ref{chapter:introgrowth} showed that one or other of a product set or sumset must always be large. However those results were concerned with \textit{complete} sumsets and product sets, whereas here   only \textit{partial} sum-products are controlled. Fortunately, Chapter \ref{chapter:introapproxgroup} has  tools designed for this kind of situation, in the form of Balog-Szemer\'edi-Gowers type results that allow passage from partial to complete sets. 

The main results in this section are the following `partial sum-product' results for $\mathbb{F}_p$. They will be proved using a combination of BSG-type results and Rudnev's finite field sum-product estimate Theorem \ref{theorem:mishasumprodvariation}.

\begin{lemma}[Partial sum-products in $\mathbb{F}_p$, v1]\label{theorem:partialsumprod}
Let $A,B \subseteq \mathbb{F}_p$ and $G \subseteq A \times B$. If $|A|\ll p^{1/2} $ then
	$$|G|^{55}\lesssim |A|^{36}|B|^{37}|A \overset{G}-B|^{28}|A \overset{G}/ B|^8.$$
\end{lemma}

\begin{lemma}[Partial sum-products in $\mathbb{F}_p$, v2]\label{theorem:partialsumprod2}
Let $A,B \subseteq \mathbb{F}_p$ and $G \subseteq A \times B$. If $|G| \ll p^{1/2} |B|$ then
	$$|G|^{67}\lesssim |A|^{44}|B|^{45}|A \overset{G}-B|^{28}|A \overset{G}/ B|^{16}.$$
\end{lemma}

Because sumsets and product sets are being analysed simultaneously, care is required when using Balog-Szemer\'edi-Gowers type results to prove these lemmata. Results from Chapter \ref{chapter:introapproxgroup} yield the following preliminary result, which captures all the additive and multiplicative properties that need to be controlled. 

\begin{lemma}[BSG-type for sum-products]\label{theorem:halfbsg}
If $A,B \subseteq \mathbb{F}_p$ and $G \subseteq A \times B$ then there exists $A' \subseteq A$ with $|A'|\gg \frac{|G|}{|B|}$ such that 
\begin{enumerate}
\item $|A'-A'|\ll \frac{|A\overset{G}-B|^4|A|^4|B|^3}{|G|^5}$
\item $|A'/A'|\ll \frac{|A\overset{G}/B|^4|A|^4|B|^3}{|G|^5}$
\item $E_{\times}(A') \gg \frac{|G|^2|A'|^4}{|A\overset{G}/B|^2|A|^2|B|}$.
\end{enumerate}
\end{lemma}

\begin{proof}
Let $\epsilon>0$ be sufficiently small. By Lemma \ref{theorem:bsgrefine} there exist $A_1 \subseteq A$ and $H \subseteq A_1 \times A_1$ with 
	$$|A_1|\gg \frac{|G|}{|B|}$$
and 
	$$|H|\geq (1-\epsilon) |A_1|^2$$
such that every pair $(a_1,a_2)\in H$ has joint $G$-degree at least $\frac{\epsilon |G|^2}{2|A|^2|B|}$. By Lemma \ref{theorem:bsgsufficient}, applied once additively and once multiplicatively,
	\begin{align*}
	|A_1\overset{H}-A_1|&\ll \frac{|A\overset{G}- B|^2 |A|^2 |B|}{|G|^2}\\
	|A_1\overset{H}/A_1|&\ll \frac{|A\overset{G}/ B|^2 |A|^2 |B|}{|G|^2}.
	\end{align*}
Apply Lemma \ref{theorem:densebsg} once additively and once multiplicatively to obtain $A_2,A_3 \subseteq A_1$ with 		
	$$|A_2|,|A_3|\geq(1-\sqrt{\epsilon})|A_1|$$
such that
	\begin{align*}
	|A_2-A_2|&\ll \frac{|A_1\overset{H}-A_1|^2}{|A_1|}\ll \frac{|A\overset{G}-B|^4|A|^4|B|^3}{|G|^5}\\
	|A_3/A_3|&\ll \frac{|A_1\overset{H}/A_1|^2}{|A_1|}\ll \frac{|A\overset{G}/B|^4|A|^4|B|^3}{|G|^5}.\\
	\end{align*}
Let $A'=A_2 \cap A_3$. By the cardinalities of $A_2$ and $A_3$,
	$$|A'|\geq (1-2\sqrt\epsilon)|A_1|\gg \frac{|G|}{|B|}.$$ 
That $A'$ satisfies properties 1 and 2 is immediate from the above difference and ratio set estimates. To show that property 3 is also satisfied, let 
	$$H' = H \cap (A'\times A').$$
Since both $H$ and $A' \times A'$ are of cardinality at least $(1-2\sqrt{\epsilon})^2|A_1|^2$ we have 
	$$|H'|\geq (1-8\sqrt{\epsilon}) |A_1|^2$$
and so
	$$E_{\times}(A')\geq \frac{|H'|^2}{|A'\overset{H'}/A'|} \gg \frac{|A_1|^4}{|A_1\overset{H}/A_1|} \gg \frac{|G|^2|A_1|^4}{|A\overset{G}-B|^2|A|^2|B|}$$
which completes the proof.
\end{proof}

We can now prove Lemma \ref{theorem:partialsumprod} and Lemma \ref{theorem:partialsumprod2}.

\begin{proof}[Proof of Lemma \ref{theorem:partialsumprod}.]
Apply Lemma \ref{theorem:halfbsg} to obtain $A' \subseteq A$ with $|A'|\gg \frac{|G|}{|B|}$ such that 
	\begin{align*}
	|A'-A'|&\ll \frac{|A\overset{G}-B|^4|A|^4|B|^3}{|G|^5}\\
	E_{\times}(A') &\gg \frac{|G|^2|A'|^4}{|B||A|^2|A\overset{G}/B|^2}.
	\end{align*}
Since $|A|\ll p^{1/2}$, the set $A' \subseteq A$ is of cardinality at most $O\left(p^{1/2}\right)$ as well, and the sum-product estimate Theorem \ref{theorem:mishasumprodvariation} can be applied to obtain
	$$\frac{|G|^{8}|A'|^{16}}{|B|^{4}|A|^8|A \overset{G} / B|^8}\ll E_{\times}(A')^4 \lesssim |A' - A'|^7 |A'|^4 \ll \frac{|A\overset{G}-B|^{28}|A|^{28}|B|^{21}|A'|^4}{|G|^{35}}. $$
Rearranging gives 
	$$|G|^{43}|A'|^{12} \lesssim |A \overset{G}-B|^{28}|A \overset{G}/ B|^8|A|^{36}|B|^{25}$$
and so 
	$$|G|^{55}\lesssim A|^{36}|B|^{37}|A \overset{G}-B|^{28}|A \overset{G}/ B|^8| $$
as required.
\end{proof}

\begin{proof}[Proof of Lemma \ref{theorem:partialsumprod2}.]
Apply Lemma \ref{theorem:halfbsg} to obtain $A' \subseteq A$ with $|A'|\gg \frac{|G|}{|B|}$ such that
	\begin{align*}
	|A'-A'|&\ll \frac{|A\overset{G}-B|^4|A|^4|B|^3}{|G|^5}\\
	|A'/A'|&\ll \frac{|A\overset{G}/B|^4|A|^4|B|^3}{|G|^5}.
	\end{align*}
Unlike in the proof of Lemma \ref{theorem:partialsumprod}, where multiplicative energy was involved, these properties are preserved when passing to subsets of $A'$ and so we may assume that $|A'|\approx \frac{|G|}{|B|}$. Since $|G| \ll p^{1/2} |B|$ the set $A'$ is then of cardinality at most $O\left(p^{1/2}\right)$ and so Theorem \ref{theorem:mishasumprodvariation} can be applied to obtain
	\begin{align*}
	\frac{|G|^{36}}{|B|^{28}|A|^{16}|A \overset{G}/B|^{16}}\ll \left(\frac{|A'|^4}{|A'/A'|}\right)^4  \ll E_{\times}(A')^4& \lesssim |A'|^4|A'-A'|^7 \\
	&\ll \frac{|A\overset{G}-B|^{28}|A|^{28}|B|^{17}}{|G|^{31}}.
	\end{align*}
Rearranging gives 
	$$|G|^{67}\lesssim |A|^{44}|B|^{45}|A \overset{G}-B|^{28}|A \overset{G}/ B|^{16}$$
as required.
\end{proof}

\section{Proving Theorems \ref{theorem:incidences} and \ref{theorem:beck2}}\label{section:incidenceproofs}

This section uses the analysis from Sections \ref{section:interpretation}, \ref{section:refine} and \ref{section:bounding} to give proofs of Theorems \ref{theorem:incidences} and \ref{theorem:beck2}.

\begin{proof}[Proof of Theorem \ref{theorem:incidences}.]

Suppose that $I(P,L)\gg N^{3/2-\epsilon}$. The aim is to show that $\epsilon \geq \frac{1}{662}-o(1)$. 

First, a standard argument enables overly rich lines to be discarded. Let $L_1$ be the set of lines in $L$ incident to at most $O\left(N^{1/2+\epsilon}\right)$ points in $P$. We have 		
	$$I(P,L_1)\approx N^{3/2-\epsilon}$$ 
since if $L_+$ is the set of lines in $L$ incident to at least $C N^{1/2+\epsilon}$ points in $P$ then
	\begin{align*}
I(P,L_+)&=\sum_{l \in L_2}\sum_{p \in P}\delta_{pl}\\
&\leq \frac{1}{C N^{1/2+\epsilon}}\sum_{l \in L_+}\sum_{p_1,p_2}\delta_{p_1l}\delta_{p_2l}\\
&\leq \frac{1}{C N^{1/2+\epsilon}}\left(I(P,L)+|P|^2\right)\\
&\leq \frac{2N^{3/2-\epsilon}}{C}	
	\end{align*}
and so by an appropriate choice of constant $I(P,L_+)\leq \frac{I(P,L)}{2}$. 

By a dyadic pigeonholing there is a subset $P_1 \subseteq P$ and an integer $K$ with 
	\begin{equation}\label{eq:incidences7}
	|P_1|K \gtrsim N^{3/2- \epsilon}
	\end{equation}
such that every point in $P_1$ is incident to $\Theta(K)$ lines in $L_1$. Note moreover that
	\begin{equation}\label{eq:incidences8}
	K \gtrsim N^{1/2-\epsilon}
	\end{equation}
since $|P_1|\leq N$. Applying Lemma \ref{theorem:mainrefine} to $P_1$ and $L_1$, at least one of the following is true:
	\begin{enumerate}
	\item $K \ll \frac{|L_1|^{3/5}}{|P_1|^{1/5}}$.
	\item There is a line in $L_1$ incident to at least $\Omega\left(\left(\frac{|P_1|K^{4}}{|L_1|^2}\right)^{1-o(1)}\right)$ points in $P_1$.
	\item There is a line in $L_1$ incident to at least $\Omega\left(\left(\frac{|P_1|K^{8}}{|L_1|^4}\right)^{1-o(1)}\right)$ points in $P_1$.
	\item There exists $P_2 \subseteq P_1$ with $|P_2|\approx \frac{|P_1|K^8}{|L_1|^4}$ and points $p_1,p_2,p_3,p_4$ in an $O(K)$-sum-product	configuration.
	\end{enumerate}

The first three cases are quickly dispensed with. For the first, applying \eqref{eq:incidences7} and then \eqref{eq:incidences8} yields $\epsilon \geq \frac{1}{10}-o(1)$, which is far better than required. For the second, since every line in $L_1$ is incident to at most $N^{1/2+\epsilon}$ points in $P$, it follows that $\epsilon \geq \frac{1}{10}-o(1)$ as well. The third case is like the second, but this time yields $\epsilon \geq \frac{1}{18}-o(1)$.

The fourth case remains, and is the core of the argument. Apply Lemma \ref{theorem:reduction} to $P_2$ to obtain $A, B \subseteq \mathbb{F}_p$ with 
	$$|A|,|B| \ll K$$
and $G \subseteq A \times B$ with 
	$$|G|\approx \frac{|P_1|K^8}{|L|^4}$$ 
such that 
	\begin{equation}\label{eq:refineinvar}
	|A \overset{G}-B|,|A \overset{G}/B|\ll K.
	\end{equation}
Note that either $K < p^{1/2}$ or $K > \frac{|G|}{p^{1/2}}$ since if
	$$p^{1/2} \leq K \leq \frac{|G|}{p^{1/2}}$$
then $|G|\geq p$ which contradicts the fact that $|G| \leq N <p$.	
	
If $K < p^{1/2}$ then $|A|\ll p^{1/2}$ and an application of Lemma \ref{theorem:partialsumprod} gives 
	$$|G|^{55} \lesssim K^{109}$$ 
which implies by the cardinality of $G$ that
	$$|P_1|^{55}K^{331} \ll |L|^{220}.$$ 
Applying \eqref{eq:incidences7}, \eqref{eq:incidences8} and the fact that $|L| \leq N$ gives after rearranging
	$$N^{1/2} \lesssim N^{331 \epsilon}$$
which implies $\epsilon \geq \frac{1}{662}-o(1)$ as required.

On the other hand, suppose that $K > \frac{|G|}{p^{1/2}}$. Assume also that $K \geq p^{1/2}$ since otherwise we are done by the previous paragraph. Note that with a fixed $G$, the bounds in \eqref{eq:refineinvar} are not affected by passing to supersets of $A$ or $B$. So assume $|B|\approx K$ and thus
	$$\frac{|G|}{|B|}\ll p^{1/2}. $$
Hence Lemma \ref{theorem:partialsumprod2} is applicable, giving 
	$$|G|^{67}\lesssim K^{133}$$ 
and therefore 
	$$|P_1|^{67}K^{403}\lesssim |L|^{268}.$$
Since $K\geq p^{1/2}>N^{1/2}$ and $|L| \leq N$ this means
	$$|P_1|^{67}K^{67}\lesssim N^{100}.$$
By \eqref{eq:incidences7} it falls out that $\epsilon \geq \frac{1}{134}-o(1)$, which is far better than required.
\end{proof}

\begin{proof}[Proof of Theorem \ref{theorem:beck2}.]
For $l \in L(P)$, write $\mu(l)$ for the number of points in $P$ incident to $l$. It is clear, as per the proof of Beck's theorem in Chapter \ref{chapter:introincidences} that
	$$|P|^2 \approx \sum_{l \in L(P)}\mu(l)^2.$$
By a dyadic pigeonholing there exists $L_1 \subseteq L(P)$ and an integer $k$ such that $\mu(l)\approx k$ for all $l \in L_1$ and 
	\begin{equation}
	\label{eq:incidences9}
	|L_1|k^2 \gtrsim |P|^2.
	\end{equation}
To prove the theorem it suffices to show that either 
	$$k \gtrsim |P|,$$ 
in which case there are $\tilde{\Omega}(|P|)$ collinear points, or 
	$$k^{133} \lesssim |P|^{66}$$
in which case $k^2 \lesssim |P|^\frac{132}{133}$ and so $|L(P)|\geq |L_1|\gtrsim |P|^{1+\frac{1}{133}}$ as required.

Note that $I(P,L_1)\approx |L_1|k$. So by another dyadic pigeonholing there exists $P_1 \subseteq P$ and an integer $K$ such that every point in $P_1$ is incident to $\Theta(K)$ lines in $L_1$ and 
	\begin{equation}\label{eq:incidences16}
	|P_1|K \gtrsim |L_1|k \gtrsim \frac{|P|^2}{k}.
	\end{equation}
Since $|P_1|\leq |P|$ we also have
	\begin{equation}\label{eq:incidences17}
	K \gtrsim \frac{|P|}{k}.
	\end{equation} 
Applying Lemma \ref{theorem:mainrefine} to $P_1$ and $L_1$, at least one of the following is true:

\begin{enumerate}
\item $K \ll \frac{|L_1|^{3/5}}{|P_1|^{1/5}}$.
\item There is a line in $L_1$ incident to at least $\Omega\left(\left(\frac{|P_1|K^{4}}{|L_1|^2}\right)^{1-o(1)}\right)$ points in $P_1$.
\item There is a line in $L_1$ incident to at least $\Omega\left(\left(\frac{|P_1|K^{8}}{|L_1|^4}\right)^{1-o(1)}\right)$ points in $P_1$.
\item There exists $P_2 \subseteq P_1$ with $|P_2|\approx \frac{|P_1|K^8}{|L_1|^4}$ and points $p_1,p_2,p_3,p_4$ in a $K$-sum-product configuration.
\end{enumerate}

As with Theorem \ref{theorem:incidences}, the first three cases fall easily. In the first case \eqref{eq:incidences16} and \eqref{eq:incidences17} give
	$$\frac{|P|^6}{k^5}\lesssim |L_1|^3$$	
By \eqref{eq:incidences9} this in turn implies that $k \lesssim 1$ which is far better than required. In the second case, since all lines in $L_1$ are incident to $\Theta(k)$ points in $P$ and hence $O(k)$ points in $P_1$, it follows that 
	$$k\gg \frac{|P_1|K^4}{|L_1|^2}.$$ 
By \eqref{eq:incidences16} and \eqref{eq:incidences17} this gives $|L_1|^2k^5\gtrsim |P|^5$ and so \eqref{eq:incidences9} implies $k \gtrsim |P|$. The third case similarly yields $k \gtrsim |P|$.
 
The fourth case is left. Apply Lemma \ref{theorem:reduction} to to $P_2$ to obtain $A, B \subseteq \mathbb{F}_p$ with $|A|,|B| \ll K$ and $G \subseteq A \times B$ with 
	$$|G|\approx \frac{|P_1|K^8}{|L_1|^4}$$
such that 
	$$|A \overset{G}-B|,|A \overset{G}/B|\ll K.$$ 

As with the proof of Theorem \ref{theorem:incidences}, either $K<p^{1/2}$ or $K>\frac{|G|}{p^{1/2}}$. If $K<p$ then as per the last proof, Lemma \ref{theorem:partialsumprod} gives 
	$$|P_1|^{55} K^{331} \ll |L_1|^{220}.$$ 
By (\ref{eq:incidences16}) and (\ref{eq:incidences17}) this gives 
	$$|P|^{386} \lesssim |L|^{220}k^{331}$$ 
By (\ref{eq:incidences9}) we then get $k^{109}\lesssim |P|^{54}$, which is better than the sufficient $k^{133}\lesssim |P|^{66}$.

On the other hand, if $K>\frac{|G|}{p^{1/2}}$ and $K>p^{1/2}$ then by Lemma \ref{theorem:partialsumprod2}
	$$|P_1|^{67}K^{403}\lesssim |L_1|^{268}.$$
By (\ref{eq:incidences16}) and (\ref{eq:incidences17}) this gives 	
	$$|P|^{470} \lesssim |L_1|^{268}k^{403}$$ 	
By \eqref{eq:incidences9} it follows that 
	$$k^{133} \lesssim |P|^{66}$$
as required.
\end{proof}

\section{Further work}\label{section:incidencesconclusion}

\textbf{An intrinsic proof in the plane?} The work here, and the previous work of Bourgain-Katz-Tao and Helfgott-Rudnev uses finite field sum-product estimates as an animating force. The progress in this chapter comes down to finding a particularly efficient way of relating this to incidences by showing that a `sum-product' configuration of points and lines must arise if there are too many incidences, but that such a configuration cannot exist since it would contradict known finite field growth results.

Is it instead possible to construct a proof that lives entirely in the plane? Such an approach, if it worked, could yield better incidence results. Since the Elekes result (Theorem \ref{theorem:elekes}) in Chapter \ref{chapter:introgrowth} shows that incidence bounds yield growth results, this could also drive forward the study of growth in finite fields.

There may be some cause for optimism. The setup of points and lines that we called a `sum-product' configuration is not only interesting because of its relationship to sum-products. It is also of intrinsic interest on the plane, where it invites comparison with Desargues' theorem. This theorem, actually a defining property of the projective plane, says that two triangles of points are perspective to a point if and only if they are perspective to a line, the meaning of which is illustrated in Figure \ref{fig:conclusionfig} below.

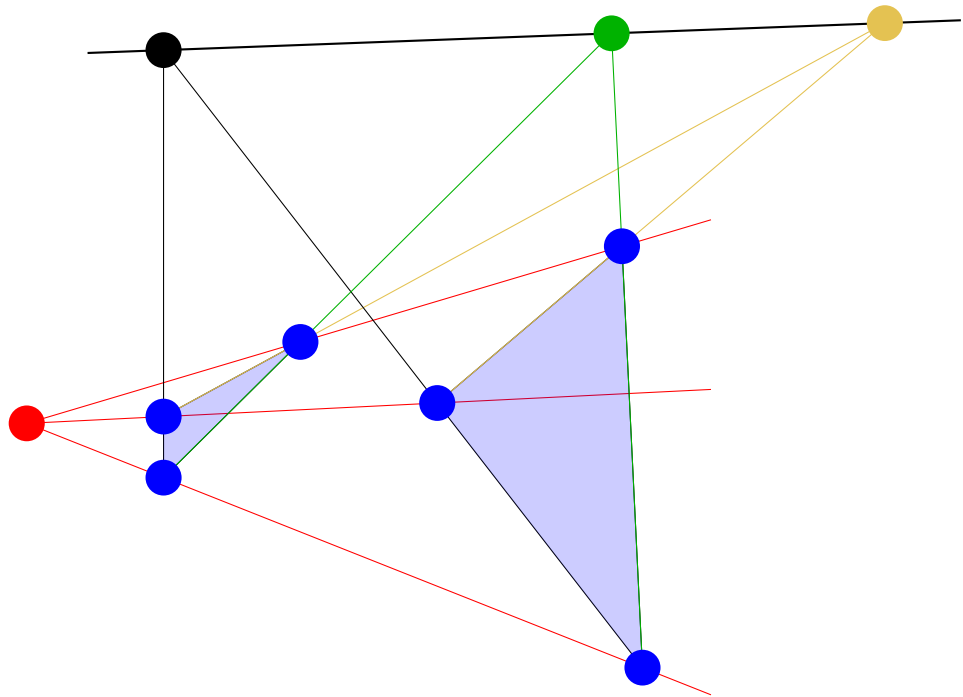
\begin{figure}[ht]
\vspace{20pt}
\centering
\begin{tikzpicture}[scale=9]
\coordinate (O) at (0,0);
\fill [red,opacity=1] (O) circle (0.75pt);
\coordinate(l1) at (1,0.05);
\coordinate(l2) at (1,0.3);
\coordinate(l3) at (1,-0.4);
\draw [name path=line1, color=red] (O) -- (l1);
\draw [name path=line2, color=red] (O) -- (l2);
\draw [name path=line3, color=red] (O) -- (l3);
\coordinate (A1) at ($ (O)!0.2!(l1) $);
\coordinate (A2) at ($ (O)!0.6!(l1) $);
\coordinate (B1) at ($ (O)!0.4!(l2) $);
\coordinate (B2) at ($ (O)!0.87!(l2) $);
\coordinate (C1) at ($ (O)!0.2!(l3) $);
\coordinate (C2) at ($ (O)!0.9!(l3) $);
\draw [name path=triangle1](A1) -- (B1) -- (C1) -- cycle;
\fill [blue,opacity=0.2] (A1)--(B1)--(C1)--cycle;
\draw [name path=triangle2](A2)--(B2)--(C2)--cycle;
\fill [blue,opacity=0.2] (A2)--(B2)--(C2)--cycle;
\coordinate (U) at (intersection of A1--B1 and A2--B2);
\coordinate (V) at (intersection of A1--C1 and A2--C2);
\coordinate (W) at (intersection of B1--C1 and B2--C2);
\draw [shorten >=-1cm,shorten <=-1cm, thick](U)--(V);
\draw [color=yellow!50!brown](A1)--(U);
\draw [color=yellow!50!brown](A2)--(U);
\draw (A1)--(V);
\draw (A2)--(V);
\draw [color=green!70!black](C1)--(W);
\draw [color=green!70!black](C2)--(W);
\fill [yellow!50!brown] (U) circle (0.75pt);
\fill [black] (V) circle (0.75pt);
\fill [green!70!black] (W) circle (0.75pt);
\foreach \point in {A1,A2,B1,B2,C1,C2}
\fill [blue,opacity=1] (\point) circle (0.75pt);
\end{tikzpicture}
\caption[An illustration of Desargues' theorem.]{\textbf{An illustration of Desargues' theorem.} The two shaded triangles are perspective to the red point, and are also perspective to the line carrying the black, green and yellow points.}\label{fig:conclusionfig}
\end{figure}

A sum-product configuration should imply the existence of many pairs of triangles perspective to the \textit{same} point and perspective to the \textit{same} line, at the \textit{same} three places. Could pursuing the geometric implications of such an arrangement be a fruitful line of inquiry?

%% file: expanders.tex
\chapter{Expander functions}
\label{chapter:expanders}

Chapter \ref{chapter:introgrowth} introduced expander functions as a type of growth result. Recall that for a field $F$, an $n$-variable expander is a function $f:F^n \to F$ such that for any subset $A$ of $F$, maybe satisfying some nondegeneracy conditions, the set $f(A)=\left\{f(a_1,\ldots,a_n)\right\}$ is of cardinality at least $|A|^{1+\delta}$ with $\delta>0$. This chapter sets new records in two variables over finite fields, and in three and four variables over real and complex numbers.

An earlier version of the two-variable finite field work formed one half of a joint paper \cite{meandollie}\footnote{The paper  is the union of two theorems proved independently by the two authors.} with Oliver Roche-Newton has been accepted for publication in the Journal of Combinatorial Theory Series A.

An earlier version of the three and four-variable real and complex results has been submitted to Discrete and Computational Geometry, and a preprint \cite{mecross} is available on the arXiv.

\section{Results}\label{section:expandersresults}

This section describes the new theorems proved in the chapter.

\subsection{Two variables}

As mentioned in Chapter \ref{chapter:introgrowth}, the strongest known two variable expanders is the function $f(a,b)=a+ab$, which was first studied by Garaev and Shen \cite{GS}. They proved three results about the size of the set $f(A)$, depending on the ambient field and the density of $A$ within it:

\begin{align}
&\label{eq:gs1}\text{If } A \subseteq \mathbb{F}_p \text{ with } |A|<p^{1/2} \text{ then } |f(A)|\gtrsim |A|^{1+\frac{1}{105}}\\
&\label{eq:gs2}\text{If } A \subseteq \mathbb{F}_p \text{ with } |A|\geq p^{2/3} \text{ then }|f(A)|\gg |A|^{1/2}p^{1/2}\\
&\label{eq:gs3}\text{If } A \subseteq \mathbb{R} \text { is finite then } |f(A)|\gg |A|^{1+\frac{1}{4}}.
\end{align}

Result (\ref{eq:gs2}) is sharp but (\ref{eq:gs1}) and (\ref{eq:gs3}) are not. The first new result of this chapter is the following improvement on (\ref{eq:gs1}).

\begin{theorem}\label{theorem:twovariableFp}
If $A \subseteq \mathbb{F}_p$ and $|A|<p^{1/2}$ then
	$$\left|f(A)\right|\gtrsim |A|^{1+\frac{1}{53}}.$$
\end{theorem}

Result (\ref{eq:gs3}), has also been recently improved, in the same paper \cite{meandollie} as that containing an earlier version of Theorem \ref{theorem:twovariableFp}, to $|f(A)|\gtrsim |A|^{1+\frac{5}{19}}$.

\subsection{Three and four variables}

As mentioned in Chapter \ref{chapter:introgrowth}, the breakthrough result of Guth and Katz \cite{GK} on the distinct distances problem yields the function $$(a-b)^2+(c-d)^2$$
as a four-variable expander over $\mathbb{R}$ with $\delta \geq 1-o(1)$. Iosevich, Roche-Newton and Rudnev \cite{IRR} used the same methodology to show that 
	$$ad-bc$$
is likewise a four-variable expander with $\delta \geq 1-o(1)$. Both this and the Guth-Katz result are sharp, as shown by the case where $A$ is an arithmetic progression.

The next two new results in this chapter are examples of functions with stronger growth properties. The first is a function in only three real variables rather than four which is nevertheless an expander with $\delta \geq 1-o(1)$. 

\begin{theorem}\label{theorem:threevariableC}
Let $g(a,b,c)=\frac{a-b}{a-c}$. For any finite $A \subseteq \mathbb{C}$ we have
	$$|g(A)|\gtrsim |A|^2.$$ 
\end{theorem}

The second is a four-variable expander with $\delta \geq 1$ instead of $1-o(1)$.

\begin{theorem}\label{theorem:fourvariableR}
Let $h(a,b,c,d)=\frac{(a-b)(c-d)}{(b-c)(a-d)}$. For any finite $A \subseteq \mathbb{R}$ we have 
	$$|h(A)|\gg |A|^2.$$ 
\end{theorem}

As with the results of Guth and Katz and Iosevich, Roche-Newton and Rudnev, the example of an arithmetic progression shows that Theorem \ref{theorem:threevariableC} is sharp up to logarithmic factors. However it is not clear that the same is true of Theorem \ref{theorem:fourvariableR}: in the case of an arithmetic progression one has $|h(A)|\gtrsim |A|^3$ and so there may well be scope for improvement.

\subsection{Structure}

In the rest of this chapter, \textbf{Section \ref{section:finiteexp}} gives the proof of Theorem \ref{theorem:twovariableFp}, and \textbf{Section \ref{section:realexp}} gives the proofs of Theorems \ref{theorem:threevariableC} and \ref{theorem:fourvariableR}. \textbf{Section \ref{section:expanderfurther}} suggests possible further work.

\section{Two-variable finite field expanders}\label{section:finiteexp}

This section is concerned with proving Theorem \ref{theorem:twovariableFp}. The overall strategy is similar to the finite field case in the work of Garaev and Shen \cite{GS}, exploiting the fact that 
	$$f(A)=\left\{a+ab:a,b \in A\right\}$$
can be written as simply the product set $A(A+1)$. A sum-product philosophy suggests that at least one of $A(A+1)$ and the difference set $A-A$ must be large. If it is the former then of course there is nothing to prove, so the idea is to deal with the latter case. That is, we want to show that if the difference set is large then so is $A(A+1)$.

The innovation here is to take a more efficient route than \cite{GS}, via the simple Ruzsa-type observation that if $ab=cd$ then 
	$$(a+ab)-(c+cd)=a-c.$$

\textbf{Section \ref{section:partial}} establishes the new key result and shows three different ways of implementing it. \textbf{Section \ref{section:proof}} then deploys these in a modified sum-product type proof to establish Theorem \ref{theorem:twovariableFp}.

\subsection{Bounding partial sumsets}\label{section:partial}

The following result is the key to our progress.

\begin{lemma}\label{theorem:idea}
Let $A,B\subseteq \mathbb{F}_p$, and let $\epsilon>0$. There exists $G \subseteq A \times B$ with $|G|\geq (1-\epsilon)|A||B|$ such that 
$$|A\stackrel{G}{-} B|\ll_{\epsilon} \frac{|A(B+1)||B(A+1)||A/B|}{|A||B|}.$$
\end{lemma}

\begin{proof}
Without loss of generality assume that $0 \notin A,B$. Note that 
	$$\sum_{x \in A/B}\left|A \cap xB\right|=|A||B|.$$
Let $X$ be the set of $x \in A/B$ for which $|A \cap xB|\geq \frac{\epsilon|A||B|}{|A/B|}$. Then
	\begin{align*}
	|A||B| &= \sum_{x \in X}|A \cap xB|+\sum_{x \notin X}|A \cap xB|\\
	& \leq \sum_{x \in X}|A \cap xB| + \epsilon|A||B|
	\end{align*}
and so 
	$$\sum_{x \in X}|A \cap xB| \geq (1-\epsilon)|A||B|.$$ 
Let $G \subseteq A \times B$ be given by
	$$G=\left\{(a,b) \in A \times B:\frac{a}{b} \in X\right\}$$ 
so that 
	\begin{align*}
	|G|&=\sum_{x \in X}|A\cap xB|\geq (1-\epsilon) |A||B|.
	\end{align*}
For each $\xi \in A \stackrel{G}{-}B$ pick $a_{\xi} \in A $ and $b_{\xi}\in B$ such that 
	$$a_{\xi}-b_{\xi}=\xi.$$ 
Let $S \subseteq (A\overset{G}-B)\times A \times B$ be given by
	$$S=\left\{\left(\xi,c,d\right):\frac{c}{d}=\frac{a_{\xi}}{b_{\xi}}\right\}$$
Note that
	$$|S|\gg_{\epsilon} \frac{|A||B||A\overset{G}-B|}{|A/B|}$$
since there are $|A\stackrel{G}{-}B|$ choices of $\xi$, each of which has at least $\frac{\epsilon|A||B|}{|A/B|}$ associated pairs $(c,d)$. We now show that 
	$$|S| \leq \left|A(B+1)\right|\left|B(A+1)\right|.$$ 
This will follow after showing that the map 
	$$\psi:S \to A(B+1) \times B(A+1)$$ 
	$$\psi(\xi,c,d)= \left(a_{\xi}+a_{\xi}d,b_{\xi}+b_{\xi}c\right)$$
is an injection, i.e. that for given $(t_1,t_2)$ in $\psi(S)$ there is only one choice of $(\xi,c,d)\in S$ for which $\psi\left(\xi,c,d\right)=(t_1,t_2)$. Indeed for given $(t_1,t_2)$ in $\psi(S)$ it is immediate that
	\begin{align*}
	\xi& = a_{\xi}-b_{\xi}\\
	&=\left(a_{\xi}+a_{\xi}d\right)-\left(b_{\xi}+b_{\xi}c\right)\\
	&=t_1-t_2
	\end{align*}
so we know $\xi$ and therefore $a_{\xi}$ and $b_{\xi}$. We therefore also know $(c,d)$ since 
	\begin{align*}
	t_1&=a_{\xi}+a_{\xi}d\\
	t_2&=b_{\xi}+b_{\xi}c.
	\end{align*}
So $\psi$ is indeed an injection and the upper bound on $|S|$ holds.

Comparing the upper and lower bounds on $|S|$ shows that 
	$$\frac{|A||B||A\overset{G}-B|}{|A/B|} \ll_{\epsilon}|S| \leq \left|A(B+1)\right|\left|B(A+1)\right|$$
and so 
	$$|A\stackrel{G}{-}B|\ll_{\epsilon} \frac{|A(B+1)||B(A+1)||A/B|}{|A||B|}$$
as required.
\end{proof}

Let's now apply this, using sumset calculus from Chapter \ref{chapter:introapproxgroup}, to give upper bounds on $|A-A|$ in terms of $|A(A+1)|$. First, a Balog-Szemer\'edi-Gowers approach yields a quite-efficient relationship:

\begin{corollary}\label{theorem:horrific}
For any set $A \subseteq \mathbb{F}_p$ there exists $A'\subseteq A$ with $|A'|\approx |A|$ such that $\left|A'-A'\right|\ll \frac{|A(A+1)|^8}{|A|^7}.$
\end{corollary}

\begin{proof}
Apply Lemma \ref{theorem:idea} with $A=B$ and some fixed sufficiently small $\epsilon>0$ to find $G \subseteq A \times A$ with $|G|\geq (1-\epsilon)|A|^2$ such that 
	$$|A \overset{G}-A|\ll \frac{|A(A+1)|^2|A/A|}{|A|^2}.$$ 
By Lemma \ref{theorem:densebsg} there exists $A' \subseteq A$ with $|A'|\approx |A|$ such that 
	$$|A'-A'|\ll \frac{|A \overset{G}-A|^2}{|A|}\ll \frac{|A(A+1)|^4|A/A|^2}{|A|^5}.$$
By Lemma \ref{theorem:ruzsa} applied multiplicatively, $|A/A|\leq \frac{|A(A+1)|^2}{|A|}$ and so the result follows.
\end{proof}

Corollary \ref{theorem:horrific} is useful by itself, but we can do better in some places. Covering results are often used in sum-product proofs, and applying Lemma \ref{theorem:cover1} to Lemma \ref{theorem:idea} yields one that will be helpful here.

\begin{corollary}\label{theorem:energy}
Let $A,B,C \subseteq \mathbb{F}_p$, and $A,B\subseteq xC+y$ for some $x \in \mathbb{F}_p^*, y \in \mathbb{F}_p$. Let $0<\epsilon<\frac{1}{16}$. Then $(1-\epsilon)|A|$ elements of $A$ can be covered by $$O_{\epsilon}\left( \frac{|C(C+1)|^2|C/C|}{|A||B|^2}\right)$$ translates of $B$. Similarly, $(1-\epsilon)|A|$ elements of $A$ can be covered by this many translates of $-B$. 
\end{corollary}

\begin{proof}
Applying Lemma \ref{theorem:idea} to the sets $A_{xy}=\frac{A-y}{x} \subseteq C$ and $B_{xy}=\frac{B-y}{x} \subseteq C$ there exists $G_{xy}\subseteq A_{xy} \times B_{xy}$ of cardinality at least $\left(1-\frac{\epsilon^2}{4}\right)|A||B|$ such that 
	\begin{align*}|A_{xy}\overset{G_{xy}}-B_{xy}|&\ll \frac{|A_{xy}(B_{xy}+1)||B_{xy}(A_{xy}+1)||A_{xy}/B_{xy}|}{|A||B|}\\
	&\leq \frac{|C(C+1)|^2|C/C|}{|A||B|}.
	\end{align*}
Then let 
	$$G=\left\{(a,b):\left(\frac{a-y}{x},\frac{b-y}{x}\right)\in G_{xy}\right\}$$ 
to obtain 
	$$|A\overset{G}-B|=|A_{xy}\overset{G_{xy}}-B_{xy}|\ll \frac{|C(C+1)|^2|C/C|}{|A||B|}.$$
The result follows by applying Lemma \ref{theorem:cover1} to $G$, with $\epsilon$ replaced by $\epsilon^2/4$.
\end{proof}

Applying Lemma \ref{theorem:cover2} yields another helpful covering result.

\begin{corollary}\label{theorem:energy'}
Let $0<\epsilon<\frac{1}{2}$. Then there exists $G\subseteq A \times A$ with $|G| \geq (1-\epsilon)|A|^2$ such that $A \overset{G}-A$ is covered by $O_{\epsilon}\left(\frac{|A(A+1)|^2|A/A|}{|A|^3}\right)$ translates of $A$.
\end{corollary}

\begin{proof}
By Lemma \ref{theorem:idea}, with $\epsilon$ replaced by $\frac{\epsilon}{2}$, there exists $G' \subseteq A \times A$ with $|G'| \geq (1-\frac{\epsilon}{2})|A|^2$ such that
	$$|A \overset{G'}-A|\ll_{\epsilon} \frac{|A(A+1)|^2|A/A|}{|A|^2}.$$
Then by Lemma \ref{theorem:cover2} there exists $G \subseteq G'$ with 
	$$|G| \geq \left(1-\frac{\epsilon}{2}\right)|G'| \geq \left(1-\frac{\epsilon}{2}\right)^2|A|^2 \geq (1-\epsilon)|A|^2$$
such that $A\overset{G}-A$ is covered by $O_{\epsilon}\left(\frac{|A(A+1)|^2|A/A|}{|A|^3}\right)$	translates of $A$, as required.
\end{proof}

\subsection{Proving Theorem \ref{theorem:twovariableFp}}\label{section:proof}

This section uses Corollaries \ref{theorem:horrific}, \ref{theorem:energy} and \ref{theorem:energy'} to prove Theorem \ref{theorem:twovariableFp}. 

By Corollary \ref{theorem:horrific} and passing to a subset of $A$ if necessary we may assume that
	\begin{equation}\label{eq:finite1}
	|A-A|\ll \frac{|A(A+1)|^8}{|A|^7}.
	\end{equation}
By Corollary \ref{theorem:cleverpl} and again passing to a subset if necessary we may assume that
	\begin{equation}\label{eq:finite2}
	|A-A-A-A|\ll \frac{|A-A|^3}{|A|^2}.
	\end{equation}
Now, by Corollary \ref{theorem:energyformulation2} and Lemma \ref{theorem:energycs1},  
	$$\sum_{a,b \in A}\left|a(A+1) \cap b(A+1) \right|=E_{\times}(A,A+1)\geq \frac{|A|^4}{|A(A+1)|}.$$
So there exists $b_0 \in A$ such that
	$$\sum_{a \in A}\left|a(A+1) \cap b_0 (A+1)\right|\geq \frac{|A|^3}{|A(A+1)|}.$$
By dyadic pigeonholing there exists $A_1 \subseteq A$ and $N \in \mathbb{N}$ such that
	$$\left|a(A+1) \cap b_0 (A+1)\right|\approx N$$ 
for all $a \in A_1$ and 
	\begin{equation}\label{eq:finite3}
	N|A_1|\gtrsim \frac{|A|^3}{|A(A+1)|}.
	\end{equation}
Since $|A_1|\leq |A|$ this also implies 
	\begin{equation}\label{eq:finite4}
	N\gtrsim \frac{|A|^2}{|A(A+1)|}.
	\end{equation}

Now consider the set $$R(A_1)=\left\{\frac{\alpha-\beta}{\gamma-\delta}:\alpha,\beta,\gamma, \delta \in A_1, \alpha \neq \beta, \gamma \neq \delta \right\}$$ and break into two cases according to whether or not $R(A_1)=\mathbb{F}_p$.

\subsubsection{$R(A_1)\neq \mathbb{F}_p$}

The important thing about $R(A_1)$ is that if $\xi \notin R(A_1)$ then 
	\begin{equation}\label{eq:abuse}
	|A_1\overset{G}+\xi A_1|= |G|
	\end{equation}
for any $G \subseteq A_1 \times A_1$. Note that \eqref{eq:abuse} employs a slight abuse of notation: by $A_1\overset{G}+\xi A_1$ it means the set of $a+\xi b$ for which $(a,b)\in G$. It holds because there can be no repetition in $A_1+\xi A_1$, since if
	$$a+\xi b = c+\xi d$$
with $(a,b)\neq (c,d)$ then $\xi = \frac{a-c}{b-d} \in R(A_1)$ which is a contradiction. 

Now since $R(A_1)\neq \mathbb{F}_p$ there must exist $\xi=\frac{\alpha-\beta}{\gamma-\delta} \in R(A_1)$ such that $\xi-1 \notin R(A_1)$. For \textit{any} $G \subseteq A_1 \times A_1$ it follows that
	\begin{align}
	|G| &= |A_1 \overset{G}+(\xi-1)A_1|\nonumber\\
	&= \left|A_1 \overset{G}+ \frac{\alpha-\beta-\gamma + \delta}{\gamma-\delta}A_1\right|\nonumber\\
	&= |(\gamma - \delta)A_1\overset{G}+ (\alpha-\beta-\gamma + \delta)A_1|\label{eq:finite7}
	\end{align}
where there is another abuse of notation in \eqref{eq:finite7}. 

Now proceed to fix a \textit{particular} choice of $G$. Let $\epsilon>0$ be sufficiently small and for convenience write
	$$\lambda=\frac{|A(A+1)|^2|A/A|}{N^2|A_1|}.$$
Applying Corollary \ref{theorem:energy} to the sets $\alpha(A_1+1)$ and $b_0(A_1+1)\cap \alpha(A_1+1)$ shows that there is a set $A_{\alpha} \subseteq A_1$ with 		
	$$\left|A_{\alpha}\right|\geq (1-\epsilon)|A_1|$$
such that $\alpha A_{\alpha}$ is contained in the union of $O_{\epsilon}\left(\lambda\right)$ translates of $b_0 A$. Similarly, there are sets $A_{\beta},A_{\gamma},A_{\delta} \subseteq A_1$ with
	$$|A_{\beta}|,|A_{\gamma}|,|A_{\delta}|\geq (1-\epsilon)|A_1| $$
such that $\beta A_{\beta}$ and $\gamma A_{\gamma}$ are contained in the union of 
$O\left(\lambda\right)$ translates of $b_0 A$, and $\delta A_{\delta}$ is contained in the union of $O\left(\lambda\right)$ translates of $-b_0 A$. Let  
	$$A_2 = A_{\alpha}\cap A_{\beta}\cap A_{\gamma} \cap A_{\delta}$$
so that 
	$$|A_2|\geq (1-4\epsilon)|A_1|.$$ 
By Corollary \ref{theorem:energy'} there exists $G \subseteq A_2 \times A_2$ with 
	$$|G| \approx |A_2|^2 \approx |A_1|^2$$
such that $A_2 \overset{G}-A_2$ is covered by $O\left(\frac{|A(A+1)|^2|A/A|}{|A_1|^3}\right)$ translates of $A_2$. Fix this choice of $G$ so that from (\ref{eq:finite7}),
	\begin{align*}
	|A_1|^2 &\ll \left|(\gamma - \delta)(A_2\overset{G}-A_2)+(\alpha-\beta)A_2\right|\\
	&\ll \frac{|(\gamma-\delta)A_2-(\alpha-\beta)A_2||A(A+1)|^2|A/A|}{|A_1|^3}\\
	&\leq \frac{|\alpha A_2- \beta A_2 - \gamma A_2 + \delta A_2||A(A+1)|^2|A/A|}{|A_1|^3}.
	\end{align*}
Since 
	$$A_2 \subseteq A_{\alpha},A_{\beta},A_{\gamma},A_{\delta},$$
and $\alpha A_{\alpha}, \beta A_{\beta}, \gamma A_{\gamma}$
are each contained in the union of $O\left(\lambda\right)$ translates of $b_0 A$, and $\delta A_{\delta}$ is contained in the union of $O\left(\lambda\right)$ translates of $-b_0 A$, it follows that
	\begin{align}
	|A_1|^2&\ll \frac{|\alpha A_{\alpha}- \beta A_{\beta} - \gamma A_{\gamma} + \delta A_{\delta}||A(A+1)|^2|A/A|}{|A_1|^3}\label{eq:sameas}\\
	&\ll \frac{\lambda^4|b_0A-b_0A-b_0A-b_0A||A(A+1)|^2|A/A|}{|A_1|^3}\nonumber\\
	&= \frac{|A-A-A-A||A(A+1)|^{10}|A/A|^5}{N^8|A_1|^7} \nonumber.
	\end{align}
By \eqref{eq:finite2}, this gives 
	\begin{equation}\label{eq:finite5}
	|A_1|^2\ll \frac{|A-A|^3|A(A+1)|^{10}|A/A|^5}{N^8|A_1|^7|A|^2}.
	\end{equation}
Now by the Ruzsa triangle inequality (Lemma \ref{theorem:ruzsa}) applied multiplicatively, 
	\begin{equation}\label{eq:finite6}
	|A/A|\leq \frac{|A(A+1)|^2}{|A|}.
	\end{equation}
Applying (\ref{eq:finite1}) and (\ref{eq:finite6}) to (\ref{eq:finite5}) yields
	\begin{align*}
	|A_1|^2&\ll \frac{|A(A+1)|^{44}}{N^8|A_1|^7|A|^{28}}.
	\end{align*}
Rearranging and applying (\ref{eq:finite3}) and (\ref{eq:finite4}) gives
	$$|A(A+1)|^{44}\gg |A_1|^9 N^8 |A|^{28} \gtrsim \frac{|A_1| |A|^{52}}{|A(A+1)|^8} \gtrsim \frac{|A|^{54}}{|A(A+1)|^9}$$
and so 
	$$|A(A+1)|\gtrsim |A|^{54/53}$$ 
as required.

\subsubsection{$R(A_1)= \mathbb{F}_p$}

Let $E$ be the number of solutions to 
	\begin{equation}\label{eq:energyeq}
	a+\xi b = c+\xi d 
	\end{equation}
with $a,b,c,d \in A_1$ and $\xi \in R(A_1)$. Moreover, for each $\xi \in R(A_1)$ recall that $E_+(A_1,\xi A_1)$ is the additive energy of $A_1$ and $\xi A_1$, i.e. the number of solutions to (\ref{eq:energyeq}) with $\xi$ fixed, so that 
	$$E=\sum_{\xi \in R(A_1)}E_+(A_1,\xi A_1).$$
	
There are no more than $|R(A_1)||A_1|^2=p|A_1|^2$ solutions to (\ref{eq:energyeq}) for which $(a,b)=(c,d)$. And there are no more than $|A_1|^4$ solutions with $(a,b)\neq(c,d)$. So in total
	$$\sum_{\xi \in R(A_1)}E_+(A_1,\xi A_1)=E\leq |A_1|^4 +p|A_1|^2.$$
Since $|A_1|\leq |A|<p^{1/2}$ this gives
	$$\sum_{\xi \in R(A_1)}E_+(A_1,\xi A_1)\ll p|A_1|^2.$$
So there exists $\xi = \frac{\alpha-\beta}{\gamma-\delta} \in R(A_1)$ such that 
	$$E_+(A_1,\xi A_1)\ll |A_1|^2.$$
Moreover, for any $A_2 \subseteq A_1$ we have also
	$$E_+(A_2,\xi A_2)\ll |A_1|^2.$$
Now by Lemma \ref{theorem:energycs1}, 
	$$E_+(A_2,\xi A_2)\geq \frac{|A_2|^4}{|A_2-\xi A_2|}$$
and so if $|A_2|\approx |A_1|$ then
	\begin{align*}
	|\alpha A_2- \beta A_2 - \gamma A_2 + \delta A_2|&\geq \left|A_2- \frac{\alpha-\beta}{\gamma-\delta} A_2\right|\\
	&= |A_2- \xi A_2|\\
	&\gg |A_1|^2.
	\end{align*}
Let $A_\alpha, A_{\beta},A_{\gamma}, A_{\delta}$ be as before and fix 
	$$A_2=A_{\alpha}\cap A_{\beta} \cap A_{\gamma} \cap A_{\delta}.$$
This yields the same situation as considered at \eqref{eq:sameas} in the $R(A_1)\neq \mathbb{F}_p$ case, but with one less factor of $\frac{|A(A+1)|^2|A/A|}{|A_1|^3}$ to deal with. So we obtain (and in fact exceed) the required bound. \qed

\section{Three and four-variable real and complex expanders}\label{section:realexp}

This section is concerned with the proofs of Theorems \ref{theorem:threevariableC} and \ref{theorem:fourvariableR}.

Let's first place these in the context of the Guth-Katz proof on distinct distances, which led to the example of the function $(a-b)^2+(c-d)^2$ as a four-variable expander. This was based on a framework of Elekes and Sharir \cite{ES}. The idea, when counting the number of distinct objects determined by a set, is to analyse functions under which that object is invariant. 

In the Guth-Katz proof, where the objective is to count distances, the approach is to analyse the group $SE_2$ of orientation-preserving isometries. This analysis can be parameterised as an incidence problem of points and lines in $\mathbb{R}^3$. The usual Szmer\'edi-Trotter theorem is too weak to be of help directly, but Guth and Katz were able to amplify it to Theorem \ref{theorem:gk}, which they developed using a novel `polynomial partitioning' technique specifically for this purpose. 

The approach of Iosevich, Roche-Newton and Rudnev for showing that $ad-bc$ is a four-variable expander follows the same approach, but uses the fact that this function can be viewed as a determinant. Thus the approach is to analyse determinant-preserving maps, i.e. elements of $SL_2$. Like the Guth-Katz result, this too required an application of Theorem \ref{theorem:gk}.

The approach for proving Theorems \ref{theorem:threevariableC} and \ref{theorem:fourvariableR} is likewise founded on the Elekes-Sharir paradigm. Whereas Guth and Katz counted distances by analysing isometries from the group $SE_2$, and Iosevich, Roche-Newton and Rudnev counted determinants by analysing $SL_2$, the functions $g$ and $h$ considered here are instances of \textbf{cross ratios}, which are preserved by the group $PSL_2$ of projective transformations of the line.

There are several advantages to working with cross ratios: 
\begin{itemize}
\item As shown in Section \ref{section:expandersresults}, we are able to prove quantitatively stronger expander results.
\item A smaller arsenal is required. Neither Theorem \ref{theorem:gk} nor any other application of the Guth-Katz polynomial partitioning technique is required. Instead only the  Szemer\'edi-Trotter theorem on points and lines is used when proving Theorem \ref{theorem:threevariableC} and only the Edelsbrunner-Guibas-Sharir theorem on points and planes is used when proving Theorem \ref{theorem:fourvariableR}. 
\item Parameterising as an incidence problem is more straightforward. The Guth and Katz and Iosevich, Roche-Newton and Rudnev proofs go through a certain amount of hassle in order to make the parameterisation and verify that appropriate nondegeneracy conditions are satisfied. But with cross ratios and projective transformations everything falls out naturally.
\end{itemize}

In what follows, \textbf{Section \ref{section:crossratio}} describes the standard theory of cross ratios, establishing that they are invariants of projective transformations. \textbf{Section \ref{section:embedding}} then identifies projective transformations in a natural way with points in three-dimensional projective space, and establishes how the transformations' behaviour corresponds to line and plane structures of points. \textbf{Section \ref{section:threeandfourvariableproof}} then uses the results of the preceding two sections to prove Theorems \ref{theorem:threevariableC} and \ref{theorem:fourvariableR}.

\subsection{Cross ratios}\label{section:crossratio}

This section records some standard theory on cross ratios. 

Recall, or consult Appendix \ref{chapter:proj}, that for a field $F$ the projective line $\mathbb{P}F^1$ is the set of equivalence classes $[x,y]$ of $F^2 \setminus \left\{(0,0)\right\}$ modulo dilation, and that it can more concretely be viewed as the extended line $\overline{F}=F \cup \left\{\infty\right\}$ by identifying $[1,0]$ with $\infty$ and $[x,1]$ with $x \in F$. Under this identification, the \textbf{cross ratio} $X(a,b,c,d) \in \overline{F}$ of four elements $a,b,c,d \in \overline{F}$ is given by 
	$$X(a,b,c,d)=\frac{(a-b)(c-d)}{(b-c)(a-d)}$$ 
interpreted in the sense of limits where necessary. Note that the functions $g$ and $h$ from Theorems \ref{theorem:threevariableC} and \ref{theorem:fourvariableR} are given by 
	\begin{align*}
	g(a,b,c)&= X\left(\infty,a,b,c\right)\\
	h(a,b,c,d)&= X\left(a,b,c,d\right)
	\end{align*}
So to prove the theorems it suffices respectively to show that
	\begin{align*}
	\#\left\{X(\infty,a,b,c):a,b,c \in A\right\} \gg |A|^{2-o(1)} \text{ for any finite }A \subseteq \mathbb{C}\\
 	\#\left\{X(a,b,c,d):a,b,c,d \in A\right\} \gg |A|^2 \text{ for any finite }A \subseteq \mathbb{R}.
	\end{align*}

The importance of the cross ratio is that it is a projective invariant of quadruples, in the sense of the following result which can be found in \cite{PS}, for example.

\begin{lemma}\label{theorem:invariant}
Let $a_i \in \overline{F}$ be distinct for $i=1,2,3,4$ and the same for $b_i \in \overline{F}$. Then
$X(a_1,a_2,a_3,a_4)=X(b_1,b_2,b_3,b_4)$ if and only if there is a projective transformation in $PSL_2(F)$ that sends each $a_i$ to $b_i$.
\end{lemma}

\begin{proof}
By Lemma \ref{theorem:frametransitive} there is a unique projective transformation $\mu \in PSL_2(F)$ that sends $a_i$ to $b_i$ for $i=1,2,3$. We shall show that $X(a_1,a_2,a_3,a_4)=X(b_1,b_2,b_3,b_4)$ if and only if $\mu$ also sends $a_4$ to $b_4$.

First note that $X(a,b,c,d)=\tau_{abc}(d)$ where $\tau_{abc} \in PSL_2(F)$ is the unique projective transformation that sends $\left(a,b,c\right)$ to $\left(\infty,1,0\right)$. To see this it suffices simply to check that 
	$$\tau_{abc}=\left[
	\left(
	\begin{array}{lr}
	b-a & (a-b)c\\
	c-b & (b-c)a
	\end{array}
	\right)\right]$$
and then that $\tau_{abc}(d)=X(a,b,c,d)$.

It follows that 
	$$X(a_1,a_2,a_3,a_4)=X(b_1,b_2,b_3,b_4)$$ 
if and only if 
	\begin{equation}\label{eq:crossratiopf}
	\tau_{a_1 a_2 a_3}(a_4)=\tau_{b_1b_2b_3}(b_4).
	\end{equation}
Now by definition of $\mu$ we have 
	$$\tau_{a_1a_2a_3}=\tau_{b_1b_2b_3}\circ \mu$$ 
and so by injectivity of $\tau_{b_1 b_2 b_3}$, equation \eqref{eq:crossratiopf} holds precisely when $\mu(a_4)=b_4.$ This completes the proof.
\end{proof}

\subsection{Points, planes and transformations}\label{section:embedding}

This section contains two results that apply to any field $F$. The first, a `points lemma', identifies projective transformations from $PSL_2(F)$ with points in $\mathbb{P}F^3$. The second, a `planes lemma', establishes that the behaviour of transformations corresponds to line and plane structures of their associated points.

\begin{lemma}[Points lemma]\label{theorem:points}
Define $\psi:PSL_2(F)\to \mathbb{P}F^3$ by 
	$$\psi\left[\left(
	\begin{array}{cc}
	p&q\\
	r&s	
	\end{array}
	\right)\right]
	=[p,q,r,s].$$
The map $\psi$ is well-defined and injective, and its image is $\mathbb{P}F^3\setminus Q$ where $Q$ is the quadratic surface given by $ps=qr$.
\end{lemma}

\begin{proof}
That $\psi$ is well-defined and injective follows from checking that if $t_1,t_2 \in SL_2(F)$ then $\psi[t_1]=\psi[t_2]$ if and only if $t_1=\pm t_2$. That the image is $\mathbb{P}F^3\setminus Q$ follows from the definition $PSL_2(F)=SL_2(F)/\pm I$. 
\end{proof}

\begin{lemma}[Planes lemma]\label{theorem:planes}
Let $\psi$ be as in the points lemma. For each $(a,b)\in \overline{F}\times \overline{F}$ there is a plane $\pi_{ab}\subseteq \mathbb{P}F^3$ with the following properties. 
	\begin{enumerate}
	\item If $\tau \in PSL_2(F)$ then $\tau(a)=b$ if and only if $\psi(\tau) \in \pi_{ab}$.
	\item No three planes are collinear.
	\item Different pairs $(a,b)\in\overline{F}\times \overline{F}$ determine different planes $\pi_{ab}$.
	\item Different pairs of planes $\left\{\pi_{ab},\pi_{cd}\right\}$ intersect in different lines $\pi_{ab}\cap \pi_{cd}$.
	\item For any $A \subseteq F$, a point $p \in \mathbb{P}F^3\setminus Q$ is incident to at most $|A|$ of the planes 		
	from $\left\{\pi_{ab}:a,b \in A\right\}$. 
	\end{enumerate}
\end{lemma}

\begin{proof}
From the theory in Appendix \ref{chapter:proj}, a projective transformation $\tau = \left[\left(
	\begin{array}{cc}
	p&q\\
	r&s
	\end{array}
	\right)\right]$
sends $a$ to $b$ if and only if 
	$$\frac{ap+q}{ar+s}=b,$$
which is the same as 
	$$ap+q-bar-bs=0.$$ 
For fixed $a,b$ this is an homogeneous linear constraint on $\psi(\tau)=[p,q,r,s]\in \mathbb{P}F^3$ and so describes a plane in $\mathbb{P}F^3$, which we define to be $\pi_{ab}$. Property $1$ is satisfied by construction, and it is now straightforward to establish properties $2$ to $5$ in turn
	\begin{enumerate}
	\setcounter{enumi}{1}
	\item It suffices to show that three planes intersect in a point. Let $(a,b,c)$ and $(d,e,f)$ be two triples of distinct elements of $F$. By Lemma \ref{theorem:frametransitive} there is a unique $\tau\in PSL_2(F)$ that sends $(a,b,c)$ to $(d,e,f)$. So 
	$$\pi_{ad}\cap \pi_{be} \cap \pi_{cf}=\psi(\tau)$$ which is a single point in $\mathbb{P}F^3$.
	\item If $\pi_{ab}=\pi_{cd}$ for some $(a,b)\neq (c,d)$ then $\pi_{ab}\cap \pi_{cd} \cap 	
	\pi_{ef}$ is either a line or a plane for any third pair $(e,f)$, which contradicts property $2$.
	\item Suppose that 
		$$\pi_{ab}\cap \pi_{cd}= \pi_{a'b'}\cap \pi_{c'd'}.$$
	Then $$\pi_{ab}\cap \pi_{cd}\cap \pi_{a'b'} = \pi_{a'b'}\cap \pi_{c'd'}.$$ But by property $2$ the set on the left hand side is a point, whereas that on the right is a line, unless $\pi_{ab} \in \left\{\pi_{a'b'},\pi_{c'd'}\right\}$. Similarly a contradiction follows unless $\pi_{cd} \in \left\{\pi_{a'b'},\pi_{c'd'}\right\}$ 
	\item Let $p$ be a point in $\mathbb{P}F^3\setminus Q$, so that $p = \psi(\tau)$ for some $\tau \in PSL_2(F)$. For 	
	each $a \in A$ there is at most one $b \in A$ for which $p$ is incident to $\pi_{ab}$, as otherwise $\tau(a)$ would take two 	
	different values. Counting over all $a \in A$ shows that $p$ is incident to at most $|A|$ planes.
	\end{enumerate}
\end{proof}

\subsection{Proving Theorems \ref{theorem:threevariableC} and \ref{theorem:fourvariableR}}\label{section:threeandfourvariableproof}

This section uses the results from Sections \ref{section:crossratio} and \ref{section:embedding} to prove Theorems \ref{theorem:threevariableC} and \ref{theorem:fourvariableR}. We first give the proof of Theorem \ref{theorem:threevariableC}, which uses the Szemer\'edi-Trotter theorem.

\begin{proof}[Proof of Theorem \ref{theorem:threevariableC}.]
Since 
	$$|g(A)|=\#\left\{X(\infty,a,b,c):a,b,c \in A\right\},$$
we want to show that 
	$$\#\left\{X(\infty,a,b,c):a,b,c \in A\right\}\gg |A|^{2-o(1)} $$
for any finite $A \subseteq \mathbb{C}$. To this end write $E(A)$ for the number of solutions to the equation
	\begin{equation}\label{eq:expanderenergy1}
	X(\infty,a_1,a_2,a_3)=X(\infty,b_1,b_2,b_3)
	\end{equation}
with each of the $a_i$ and $b_i$ in $A$. Write $\mu(x)$ for the number of $a_1,a_2,a_3 \in A$ with $X(\infty,a_1,a_2,a_3)=x$. Then 
	$$\sum_{x \in g(A)}\mu(x)\approx|A|^3$$ 
and Cauchy-Schwarz implies that
	$$|A|^6 \approx \left(\sum_{x \in g(A)}\mu(x)\right)^2 \leq |g(A)|E(A).$$
So it suffices to show 
	$$E(A)\ll |A|^{4+o(1)}.$$ 

By Lemma \ref{theorem:invariant}, equation (\ref{eq:expanderenergy1}) is satisfied precisely when there exists $\tau \in PSL_2(\mathbb{C})$ that fixes $\infty$ and sends each $a_i$ to $b_i$. So if we define 
	$$T=\bigcup_{a,b \in A}\left\{\tau:\tau(\infty)=\infty,\tau(a)=b\right\}$$
and write $N(\tau)$ for the number of $(a,b)\in A^2$ for which $\tau(a)=b$, then
	\begin{align}
	E(A)& =\sum_{a_1,a_2,a_3 \in A}\sum_{b_1,b_2,b_3 \in A}\mathds{1}\left(X(\infty,a_1,a_2,a_3)=X(\infty,b_1,b_2,b_3)\right)\nonumber\\
	&\leq \sum_{a_1,a_2,a_3 \in A}\sum_{b_1,b_2,b_3 \in A}\sum_{\tau \in T} \mathds{1} \left(\tau(a_i)=b_i \text{ for each } i\right) \nonumber\\
	&\leq \sum_{\tau \in T}N(\tau)^3 \label{eq:tripint1a}.
	\end{align}
	
Let $\psi$ be as in the points lemma. Define a set $P$ of points by 
	$$P=\psi(T)$$
and a set of $L$ lines by 
	$$L=\left\{\pi_{ab}\cap \pi_{\infty \infty}:a,b \in A\right\}$$
so that $|L|\approx |A|^2$. The points and lines all lie in the plane $\pi_{\infty \infty}$. Moreover, if we write $m(p)$ for the number of lines from $L$ incident to a point $p \in P$, then
  \begin{align}
  N(\tau)&=\#\left\{(a,b): \tau(a)=b\right\}\nonumber\\
  &=\#\left\{(a,b): \psi(\tau) \in \pi_{ab}\cap \pi_{\infty \infty}\right\}\nonumber\\
  &=m(\psi(\tau))\label{eq:tripint1b}.
  \end{align}
Combining \eqref{eq:tripint1a} and \eqref{eq:tripint1b} gives
	\begin{equation}
	\label{eq:expanderenergy1b}
	E(A) \leq \sum_{p \in P}m(p)^3.
	\end{equation}

For each $j \in \mathbb{N}$ write $P_j$ for the set of $p \in P$ with $m(p) \in [2^j,2^{j-1})$. Applying the complex Szemer\'edi-Trotter theorem in the form of Corollary \ref{theorem:toth'} gives
	$$|P_j|\ll \frac{|L|^2}{2^{3j}}+\frac{|L|}{2^j} $$
and so
	\begin{align*}
	E(A)&\ll \sum_{j=0}^{\log |A|}|P_j|2^{3j}\\
	&\ll \sum_{j=0}^{\log |A|}
	\left(\frac{|L|^2}{2^{3j}}+\frac{|L|}{2^j}\right)2^{3j}\\
	&\approx |A|^4 \log |A|
	\end{align*}
as required.
\end{proof}

We now give the proof of Theorem \ref{theorem:fourvariableR}, which uses the Edelsbrunner-Guibas-Sharir theorem.

\begin{proof}[Proof of Theorem \ref{theorem:fourvariableR}.]
This time we want to show that 
	$$|h(A)|=\#\left\{X(a_1,a_2,a_3,a_4):a_i \in A\right\}\gg |A|^2$$
for any finite $A \subseteq \mathbb{R}$. To this end, we this time write $E(A)$ for the number of solutions to the equation
	\begin{equation}\label{eq:expanderenergy2}
	X(a_1,a_2,a_3,a_4)=X(b_1,b_2,b_3,b_4)
	\end{equation}
with the $a_i,b_i \in A$. Using Cauchy-Schwarz as in Theorem \ref{theorem:threevariableC} shows that 
	$$|h(A)|\gg \frac{|A|^8}{E(A)}$$ 
so it suffices to show that 
	$$E(A)\ll |A|^6.$$ 

Equation (\ref{eq:expanderenergy2}) is satisfied precisely when there exists $\tau \in PSL_2(\mathbb{R})$ that sends $a_i$ to $b_i$ for each $i$. Define
	$$T=\bigcup_{a,b \in A}\left\{\tau:\tau(a)=b\right\}$$
and write $N(\tau)$ for the number of $(a,b)\in A^2$ for which $\tau(a)=b$. Then
	\begin{equation*}
	\label{eq:expanderenergy2a}
	E(A) \ll \sum_{\tau \in T}N(\tau)^4.
	\end{equation*}
	
Let $\psi$ be as in the points lemma. Define a set $P \subseteq \mathbb{PR}^3$ of points by 
	$$P=\psi(T)$$
and a set $\Pi$ of planes by 
	$$\Pi=\left\{\pi_{ab}:a,b \in A\right\}$$ 
so that $|\Pi|\approx |A|^2$. If we write $m(p)$ for the number of planes from $\Pi$ incident to a point $p$, then 
	$$N(\tau)=m(\psi(\tau)).$$
So, following the same argument as for \eqref{eq:expanderenergy1b} in the preceding proof,
	\begin{equation*}
	\label{eq:expanderenergy2b}
	E(A) \leq \sum_{p \in P}m(p)^4.
	\end{equation*}
For each $j \in \mathbb{N}$ write $P_j$ for the set of $p \in P$ with $m(p) \in [2^j,2^{j+1})$. Then applying the Edelsbrunner-Guiber-Sharir theorem on points and planes in the form of Corollary \ref{theorem:EGS'} gives
	\begin{align*}
	E(A)&\ll \sum_{j=0}^{\log |A|}|P_j|2^{4j}\\
	&\ll \sum_{j=0}^{\log |A|} 	
	\left(\frac{|\Pi|^3}{2^{5j}}+\frac{|\Pi|}{2^j}\right)2^{4j}\\
	&\ll 	|A|^6 \sum_{j=0}^{\infty} \frac{1}{2^{j}}\\
	& \approx |A|^6 
	\end{align*}
as required.
\end{proof}

\section{Further work}\label{section:expanderfurther}
\begin{itemize}
\item \textbf{Where else could Ruzsa-type observations lead?} The Ruzsa-type observation that 
	$$(a+ab)-(c+cd)=a-c$$
whenever $ab=cd$ leads to a doubling of the growth exponent for the function $a+ab$ over finite fields. Could similar observations enable the construction of other expander functions, or give improved bounds on existing ones like $a+\frac{1}{b}$, $a+b^2$, or $a^2+ab$? 

\item \textbf{What more can be proved using cross ratios?} Cross ratios turned out to be very useful when constructing expanders in Theorems \ref{theorem:threevariableC} and \ref{theorem:fourvariableR}. For example, Theorem \ref{theorem:threevariableC} is a sharp three-variable result whose bound of $|g(A)|\gtrsim |A|^{2}$ is as strong as the best previously known four-variable results. What more can be said?
	
As a start, it seems reasonable to conjecture that the four-variable Theorem \ref{theorem:fourvariableR} can be strengthened to $|h(A)|\gtrsim |A|^3$. This would bring the theorem into line with the case where $A$ is an arithmetic progression, which is a sharp example for Theorem \ref{theorem:threevariableC}. 

In terms of applications to other growth results, might the strength of cross-ratio estimates enable new things to be said about two-variable expanders, or about sum-product estimates?
\end{itemize}

%% file: functionfield.tex
\chapter{A sum-product theorem in function fields}
\label{chapter:functionfield}

This chapter proves a sum-product theorem in the function field $\ffield$, showing that if $A$ is a finite subset of $\ffield$ then $$\max \left\{|A+A|,|AA|\right\}\gg_q |A|^{1+\frac{1}{5}-o(1)}.$$

The exponent of $\frac{1}{5}-o(1)$ lies between the $\frac{1}{3}-o(1)$ known in the complex setting and the $\frac{1}{11}-o(1)$ known in the finite field setting. This reflects the fact that unlike finite fields, $\ffield$ has an associated norm and topology, but that this norm is very different from that on $\mathbb{C}$, with an unusually rigid `non-archimedean' geometry. 

The material in this chapter is joint work with Thomas Bloom.

\section{Results}
Recall that $\mathbb{F}_q$ denotes the finite field of order $q$, where $q=p^{\alpha}$ is a prime power. The prime $p$ is the \textbf{characteristic} of $\mathbb{F}_q$, i.e. the least $n$ such that
	$$\underbrace{x+\ldots+x}_{n}=0$$
for all $x \in \mathbb{F}_q$. Throughought this chapter the letter $p$ will be reserved for this characteristic. The \textbf{function field} $\mathbb{F}_q(t)$ is the field of rational functions of a transcendental element $t$ over $\mathbb{F}_q$. Elements are therefore of the form
$$x=\sum_{j=-\infty}^N x_j t^j$$
where the $x_j$ are elements of $\mathbb{F}_q$. Note that this means that although $\mathbb{F}_q$ is finite, the function field $\ffield$ is not.

As with finite fields, it is necessary to rule out the possibility of finite subfields to be able to say anything non-trivial about growth in $\ffield$. But unlike finite fields, this is accomplished by a \textit{minimum} rather than a \textit{maximum} condition on the cardinality of a set $A$. The reason is that the only finite subfields of $\mathbb{F}_q(t)$ are $\mathbb{F}_q$ and its subfields: to rule these out it suffices to insist that $A$ is a bit bigger than $q$. Think therefore of $q$ as being small and fixed, as opposed to finite fields where it is taken to be very large indeed. A convenient way of capturing this necessity is to introduce an implicit dependency on $q$ when formulating sum-product estimates, so that results are of the form
	$$\max\left\{|A+A|,|AA|\right\}\gg_q |A|^{1+\delta} $$
for an absolute $\delta>0$ and any finite $A\subseteq \ffield$. 

As mentioned in Chapter \ref{chapter:introgrowth}, Li and Roche-Newton \cite{LiORN1} obtained a sum-product estimate $\delta \geq \frac{1}{11}-o(1)$ for a finite field $\mathbb{F}_q$ whose order is not necessarily prime, extending the applicability of Rudnev's result for $\mathbb{F}_p$. Because of its combinatorial generality this proof should go through in the function field setting without additional complication.

It is possible to do better by developing techniques specific to function fields. This chapter proves the following theorem 

\begin{theorem}\label{theorem:functionfieldsumproduct}
If $A \subseteq \ffield$ is finite then
	$$|A+A|^3|AA|^2 \gg_{q} |A|^{6-o(1)}.$$
\end{theorem}

A sum-product result for function fields with $\delta \geq \frac{1}{5}-o(1)$ follows immediately.

\begin{corollary}
If $A \subseteq \ffield$ is finite then
	$$\max \left\{|A+A|,|AA|\right\}\gg_{q} |A|^{1+\frac{1}{5}-o(1)}.$$
\end{corollary}

The next section provides more background on function fields, and explains the structure of the proof and the rest of the chapter. 
 
\section{Function fields}\label{section:ff}

This section has two parts. The first part, Section \ref{section:ff1}, gives some standard background on the geometry of function fields. The second part, Section \ref{section:ff2}, explains how this will be used in the proof of Theorem \ref{theorem:functionfieldsumproduct} and describes how the rest of the chapter is organised. 

\subsection{Background}\label{section:ff1}

The field $\ffield$ has a norm or valuation $|\cdot|$ given by 
	$$|x|=
	\left\{
	\begin{array}{ll}
	q^{\deg(x)}& x \neq 0\\
	0& x =0
	\end{array}
	\right.
	$$
where $\deg(x)$ is the \textbf{degree} of $x$, i.e. the maximal $j$ for which $x_j$ is non-zero. This valuation has the \textbf{non-archimedean} property that
	$$|x+y|\leq \max \left\{|x|,|y|\right\}$$
which is stronger than the usual triangle inequality. As a consequence $\mathbb{F}_q(t)$ has an unusually rigid geometry, which  will be exploited when proving sum-product estimates. A particular concern will be the behaviour of \textbf{balls}, which are as usual sets of the form
	$$B(x,r)=\left\{y \in \ffield:|x-y|\leq r\right\}.$$
In $\ffield$, the non-archimedean property implies the following fact, which is considered to be standard.

\begin{lemma}\label{theorem:nonarchball}
If $B_1$ and $B_2$ are balls in $\ffield$ then either they are disjoint, or $B_1 \subseteq B_2$, or $B_2 \subseteq B_1$. If in addition $B_1$ and $B_2$ have the same radius then either they are disjoint or $B_1=B_2$.
\end{lemma}

\begin{proof}
Let $B_1=B(x,r)$ and $B_2=B(y,s)$. If there exists $a \in B(x,r) \cap B(y,s)$ then 
	$$|x-y|\leq \max\left\{|a-x|,|a-y|\right\}\leq \max\left\{r,s\right\}.$$ 
If $r \leq s$ then this implies $B(x,r)\subseteq B(y,s)$ since if $b \in B(x,r)$ then 
	$$|y-b|\leq \max\left\{|y-x|,|b-x|\right\}\leq \max\left\{r,s\right\}=s.$$
Conversely if $s \leq r$ then $B(y,s)\subseteq B(x,r)$. Hence if $r=s$ then $B(x,r)=B(y,s)$.
\end{proof}

%\begin{lemma}\label{theorem:grouptranslate}
%Let $A \subseteq \ffield$ and let $r,k$ be integers. If $A$ is contained in a ball of radius $r$, then the $k$-fold sumset $kA$ is also contained in a ball of radius $r$.
%\end{lemma}
%
%\begin{proof}
%This is an immediate consequence of the fact that any ball in $\mathbb{F}_q(t)$ is a translation of an additive subgroup.
%\end{proof}

\subsection{Sum-product considerations}
\label{section:ff2}

The $\ffield$ sum-product proof builds upon an approach of Solymosi \cite{solymosi3} for sum-products in $\mathbb{C}$. When adapting this method, the non-archimedean geometry of \(\ffield\) turns out to be a mixed blessing. 

First, the bad news. Solymosi's argument fails at a critical point in the \(\ffield\) setting, for the following reason. For each $a \in A$, let $a' \in A \setminus\left\{a\right\}$ be such that $|a-a'|$ is minimal, and let $B_a$ be the ball of radius $|a-a'|$ centred on $a$. Solymosi's method uses the crucial fact that a single complex number can be contained in at most $O(1)$ of the $B_a$. This fails spectacularly in \(\ffield\): in this setting an element could be contained in as many as $|A|$ of the $B_a$, as demonstrated by the example 
	$$A=\left\{t^j:0 \leq j \leq n\right\}$$
where
	$$B_{t^j}=\left\{x \in \ffield:|x|\leq q^j\right\}$$
for $j\geq 1$ and $B_{1}=B_{t}$, meaning that every one of the $|A|$ balls contains $0$ as an element. 

But all is not lost. In the example above, the astute reader will notice that $|A+A|\approx |A|^2$, and so a strong-sum product estimate holds despite the failure of Solymosi's argument. In fact we will be able to show that something like this is possible whenever the Solymosi argument fails, by defining and considering \textbf{separable sets}.

Say that a set $A$ is \textbf{separable} if its elements can be indexed as
	$$A=\left\{a_1,\ldots,a_{|A|}\right\}$$ 
in such a way that for any $1\leq j \leq |A|$ there is a ball $B_j$ with
	$$A \cap B_j = \left\{a_1,\ldots,a_j\right\}.$$

Separability is fairly unexciting in the complex setting, but in the non-archimedean regime of $\ffield$ it is a stronger notion. The rigid geometry makes it harder to find separable sets, but where they do exist it will in fact imply the existence of large sumsets. The idea, therefore, is to show that a large separable sets must exist whenever the Solymosi approach fails. Combining this with an analysis of separable sets as having large sumsets will lead to a proof of Theorem \ref{theorem:functionfieldsumproduct}. 

In what follows, \textbf{Section \ref{section:separable}} analyses separable sets and develops the necessary results about their sumsets. \textbf{Section \ref{section:solymosi}} then adapts Solymosi's proof from \cite{solymosi3} to establish that if $|A+A|$ and $|AA|$ are both small then $A$ must contain a large separable set. \textbf{Section \ref{section:functionfieldproof}} uses these results to prove Theorem \ref{theorem:functionfieldsumproduct}. \textbf{Section \ref{section:functionfieldfurther}} considers some directions for further work.

\section{Separable sets}\label{section:separable}

This section analyses sumsets of separable sets. Recall that a set $A \subseteq \ffield$ is \textbf{separable} if its elements can be indexed as
	$$A=\left\{a_1,\ldots,a_{|A|}\right\}$$
in such a way that for each $1 \leq j \leq |A|$ there is a ball $B_j$ with 
	$$A \cap B_j =\left\{a_1,\ldots,a_j\right\}.$$  
Say that the balls $B_j$ \textbf{separate} $A$. It is an immediate consequence of the definition that a subset of a separable set is itself separable.

As in Chapter \ref{chapter:introapproxgroup} write $kA$ for the $k$-fold sumset of $A$. The following lemma shows that if $A$ is separable then  $kA$ has essentially maximum-possible cardinality.  

\begin{lemma}\label{theorem:separablegrowth}
If $A$ is separable then 
	$$|k A|\gg_k|A|^{k}$$ 
for any natural number $k$.
\end{lemma}

\begin{proof}
Let $E_k(A)$ denote the $k$-fold additive energy of $A$, i.e. the number of solutions to
	\begin{equation}\label{eq:kenergy}
	a_1+\ldots+a_k=b_1+\ldots+b_k
	\end{equation}
with the $a_i,b_i \in A$. For $x \in kA$ write $\mu(x)$ for the number of solutions to $x=a_1+\ldots+a_k$. By Cauchy-Schwarz as in Lemma \ref{theorem:energycs1},
	$$|A|^{2k}\approx \left(\sum_{x \in kA}\mu(x)\right)^2\leq |kA|E_k(A)$$ 
and so it suffices to show that $E_k(A)\ll_k |A|^k$, i.e. that there are at most $O_k\left(|A|^k\right)$ solutions to \eqref{eq:kenergy}. 

Say that a solution to \eqref{eq:kenergy} is \textbf{trivial} if at least $2k-1$ of the $2k$ terms occur with multiplicity at least $2$. By elementary counting there are at most $O_k(|A|^k)$ trivial solutions, so it suffices to show that there are no non-trivial solutions.

Suppose for a contradiction that a non-trivial solution to \eqref{eq:kenergy} exists. Gathering terms gives an expression of the form 
	\begin{equation}\label{eq:energygathered}
	n_1c_1+\ldots+n_tc_t =0
	\end{equation}    
where the $c_i $ are distinct elements of $A$ and, since we are in characteristic $p$, the $n_i$ are integers non-strictly between $1$ and $p-1$. The assumption of non-triviality implies that $t \geq 2$, since at least two of the terms have $n_i\in \left\{1,p-1\right\}$. Additionally, note that
	\begin{equation}\label{eq:coeffsgathered}
	n_1+\ldots+n_t\equiv 0\pmod p .
	\end{equation}    
Indeed after gathering terms on the left the different multiplicities $n_i$ must sum to zero, since there are the same number of terms on the left of \eqref{eq:kenergy} as on the right. Any $c_i$ for which $n_i \equiv 0\pmod p$ is discarded, meaning that the sum of the remaining multiplicities is $0\pmod p$ as well. 

Now since $A$ is separable and the $c_i$ are in $A$ we may relabel them if necessary to assume the existence of a ball $B(x,r)$ such that $c_1 \notin B$ but $c_2,\ldots,c_t \in B$. By \eqref{eq:energygathered},
	\begin{align*}
	|c_1-x|&=|n_1c_1-n_1x|\\
	&=|n_2c_2+\ldots+n_tc_t + n_1x|.
	\end{align*}
Then by \eqref{eq:coeffsgathered} and the non-archimedean property it follows that
	\begin{align*}
	|c_1-x|&=|n_2(c_2-x)+\ldots+n_t(c_t-x)|\\
	&\leq \max \left\{|c_2-x|,\ldots,|c_t-x|\right\}\\
	&\leq r
	\end{align*}
and hence $c_1 \in B(x,r)$ which is a contradiction. Thus there are no non-trivial solutions and the proof is complete.
\end{proof}

\section{Finding many separable sets}\label{section:solymosi}

The goal in this section is to show that if the sumset and product set of a set $A$ are both small then it must contain a large separable set. For this it adapts the argument of Solymosi \cite{solymosi3} for complex sum-products discussed in Section \ref{section:ff2}. Note that all of the analysis remains in the $\ffield$ setting; indeed some of the facts of non-archimedean geometry deployed here are manifestly false in $\mathbb{C}$.

A couple of new definitions are required. Define 
	\begin{align*}
	r_A(a)&=\min_{\substack{a'\in A\\ a'\neq a}}|a-a'|\\
	B_A(a)&=B(a,r_A(a)).\\
	\end{align*}
Additionally, say that $C \subseteq A$ is an \textbf{$A$-chain} if its elements can be indexed as $C=\left\{c_1,\ldots,c_n\right\}$ in such a way that
	$$B_A(c_1)\subseteq \ldots \subseteq B_A(c_n).$$
The following argument, a strengthened form of that found in \cite{solymosi3}, finds a large chain in $A$ as long as the sumset and product set are both small. In the event that this condition were to fail there would of course be nothing to prove.

\begin{lemma}\label{theorem:chains}
Any set $A$ contains an $A$-chain of cardinality \[\Omega\left(\frac{|A|^5}{|A+A|^2|AA|^2\log^3|A|}\right).\]
\end{lemma}

\begin{proof}
For each $a \in A$ write $N(a)$ for the maximal cardinality $N$ of an $A$-chain $C=\left\{c_1,\ldots,c_{N}\right\}$ for which $c_{N}=a$. Note for future reference that 
	$$N(a) \leq \left|B_A(a) \cap A\right|$$ 
since if $C$ is such an $A$-chain then $C \subseteq A$ by definition and for each $c \in C$ we have $c \in B_A(c)\subseteq B_A(a)$.

It suffices to find $a \in A$ such that
	$$N(a)\gg \frac{|A|^5}{|A+A|^2|AA|^2\log^3|A|}.$$
Begin with a dyadic pigeonholing. For each $0 \leq j \leq \log_2 |A|$ define $A_j$ to be the set of $a \in A$ for which $2^j \leq N(a) < 2^{j+1} $. The $A_j$ partition $A$ and so
	\[\sum_{j=0}^{\log_2 |A|}|A_j|=|A|.\]
Hence there exists $j$ for which $|A_j|\gg |A|/\log |A|$. We shall show that 
	$$2^j\gg \frac{|A|^5}{|A+A|^2|AA|^2\log^3|A|.}$$ 

To this end, say that a pair $(a,c)\in A \times A$ is \textbf{additively good} if 
	$$|(A+A)\cap (B_A(a)+c)|\leq \frac{2^{j+3}|A+A|}{|A_j|}$$	
and that $(a,d)\in A \times A$ is \textbf{multiplicatively good} if	
	$$|(AA)\cap (d\cdot B_A(a))| \leq \frac{2^{j+3}|AA|}{|A_j|}.$$	
Say that a quadruple $(a,b,c,d) \in A^4$ is \textbf{good} if
	\begin{enumerate}
	\item $a \in A_j$.
	\item $b \in B_A(a) \cap A$.
	\item $(a,c)$ is additively good.
	\item $(a,d)$ is multiplicatively good.
	\end{enumerate}

Write $Q$ for the number of good quadruples. We shall bound $Q$ from below to obtain
	\begin{equation}\label{eq:tom1}
	Q \gg 2^j |A_j||A|^2
	\end{equation}
and bound it from above to obtain
	\begin{equation}\label{eq:tom2}
	Q \ll \frac{2^{2j}|A+A|^2|AA|^2}{|A_j|^2}. 
	\end{equation}
Comparing (\ref{eq:tom1}) and (\ref{eq:tom2}) will then give the required bound on $2^j$ since $|A_j|\gg |A|/\log |A|$. Let's first establish (\ref{eq:tom1}). For fixed $c \in A$ we have
	\begin{align*}
	\sum_{a \in A_j}|(A+A)\cap (B_A(a)+c)|&=\sum_{a \in A_j}\sum_{u \in A+A}\mathds{1}\left( u \in B_A(a)+c\right)\\
	&= \sum_{v \in A+A-c} \sum_{a \in A_j}\mathds{1}\left(v \in B_A(a)\right)\\
	&= \sum_{v \in A+A-c}|C_j(v)|
	\end{align*}
where $C_j(v)$ is the set of $a \in A_j$ with $v \in B_A(a)$. 

Note that $C_j(v)$ is an $A$-chain. This follows from Lemma \ref{theorem:nonarchball} since for any two $a,b \in C_j(v)$ we have $v \in B_A(a) \cap B_A(b)$ and so either $B_A(a) \subseteq B_A(b)$ or $B_A(b) \subseteq B_A(a)$. 

Now since $C_j(v) \subseteq A_j$ and $C_j(v)$ is an $A$-chain, there is an $a \in A_j$ for which 
	$$\left|C_j(v)\right| \leq N(a) \leq 2^{j+1}.$$ 
We therefore have
	$$\sum_{a \in A_j}|(A+A)\cap (B_A(a)+c)|\leq 2^{j+1}|A+A|$$
and hence 
	$$|(A+A)\cap (B_A(a)+c)|\leq \frac{2^{j+3}|A+A|}{|A_j|}$$
holds for at least $3|A_j|/4$ elements $a \in A_j$. So for fixed $c \in A$ there are at least $3|A_j|/4$ elements $a \in A_j$ for which $(a,c)$ is additively good.

By the same argument we may show that for fixed $d \in A\backslash\{0\}$ there are at least $3|A_j|/4$ elements $a \in A_j$ for which $(a,d)$ is multiplicatively good.

Thus for any $c\in A$ and $d\in A\backslash\{0\}$ there are at least $|A_j|/2$ elements $a \in A_j$ for which $(a,c)$ is additively good and $(a,d)$ is multiplicatively good, i.e. for which conditions 3 and 4 hold. Furthermore for each such $a \in A_j$ there are at least $2^j$ elements $b \in A$ for which condition 2 holds, since 
	$$2^j \leq N(a) \leq \left|B_A(a)\cap A\right|.$$ 
In total therefore,
	$$Q \gg |A|^2|A_j|2^j$$
which concludes the proof of \eqref{eq:tom1}.

We now prove \eqref{eq:tom2}. Note that the map
 $$(a,b,c,d)\mapsto (a+c,b+c,ad,bd)$$
is injective and so it suffices to bound the number of possibilities for this latter expression, subject to the constraint that $(a,b,c,d)$ is good. There are certainly at most $|A+A|$ possibilities for $a+c$ and at most $|AA|$ for $ad$, so it suffices to show that if these are fixed then there are at most $O\left(2^j|A+A|/|A_j|\right)$ possibilities for $b+c$ and at most $O\left(2^j|AA|/|A_j|\right)$ for $bd$.

First establish the bound on the number of $b+c$. Note that if 
	$$a+c=a'+c'$$
then either 
	$$B_A(a)+c \subseteq B_A(a')+c'$$ 
or 
	$$B_A(a')+c' \subseteq B_A(a)+c$$ 
since both sets are balls with the same centre $a+c$. 

As a consequence, if $G\subseteq A \times A $ is the set of additively good pairs $(a,c)$, then for any $x \in A\overset{G}+A$ there is a fixed additively good pair $(a_x,c_x)$ such that 
	$$B_A(a)+c \subseteq B_A(a_x)+c_x$$
whenever $a+c=x$ and $(a,c)$ is additively good. Thus if $a+c=x$ is the fixed first co-ordinate and $b+c$ is a possible second co-ordinate then since $b \in B_A(a) \cap A$ and $c \in A$ we have
	\begin{align*}
	b+c &\in \left(A+A\right) \cap \left(B_A(a)+c\right)\\
	& \subseteq \left(A+A\right) \cap \left(B_A(a_x)+c_x\right).
	\end{align*}
Since $(a_x,c_x)$ is additively good, there are, as required, at most $O\left(2^j|A+A|/|A_j|\right)$ possibilities for $b+c$. The argument that there are at most at most $O\left(2^j|AA|/|A_j|\right)$ for $bd$ is similar. 

In total therefore
	$$Q \leq \frac{2^{2j+4}|A+A|^2|AA|^2}{|A_j|^2}$$ 
which concludes the proof of (\ref{eq:tom2}) and thus of the lemma.
\end{proof}

The following result shows that any chain contains a large separable subset, allowing Lemma \ref{theorem:separablegrowth} to be applied to the chain found in Lemma \ref{theorem:chains}.

\begin{lemma}\label{theorem:strict}
If $C$ is an $A$-chain then $C$ contains a separable set of cardinality at least $|C|/q$.
\end{lemma}

\begin{proof}
Observe that any subset $\left\{c_1,\ldots,c_n\right\} \subseteq C$ with
	\[B_A(c_1)\subsetneq \ldots \subsetneq B_A(c_n).\]
is separable. Indeed, such a set is separated by the balls $B_A(c_i)$ because if $c_{i+1}$ were an element of $B_A(c_i)$ it would follow that $r_A(c_{i+1})=r_A(c_i)$  and so by Lemma \ref{theorem:nonarchball} we would have the contradiction $B_A(c_i)=B_A(c_{i+1})$

Define an equivalence relation on elements of $A$ by $a \sim b$ if and only if $B_A(a)=B_A(b)$. To prove the lemma it suffices to show that each equivalence class contains at most $q$ elements of $A$. 

Note first that if $a \sim b$ then 
	$$|a-b|=r_A(a)=r_A(b).$$ 
Indeed, since $B_A(a)=B_A(b)$ it follows that $b \in B_A(a)$ and so $|a-b|\leq r_A(a)$. However by minimality, $|a-b|\geq r_A(a)$ and so $|a-b|=r_A(a)$. Similarly $|a-b|=r_A(b)$. 

Suppose for a contradiction that there is an equivalence class containing elements $a_1,\ldots,a_{q+1}$. Consider differences $a_1-a_i$ for $2 \leq i \leq q+1$. By the last paragraph we have
	$$|a_1-a_i|=r_A(a_1)=r_A(a_i).$$ 
Now look at the leading terms of the $a_1-a_i$. Since the leading term must be non-zero, there are only $q-1$ possibilities and so by the pigeonhole principle there must exist $i \neq j$ such that $a_1-a_i$ and $a_1-a_j$ have the same leading term. Since 		
	$$|a_1-a_i|=|a_1-a_j|=r_A(a_1)$$
it follows that $a_1-a_i$ and $a_1-a_j$ have the same degree, and that this is strictly greater than the degree of
	$$a_i-a_j=(a_1-a_j)-(a_1-a_i).$$
This yields the contradiction 
	$$r_A(a_i)=|a_i-a_j|<|a_1-a_i|=r_A(a_i)$$
and so concludes the proof.
\end{proof}

\section{Proof of Theorem \ref{theorem:functionfieldsumproduct}}\label{section:functionfieldproof}

Theorem \ref{theorem:functionfieldsumproduct} now follows by combining Lemma \ref{theorem:separablegrowth} from Section \ref{section:separable} with Lemma \ref{theorem:chains} and Lemma \ref{theorem:strict} from Section \ref{section:solymosi}.

\begin{proof}[Proof of Theorem \ref{theorem:functionfieldsumproduct}]
By Lemma \ref{theorem:chains}, the set $A$ contains an $A$-chain of cardinality
	$$\Omega\left(\frac{|A|^5}{|A+A|^{2}|AA|^{2}\log^{3}|A|}\right).$$
By Lemma \ref{theorem:strict} it therefore contains a separable set $S$ of cardinality
	$$\Omega\left(\frac{|A|^5}{q|A+A|^{2}|AA|^{2}\log^{3}|A|}\right)$$
and so Lemma \ref{theorem:separablegrowth} implies
\begin{align*}
|kA|\gg |kS| \gg \left(\frac{|A|^5}{q|A+A|^{2}|AA|^{2}\log^{3}|A|}\right)^k
\end{align*}
Pl\"unnecke's inequality (Lemma \ref{theorem:pl}) shows that $|kA|\ll \frac{|A+A|^k}{|A|^{k-1}}$ for any $k \in \mathbb{N}$ and so combining upper and lower bounds on $|kA|$ gives
	$$|A+A|^k \gg_k \left(\frac{|A|^5}{q|A+A|^{2}|AA|^{2}\log^{3}|A|}\right)^k |A|^{k-1}.$$
Taking $k$-th roots, we get
	$$|A+A| \gg_k \frac{|A|^{6-\frac{1}{k}}}{q|A+A|^{2}|AA|^{2}\log^{3}|A|}.$$
Letting $k$ tend to infinity and rearranging then yields
	$$|A+A|^3 |AA|^2 \gg_q |A|^{6-o(1)}$$
as required.

%Let 
%	$$N= \frac{|A+A|^2|AA|^2\log^3|A|}{|A|^{4}}$$ 
%and 
%	$$K= \frac{|A|^5}{|A+A|^{2}|AA|^{2}\log^{3}|A|}.$$
%By Corollary \ref{theorem:getseparable}, the set $A$ contains a subset $U$ which is a disjoint union of $\Theta_q(N)$ separable sets, each of cardinality $\Theta_q(K)$. For any natural number $k$, Lemma \ref{theorem:analyseseparable} implies that
%	\begin{align*}	
%	|kA|&\geq |kU|\\
%	&\gg_{q,k} N^{1/k}K^{k-1}\\
%	&= \frac{|A|^{5k-5-\frac{4}{k}}}{|A+A|^{2k-2-\frac{2}{k}}|AA|^{2k-2-\frac{2}{k}}\log^{3k-3-\frac{3}{k}} |A|}.
%	\end{align*}
%Pl\"unnecke's inequality (Lemma \ref{theorem:pl}) shows that $|kA|\ll \frac{|A+A|^k}{|A|^{k-1}}$ and so
%	$$|A+A|^{3k-2-\frac{2}{k}}|AA|^{2k-2-\frac{2}{k}}\gg_{q,k}\frac{|A|^{6k-6-\frac{4}{k}}}{\log^{3k-3-\frac{3}{k}} |A|}. $$
%Taking $k$-th roots of both sides and using the trivial inequalities $|A+A|,|AA|\geq |A|$ gives 
%	$$|A+A|^{3}|A+A|^2 \gg_{k,q} \frac{|A|^{6-\frac{2}{k}}}{\log^{3- \frac{3}{k}-\frac{3}{k^2}} |A|}.$$
%Theorem \ref{theorem:functionfieldsumproduct} follows by taking $k$ sufficiently large.
\end{proof}

\section{Further work}\label{section:functionfieldfurther}

\begin{itemize}
\item \textbf{Incidences and expanders.} Now that we have a sum-product estimate in $\ffield$, it is possible to obtain expander results like Theorem \ref{theorem:twovariableFp} and incidence results like Theorem \ref{theorem:incidences} and Theorem \ref{theorem:beck2} for function fields without much fuss. These will be stronger than the finite field case, but the gap versus the $\ffield$ sum-product estimate will be quite large.

Is it possible to do better by working directly with the non-archimedean geometry of $\ffield$? For example, can we obtain an incidence bound in $\ffield^2$ that is almost as strong as the Szmerer\'edi-Trotter theorem in $\mathbb{R}^2$? 
	
\item \textbf{Computer science applications.} At the start of the thesis, we said that we would not worry about applications. However it is worth mentioning that some existing applications of finite field growth results to theoretical computer science may be improved by considering $\ffield$ instead. 

\item \textbf{Other arithmetic combinatorics problems.} There are many more topics in arithmetic combinatorics than considered in this chapter, or indeed in this thesis. Thomas Bloom, with whom the work in this chapter is joint, is investigating a number of such problems in $\ffield$. His upcoming thesis is likely to be worth a read.
\end{itemize}

%% file: appendices.tex
\appendix
\addappheadtotoc

\chapter{Pigeonholing}
\label{chapter:pigeon}

This appendix summarises some standard pigeonholing results used throughout the thesis.

\section{Averaging}

Averaging results enable us to take information about the average behaviour of a set and deduce the existence of elements with particular behaviour. The most basic result of this kind shows that at least one element must be at least average, and at least one element must be at most average. It follows by elementary pigeonholing and so we record it without proof. 

\begin{lemma}\label{theorem:average1}Let $A$ be a finite set of real numbers. Then at least one element of $A$ must be greater than or equal to $\frac{1}{|A|}\sum_{a \in A}a$ and at least one must be less than or equal to $\frac{1}{|A|}\sum_{a \in A}a$.
\end{lemma}

Lemma \ref{theorem:average1} is so standard that it is used throughout the thesis without reference. More developed versions of this approach enable one to show that not just one element but many must exhibit behaviour not very much different from the average. These are constructed in the main body of the thesis to deal with particular situations.

\section{Dyadic pigeonholing}

The phrase `dyadic pigeonholing' refers to the following result, which at the price of a logarithmic factor allows us to assume that a variable is essentially constant.

\begin{lemma}[Dyadic pigeonholing]
Let $A$ be a finite set of real numbers strictly greater than one and less than or equal to $\alpha$. Then there exists an integer $k$ and a subset $A'$ of $A$ such that every element of $A'$ lies in the interval $(k,2k]$ and 
	$$|A'|k \gg \frac{\sum_{a \in A}a}{\log \alpha}.$$
\end{lemma}

\begin{proof}
For each integer $0 \leq j \leq \left\lceil \log_2 \alpha \right\rceil$, let 
	$$A_j = A \cap (2^j,2^{j+1}].$$
The sets $A_j$ partition $A$ and so 
	$$ \sum_{a \in A}a \approx \sum_{j=1}^{\left\lceil \log_2 \alpha \right\rceil}|A_j|2^j.$$
Hence there is a $j$ for which 
	$$|A_j|2^j \gg \frac{\sum_{a \in A}}{\log \alpha}.$$
Setting $k=2^j$ and $A'=A_{2^j}$, the proof is complete.
\end{proof}

\section{The Cauchy-Schwarz inequality}

The phrase `by Cauchy-Schwarz' typically has two meanings in the literature, and this is reflected in the thesis. It should be clear from the context which meaning is implied. The first meaning is the standard Cauchy-Schwarz inequality.

\begin{lemma}[Cauchy-Schwarz]
Let $A,B$ be finite sets of real numbers. Then
	$$\sum_{a \in A, b \in B}ab \leq \left(\sum_{a \in A}a^2\right)^{1/2}\left(\sum_{b \in B}b^2\right)^{1/2}. $$
\end{lemma}

The second meaning is a particular application of Cauchy-Schwarz to pairwise intersection of sets, as follows.

\begin{corollary}
Let $A$ be a finite set, and suppose we have a collection of subsets $A_i \subseteq A$, indexed by a finite set $I$. Then
	$$\sum_{i \in I}|A_i|\leq |A|^{1/2}\left(\sum_{i,j \in I}|A_i \cap A_j|\right)^{1/2}.$$
\end{corollary}

\begin{proof}
We have
	$$\sum_{i \in I}|A_i|= \sum_{a \in A}\left(\sum_{i \in I}\mathds{1}\left(a \in A_i\right)\right).$$
Hence by the Cauchy-Schwarz inequality,
	\begin{align*}
	\sum_{i \in I}|A_i|&\leq |A|^{1/2} \left(\sum_{a \in A} \sum_{i,j \in I} \mathds{1}\left(a \in A_i\right) \mathds{1}\left(a \in A_j\right)\right)^{1/2}\\
	&=|A|^{1/2}\left(\sum_{i,j \in I}|A_i \cap A_j|\right)^{1/2}.
	\end{align*}
\end{proof}

\chapter{Projective geometry}
\label{chapter:proj}

This appendix gives the background in projective geometry necessary for Chapters \ref{chapter:incidences} and \ref{chapter:expanders}. 

First it defines projective space $\mathbb{P}F^n$ over a field $F$. Then it shows how $\mathbb{P}F^n$ can be considered as the union of affine space $F^n$ and a hyperplane `at infinity'. Lastly, it defines projective transformations and establishes some useful facts about their transitivity. 

The approach is based on that in \cite{PS}. 

\section{Projective space}

Let $F$ be a field. Define \textbf{projective $n$-space} $\mathbb{P}F^n$ to be 
	$$\mathbb{P}F^n = \left(F^{n+1}\setminus\left\{\underline{0}\right\}\right)/\sim$$
where $\sim$ is the equivalence relation given by dilation, i.e if $x,y \in F^{n+1}\setminus\left\{\underline{0}\right\}$ then $x\sim y$ if and only if $\lambda x=y$ for some $\lambda \in F\setminus\left\{0\right\}$. Elements of $\mathbb{P}F^n$ are therefore equivalence classes, and we write $\left[x\right]$ for the equivalence class containing $x \in F^{n+1} \setminus \left\{\underline{0}\right\}$.

We will be concerned with linear subspaces of $\mathbb{P}F^n$. In \textit{affine} space $F^n$ an $(n-1)$-dimensional hyperplane is the locus of zeroes $x=\left(x_1,\ldots,x_n\right)$ of a linear, possibly inhomogeneous, equation in $n$ variables
	\begin{equation}\label{eq:projeq1}
	a_1 x_1 + \ldots + a_n x_n + a_{n+1} =0
	\end{equation}
where the $a_i$ are fixed elements of $F$.

But in \textit{projective} space, an $(n-1)$-dimensional hyperplane is  the set of $\left[x\right]\in \mathbb{P}F^n$ for which  $x\in F^{n+1}$ satisfies a linear \textit{homogeneous} equation in $n+1$ variables
	\begin{equation}\label{eq:projeq2}
	a_1 x_1 + \ldots + a_n x_n + a_{n+1}x_{n+1} =0.
	\end{equation}
The homogeneity of \eqref{eq:projeq2} ensures that this is well-defined. Note that $\left[x\right]$ and $\left[y\right]$ lie in the same $(n-1)$-dimensional projective hyperplane of $\mathbb{P}F^n$ if and only if $x$ and $y$ lie in the same $n$-dimensional affine hyperplane of $F^{n+1}$.

\section{The hyperplane at infinity}

It is often helpful to think of $\mathbb{P}F^n$ as the union of $F^n$ with an $(n-1)$-dimensional hyperplane `at infinity'. The idea is that two parallel $(n-1)$-spaces are disjoint in $F^n$, but in $\mathbb{P}F^n$ they intersect in an $(n-2)$-space on the hyperplane at infinity. Moreover, all $(n-1)$-spaces of the same gradient will intersect in the \textit{same} $(n-2)$ space at infinity.

For example, the projective line $\mathbb{P}F^1$ can be viewed as the extended line $F \cup \left\{\infty\right\}$. And the projective plane $\mathbb{P}F^2$ can be viewed as $F^2 \cup l_{\infty}$ where $l_{\infty}$ is the projective line at infinity. In this latter case, two parallel lines in $F^2$ intersect in a point on $l_{\infty}$, and all lines of the same gradient intersect at the same such point.

To justify this interpretation, view $F^{n+1}$ as $F^n \times F$ and identify $x \in F^n$ with $\left[x,1\right] \in \mathbb{P}F^{n}$. This preserves hyperplanes, since $x$ lies in the affine hyperplane given by \eqref{eq:projeq1} if and only if $\left[x,1\right]$ lies in the projective hyperplane given by \eqref{eq:projeq2}. This accounts for all elements of $\mathbb{P}F^n$ apart from those of the form  $\left[x,0\right]$ with $x \in F^n$. These form the projective $(n-1)$-space given by $x_{n+1}=0$, which we call the `hyperplane at infinity'. The verification of the claim that all $(n-1)$-spaces of the same gradient intersect in the same $(n-2)$-space is left as an exercise. 

\section{Projective transformations}

The group $PSL_{n+1}(F)$ of \textbf{projective transformations} of $\mathbb{P}F^n$ is defined by 
	$$PSL_{n+1}(F)=SL_{n+1}(F)/\pm I$$ where $I$ is the identity. Elements are therefore equivalence classes $[T]$ of linear transformations $T \in SL_{n+1}(F)$. The group has an action on $\mathbb{P}F^n$ given by 
	$$[T][x]=[T(x)].$$
It is easy to check that this action is well-defined, that elements of $PSL_{n+1}(F)$ are permutations of $\mathbb{P}F^{n}$, and that they preserve linear subspaces. 

The action also has an important transitivity property. Say that $(n+2)$ distinct points $p_i \in \mathbb{P}F^n$ are a \textbf{frame} if no $n+1$ of them lie in the same $(n-1)$-dimensional projective hyperplane. For example, three distinct elements of $\mathbb{P}F^1$ are a frame, and four distinct elements of $\mathbb{P}F^2$ are a frame if no three of them are collinear. The following result shows that the action of $PSL_n(F)$ is sharply transitive on frames.

\begin{lemma}[Sharp transitivity on frames]\label{theorem:frametransitive}
Let $(p_1,\ldots,p_{n+2})$ and $(q_1,\ldots,q_{n+2})$ be two frames of points in $\mathbb{P}F^n$. There is a unique projective transformation $\tau \in PSL_{n+1}(F)$ such that 
$$(q_1,\ldots,q_{n+2})=(\tau(p_1),\ldots,\tau(p_{n+2})).$$  
\end{lemma}

\begin{proof}
Let $e_i$ with $1 \leq i \leq n+1$ be elements of the canonical basis of $F^{n+1}$, let $f_i=[e_i]\in \mathbb{P}F^{n}$ and define
	$$e_*=\sum_{i=1}^{n+1}e_i$$
and $f_*=\left[e_*\right]$. 

It suffices to show that for any frame $(p_1,\ldots,p_{n+2})$ of points in $\mathbb{P}F^{n}$ there exists a unique $\tau \in PSL_{n+1}(F)$ that sends $p_i$ to $f_i$ for each $1 \leq i \leq n+1$ and sends $p_{n+2}$ to $f_*$. Indeed, if this is established then given frames $(p_1,\ldots,p_{n+2})$ and $(q_1,\ldots,q_{n+2})$ there are unique $\tau_1,\tau_2 \in PSL_{n+1}(F)$ that send both to $(f_1,\ldots,f_{n+1},f_*)$. Then $\mu = \tau_2^{-1}\tau_1$ is the unique map that sends $(p_1,\ldots,p_{n+2})$ to $(q_1,\ldots,q_{n+2})$.

We first prove the existence of an appropriate $\tau$. Say that $p_i \in \mathbb{P}F^n$ is given by $p_i=[t_i]$ with $t_i \in F^{n+1}$. Since the first $n+1$ points $p_i$ are not coplanar in $\mathbb{P}F^n$, the corresponding $t_i$ are not coplanar in $F^{n+1}$ and so form a basis. Hence we can write
	\begin{equation}\label{eq:frame}
	t_{n+2}=\sum_{i=1}^{n+1}\lambda_i t_i
	\end{equation}
with the $\lambda_i$ all elements of $F$. Let $T \in SL_{n+1}(F)$ be a linear transformation that sends the $F^{n+1}$-basis $\left(\lambda_1t_1,\ldots,\lambda_{n+1}t_{n+1}\right)$ to a scalar multiple of the canonical basis $(e_1,\ldots,e_{n+1})$. Let $\tau = [T]$. Then
	$$\tau (p_i) = [T][t_i]=[T(t_i)]=[e_i]=f_i$$
for each $1 \leq i \leq n+1$. Additionally \eqref{eq:frame} and the choice of $T$ imply that
	$$\tau (p_{n+2}) = [T(t_{n+2})]=\left[\sum_{i=1}^n\tau(\lambda_i t_i)\right]=f_*$$
and so we have established existence.

We now prove uniqueness, for which it suffices to show that our choice of $[T]\in PSL_{n+1}(F)$ is the only one that sends $p_i$ to $f_i$ for $1 \leq i \leq n+1$ and $p_{n+2}$ to $f_*$. So suppose that $[T]$ has this property. Then $[T][t_i]=[e_i]$  and so there exists $\mu_i \in F$ such that 
	$$T\left(\lambda_i t_i\right)=\mu_i e_i.$$ 
But since $[T][t_{n+2}]=[e_*]$ there exists $\mu_{*}$ such that 
	$$T\left(\sum_{i=1}^{n+1}\lambda_it_i\right)=\mu_* \sum_{i=1}^{n+1}e_i.$$
Combining gives
	$$\sum_{i=1}^{n+1}\mu_i e_i = \mu_* \sum_{i=1}^{n+1}e_i$$
and so by linear independence $\mu_i = \mu_*$ for all $i$. In other words, $T$ sends the basis $\left(\lambda_1t_1,\ldots,\lambda_{n+1}t_{n+1}\right)$ to a scalar multiple of the canonical basis $(e_1,\ldots,e_{n+1})$, as required.
\end{proof}